\documentclass{amsart}

\usepackage{amssymb, amsmath,mathtools,mathabx}
\usepackage{mathrsfs}
\usepackage{amscd}
\usepackage{verbatim}
\usepackage{stmaryrd}
\usepackage{bbm}
\usepackage[dvipsnames]{xcolor}

\usepackage{enumerate}

\usepackage[colorlinks,linkcolor={blue},citecolor={blue},urlcolor={purple},]{hyperref}

\usepackage{tabularx}
\newcolumntype{L}{>{\arraybackslash}X}
\usepackage{multirow}

\theoremstyle{plain}
\newtheorem{theorem}{Theorem}[section]
\theoremstyle{remark}
\newtheorem{remark}[theorem]{Remark}

\theoremstyle{plain}
\newtheorem{corollary}[theorem]{Corollary}
\newtheorem{lemma}[theorem]{Lemma}
\newtheorem{proposition}[theorem]{Proposition}
\newtheorem{definition}[theorem]{Definition}

\numberwithin{equation}{section}

\def\N{{\mathbb N}}
\def\Z{{\mathbb Z}}
\def\Q{{\mathbb Q}}
\def\R{{\mathbb R}}

\def\K{{\mathcal K}\,}

\renewcommand{\P}{\mathbf{P}}
\newcommand{\E}{\mathbf{E}}
\newcommand{\F}{\mathscr{F}}

\newcommand{\p}{\mathbb{P}}
\newcommand{\q}{\mathbb{Q}}
\newcommand{\T}{\mathbb{T}}
\newcommand{\Tor}{\mathbb{T}}

\newcommand{\calL}{\mathscr{L}}

\newcommand{\ellip}{\mu}

\newcommand{\D}{\mathcal{D}}

\newcommand{\Borel}{\mathscr{B}}
\newcommand{\Progress}{\mathscr{P}}
\newcommand{\B}{\mathcal{B}}

\renewcommand{\emptyset}{\varnothing}

\newcommand{\g}{\gamma}

\newcommand{\om}{\omega}
\renewcommand{\O}{\Omega}

\newcommand{\A}{\mathscr{A}}

\newcommand{\wt}{\widetilde}

\newcommand{\supp}{\mathrm{supp}}
\newcommand{\one}{\mathbf{1}}
\newcommand{\dd}{\mathrm{d}}

\newcommand{\X}{\mathcal{X}}
\newcommand{\Y}{\mathcal{Y}}

\newcommand{\vor}{\zeta}
\newcommand{\vorc}{\zeta_{{\rm cut}}}
\newcommand{\uc}{u_{{\rm cut}}}
\newcommand{\vorcn}{\zeta_{{\rm cut},n}}
\newcommand{\ucn}{u_{{\rm cut},n}}
\newcommand{\vord}{\zeta_{{\rm det}}}
\newcommand{\vordd}{\zeta_{{\rm cut, d}}}
\newcommand{\vorcc}{\xi_{{\rm cut}}}
\newcommand{\Kin}{\mathcal{E}}

\newcommand{\embed}{\hookrightarrow}

\newcommand{\ud}{u_{\mathrm{det}}}

\newcommand{\udd}{u_{\mathrm{cut, d}}}
\newcommand{\pdd}{p_{\mathrm{cut, d}}}
\newcommand{\rhod}{\varrho_{\mathrm{det}}}
\newcommand{\rhodn}{\varrho_{\mathrm{det},n}}
\renewcommand{\i}{\mathrm{i}}
\newcommand{\Bs}{\mathbb{B}}
\newcommand{\Ls}{\mathbb{L}}
\newcommand{\Hs}{\mathbb{H}}
\newcommand{\Set}{\mathcal{S}_{\ell^2}^0}
\newcommand{\qxi}{Q_{\xi}}
\newcommand{\Stok}{\mathcal{S}}

\newcommand{\mf}{{\rm mz}}
\renewcommand{\L}{{\rm L}}
\newcommand{\Sm}{{\rm S}}

\allowdisplaybreaks

\begin{document}

\thanks{The author has received funding from the VICI subsidy VI.C.212.027 of the Netherlands Organisation for Scientific Research (NWO). The author is a member of GNAMPA (IN$\delta$AM)}

\author{Antonio Agresti}
\address{Delft Institute of Applied Mathematics\\
	Delft University of Technology \\ P.O. Box 5031\\ 2600 GA Delft\\The
	Netherlands.}
\curraddr{Department of Mathematics Guido Castelnuovo\\
	Sapienza University of Rome \\ P.le Aldo Moro 5\\ 00185 Rome\\ Italy.}
\email{antonio.agresti92@gmail.com}

\date\today

\title[On anomalous dissipation induced by transport noise]{On anomalous dissipation induced\\ by transport noise}



\begin{abstract}
In this paper, we show that suitable transport noises produce anomalous dissipation of both enstrophy of solutions to 2D Navier-Stokes equations and of energy of solutions to diﬀusion equations in all dimensions.
The key ingredients are Meyers' type estimates for SPDEs with transport noise, which are combined with recent scaling limits for such SPDEs. The former enables us to establish, for the first time, uniform-in-time convergence in a space of positive smoothness for such scaling limits. Compared to previous work, one of the main novelties is that anomalous dissipation might take place even in the presence of a transport noise of arbitrarily small intensity. Physical interpretations of our results are also discussed.
\end{abstract}

\maketitle

\section{Introduction and statement of the main results}
The primary goal of this manuscript is to investigate the effect of transport noise on the \emph{anomalous dissipation of enstrophy} for the 2D Navier-Stokes equations (NSEs in the following) in vorticity formulation:
\begin{equation}
\label{eq:NS_2d_vorticity}
\left\{
\begin{aligned}
\partial_t \vor^{\nu} &+(u^\nu \cdot\nabla) \vor^{\nu} 
=\nu \Delta \vor^{\nu} + \sqrt{2\mu }\sum_{k\in \Z^2_0} \theta_k^{\nu}	\,
(\sigma_k \cdot\nabla) \vor^{\nu} \circ \dot{W}_t^k & \text{on }&\T^2,\\
\vor^{\nu}(0,\cdot)&=\vor_0& \text{on }&\T^2,
\end{aligned}
\right.
\end{equation}
where $u^\nu
=\K \vor^{\nu}\stackrel{{\rm def}}{=}(-\partial_y,\partial_x) \Delta^{-1}\vor^{\nu}
$, $\K$ denotes the Biot-Savart operator, $\T^2=(\R/\Z)^2$ the two-dimensional torus, $\nu>0 $ the kinematic viscosity of the fluid, $(W^k)_{k\in \Z^2_0}$ a family of  complex Brownian motions, $\circ$ the Stratonovich product, $\theta^{\nu}=(\theta_k^{\nu})_{k\in \Z^2_0}\in \ell^2$ with $\Z^2_0=\Z^2\setminus\{0\}$, $(\sigma_k)_{k\in \Z^2_0}$ a family of divergence-free vector field described in Subsection \ref{ss:structure_noise}, and $\mu>0$ the noise's intensity (cf.\ \eqref{eq:small_noise_transport} below).

Anomalous energy dissipation in fluid flows has been proven experimentally to a large degree \cite{SA97} and stands at the basis of turbulence theory \cite{Frisch95,K41_1,K41_2,K41_3}. For this reason, it is sometimes referred to as the zeroth law of turbulence. Roughly speaking, anomalous dissipation states that at high Reynolds numbers (i.e., $\nu\downarrow 0$) the averages of the energy dissipation are uniformly bounded from below in $\nu$. 
In the context of the Kraichnan-Batchelor theory (KB in the following) of 2D turbulence \cite{B69_2D_spectrum,R67_2D_turbulence} (see \cite[Section 9.7]{Frisch95} for additional references), the relevant anomalously dissipated quantity is the \emph{enstrophy}, and the anomalous dissipation of enstrophy reads as:
\begin{equation}
\label{eq:anomalous_dissipation_naive}
\liminf_{\nu\downarrow 0}\big\langle\,\nu\, |\nabla \vor^{\nu}|^2\big\rangle>0.
\end{equation}
In the above, the (unspecified) operator $\langle\cdot \rangle$ typically represents an ensemble average, e.g., space-time average or an expected value of it in the case of random environments. The physical mechanism behind the anomalous dissipation of enstrophy is the transfer, as $\nu\downarrow 0$, of the enstrophy from large to small spatial scales by the nonlinear convective term. This transference produces high gradients $|\nabla \vor^{\nu}|^2$ which cannot be compensated by the small multiplicative factor $\nu$ and therefore yield a non-trivial limit energy dissipation rate as $\nu\downarrow 0$. Let us mention that, sometimes in the literature, the $\liminf_{\nu\downarrow 0}$ in \eqref{eq:anomalous_dissipation_naive} is replaced by a $\limsup_{\nu\downarrow 0}$, resulting in a weaker notion that only guarantees anomalous dissipation of solutions to the 2D NSEs \eqref{eq:NS_2d_vorticity} along some subsequence of viscosities $\nu=\nu^k\to 0$.

\smallskip

In the deterministic setting (i.e., $\theta^\nu\equiv 0$), there are serious obstructions in making the above pictures rigorous. Indeed, as explained in \cite[Section 1]{LFMNL06}, in the case of $\vor_0\in L^2(\T^2)$, the anomalous dissipation of enstrophy \eqref{eq:anomalous_dissipation_naive} implies the existence of enstrophy-dissipative (suitable) weak solutions to the deterministic incompressible 2D Euler equations obtained via vanishing viscosity limits. However, the latter is known not to be true, see e.g., \cite{CS15_renormalized,MR4260789} or \cite[Remark 2]{CDE22_yudovich}. Moreover, the proposed ways to fix such a paradox in \cite{E01_dissipation} appeared not to be resolutive \cite[Section 7]{LFMNL06}.

\smallskip

The anomalous dissipation result for the enstrophy for the 2D NSEs \eqref{eq:NS_2d_vorticity} in Theorem \ref{t:anomalous_diss_NS} below appears to be the \emph{first} result in the direction of the KB theory in the physically relevant case of $L^2(\T^2)$ initial enstrophy, see also the comments below \eqref{eq:energy_balance_vor}.
However, here we are not aiming at capturing the sophisticated mechanics behind anomalous dissipation as described below \eqref{eq:anomalous_dissipation_naive}. 
Instead, we prove that transport noises, localised at high frequencies and varying with the viscosity, lead to anomalous dissipation.   
More precisely, anomalous dissipation is created by the `turbulent flows' $\sum_{k\in \Z^2_0}\theta^{\nu}_k\sigma_k\,\dot{W}^{k}$ rather than the convective nonlinearity as expected in the KB theory. 
Moreover, this is also true for noises of `small intensities' and for `nonlinear passive scalar' such as the 2D NSEs in vorticity formulation (here, we mean the turbulent flows differ from the advected quantity). 
Surprisingly, stochastic perturbations of 2D NSEs provide \emph{new} dissipation mechanics compared to the deterministic case.
The reader is also referred to Section \ref{s:2D_NSE_velocity}, where we also discuss the velocity formulation of the NSEs and the three-dimensional case. Further details about the physical interpretations of our results are given in Subsection
\ref{ss:heuristics_viscosity_dependent}.

\smallskip

The arguments in the current manuscript also extend the case of \emph{anomalous dissipation of energy} for passive scalars in all dimensions $d\geq 1$:
\begin{equation}
\label{eq:diffusive_scalars}
\left\{
\begin{aligned}
\partial_t \varrho^{\g}&=\g \Delta\varrho^{\g} 
+ \sqrt{c_d\mu }\sum_{k\in \Z^d_0}\sum_{1\leq \alpha\leq d-1} \theta_{k} ^{\g}\,(\sigma_{k,\alpha} \cdot\nabla) \varrho^{\g} \circ \dot{W}_t^{k,\alpha} & \text{on }&\T^d,\\
\varrho^\g(0,\cdot)&=\varrho_0& \text{on }&\T^d.
\end{aligned}
\right.
\end{equation} 
Here, $\T^d=(\R/\Z)^d$ is the $d$-dimensional torus, $\g>0$ the diffusivity of the scalar, $\theta^{\nu}=(\theta_k^{\g})_{k\in \Z^d_0}\in \ell^2$, $c_d=\frac{d}{d-1}$, while $(W^{k,\alpha})_{k,\alpha}$ and $(\sigma_{k,\alpha})_{k,\alpha}$ are families of complex Brownian motions and divergence-free vector fields described in Subsection \ref{ss:structure_noise}.

Anomalous dissipation for passive scalars advected by turbulent flows has attracted a lot of interest in recent years. For passive scalars, anomalous dissipation by turbulent flows is at the basis of the corresponding theory of scalar turbulence \cite{DSY05,SS00_scalar_turbulence,S96_spectrum_turbulence}. 
It can be defined by replacing in  \eqref{eq:anomalous_dissipation_naive}  the viscosity $\nu$ and the vorticity $\vor^{\nu}$ by the diffusivity $\g>0$ and the intensity of the passive scalar $\varrho^{\g}$ (solving an advection-diffusion such as \eqref{eq:diffusive_scalars}), respectively. 
In the deterministic setting, many results have been established in the context of anomalous dissipation. It is not possible to give here a complete overview of the deterministic results, and the reader is referred to, e.g., \cite{AV23_hom,BDL23_anomalous,BSW23,CCS23,CTV14,EL23_anomalous,hess2025universal} and the references therein. 

\smallskip

Before discussing works on anomalous dissipation and related topics in the context of stochastic fluid dynamics, let us first discuss the physical relevance of transport noise in stochastic fluid dynamics. 
Nowadays, there are several derivations of NSEs with transport noise available in the literature, see e.g., \cite{DP22_two_scale,FP20_small,FlaPa21,H15_SVP,M14_derivation,MR04}, and it is by now a well-established model in stochastic fluid dynamics \cite{BCF91,BCF92,BP00_strong,CP01,FlGa,FL32_book,HLN19,HLN21_annals}. To some extent, at the basis of such derivations is the idea of the separation of scales. A heuristic derivation using the latter principle is given in Subsection \ref{ss:heuristics_viscosity_dependent} below. This allows us to stress the physical relevance of the $\nu$-dependence of transport noise in \eqref{eq:NS_2d_vorticity} and provide an interpretation of our results for 2D NSEs. Starting from the works \cite{FL19,G20_convergence}, the effect of transport noise on mixing and enhanced dissipation of NSEs or advection-diffusion equations is by now well-understood, see e.g.,  \cite{FGL21_quantitative,FGL22_eddy,L23_enhanced}
(and also \cite{A23,BFL24_magnetic,FGL21,FHLN22,FD23,FLL24,L23_regularization,LTZ24} for related works).
The same is not true for anomalous dissipation. 
To the best of our knowledge, the only results on the anomalous dissipation for SPDEs with transport noise are given in \cite{HPZZ23,R23_kraichnan}. A comparison is postponed to Subsection \ref{ss:comparison}, where a more detailed discussion on our contribution is possible.
Finally, from a technical point of view, our proofs rely on a \emph{refinement} of the scaling limit arguments \cite{FL19,G20_convergence} in the parabolic setting, which is interesting on its own. A detailed discussion is given in Subsection \ref{sss:novelty} below.

\subsection{Anomalous dissipation results}
\label{ss:2D_NS_anomalous}
In this subsection, we state the main result of the manuscript. 
Below, for convenience, we employ the following notation 
\begin{equation}
\label{eq:symmetric_coefficients_set}
\Set\stackrel{{\rm def}}{=}
\Big\{\theta=(\theta_k)_{k\in \Z^d_0}\in \ell^2\,:\, 
\|\theta\|_{\ell^2}=1 \,\text{ and }\,  \#\{k\,:\, \theta_k \neq 0\}<\infty\Big\}
\end{equation}
for the set of normalized $\ell^2$-vectors with finitely non-zero components. Here, for simplicity, we did not display the dependence on the dimension $d\geq 1$. 

\begin{theorem}[Anomalous dissipation of enstrophy by transport noise -- 2D NSEs]
\label{t:anomalous_diss_NS}
Let $N\geq 1$ and $\delta>0$ be fixed. Then, for all $\mu>0$, there exists a family $(\theta^{\nu})_{\nu\in (0,1)}\subseteq \Set$ such that, for all mean-zero $\vor_0\in H^{\delta}(\T^2)$ satisfying $ \|\vor_0\|_{H^{\delta}(\T^2)}\leq N$, 
we have
$$
\liminf_{\nu\downarrow 0} \E\int_0^1 \int_{\T^2}  \nu \,|\nabla \vor^{\nu}|^2\,\dd x \,\dd t \geq \frac{1}{2} (1-e^{-\mu/(4\pi^2)})\, \|\vor_0\|_{L^2(\T^2)}^2,
$$
where $\vor^\nu$ is the unique global smooth solution to \eqref{eq:NS_2d_vorticity}.
\end{theorem}

Solutions to \eqref{eq:NS_2d_vorticity} are defined in Definition \ref{def:sol_NSE}. The existence of global-in-time and smooth-in-space solutions to \eqref{eq:NS_2d_vorticity} with $\theta^{\nu} \in \Set$ follows from the arguments presented in \cite[Theorems 2.7 and 2.12]{AV21_NS}. Under additional assumption on $\vor_0$, Theorem \ref{t:anomalous_diss_NS} also holds with $\inf$ in place of $\liminf$; see Proposition \ref{prop:anomalous_diss_NS2}.

As commented in \eqref{eq:energy_small_scales_two_sides} below, as $\theta\in \Set$, the `energy' of the transport noise $(\sqrt{2\mu} \,\theta_k \sigma_{k})_{k\in\Z^2_0}$ is related to the intensity parameter $\mu$: 
\begin{equation}
\label{eq:small_noise_transport}
\sqrt{2\mu}\sup_{\nu\in(0,1)}\|(\theta^{\nu}_k\sigma_k)_{k\in \Z^2_0}\|_{L^{\infty}(\T^2;\ell^2)}\leq \sqrt{2\mu}.
\end{equation}
Theorem \ref{t:anomalous_diss_NS} shows that, with an appropriate choice of the noise coefficient $\theta^\nu$, anomalous dissipation for the 2D NSEs \eqref{eq:NS_2d_vorticity} can occur even in the presence of \emph{small noise} (i.e., for any $\mu > 0$). This stands in contrast to the existing literature, which typically requires noise of sufficiently large intensity to observe anomalous dissipation (see Subsection \ref{ss:comparison}). The main novelty of this work lies in showing that anomalous dissipation can arise even when the transport noise has \emph{arbitrarily small intensity}.
Finally, we note that Theorem \ref{t:anomalous_diss_NS} also allows one to \emph{prescribe} the fraction $\eta \in (0,1)$ of the initial energy $\|\vor_0\|_{L^2}^2$ that is anomalously dissipated by the 2D NSEs \eqref{eq:NS_2d_vorticity}. Specifically, by choosing $\mu = -\ln(1 - \eta)$, Theorem \ref{t:anomalous_diss_NS} guarantees that the energy dissipation $\liminf_{\nu\downarrow 0} \E\int_0^1 \int_{\T^2}  \nu \,|\nabla \vor^{\nu}|^2\,\dd x \,\dd t$ is at least $\frac{\eta}{2}\|\vor_0\|_{L^2(\T^2)}^2$.

To connect Theorem \ref{t:anomalous_diss_NS} with the discussion around \eqref{eq:anomalous_dissipation_naive} on the KB theory, let us discuss its implications. Firstly, as $\nabla \cdot \sigma_{k,\alpha}=0$, solutions to \eqref{eq:NS_2d_vorticity} satisfy:
\begin{equation}
\label{eq:energy_balance_vor}
\frac{1}{2}\|\vor^\nu(t,\cdot)\|_{L^2(\T^2)}^2
+ \int_{0}^t \int_{\T^2}\nu\, |\nabla \vor^\nu|^2\,\dd x\,\dd s=
\frac{1}{2}\|\vor_0\|_{L^2(\T^2)}^2\  \text{ a.s.\ for all }t\geq 0.
\end{equation}
If $\theta^\nu$ is chosen as in Theorem \ref{t:anomalous_diss_NS}, then the 2D NSEs \eqref{eq:NS_2d_vorticity} exhibit \emph{anomalous dissipation of enstrophy} at time $t=1$, i.e., 
$
\liminf_{\nu\downarrow 0} \E\int_0^1 \int_{\T^2}  \nu \,|\nabla \vor^{\nu}|^2\,\dd x \,\dd t >0,
$
or equivalently,
$$
\limsup_{\nu\downarrow 0}\,\E\|\vor^\nu(1,\cdot)\|^2_{L^2(\T^2)}<\|\vor_0\|_{L^2(\T^2)}^2.
$$
Thus, with the aid of transport noise, we obtain anomalous dissipation of enstrophy while the energy balance \eqref{eq:energy_balance_vor} remains true at fixed $\nu>0$ and with $\vor_0\in L^2(\T^2)$. As mentioned above, this cannot be obtained in the absence of noise, see \cite{LFMNL06}. 

\smallskip

An inspection of the proof of Theorem \ref{t:anomalous_diss_NS} shows that the time $t=1$ can be replaced by any time $t>0$; however, the corresponding choice of $(\theta^{\nu})_{\nu \in (0,1)}$ depends on such time $t$.  
In addition, we can take $\theta^{\nu}$ to be constants for $\nu\in [\nu_{j+1},\nu_{j})$ where $(\nu_j)_{j\geq 1}\subseteq (0,1]$ is a sequence satisfying $\nu_1=1$, $\nu_{j+1}\leq \nu_j$  and $\lim_{j\to \infty} \nu_j=0$. Finally, the proof of Theorem \ref{t:anomalous_diss_NS} shows that 
$\theta^{\nu}$'s are localized at high frequencies:
\begin{equation}
\label{eq:supp_theta_high_frequencies}
\supp\,\theta^{\nu}\subseteq \big\{k\in \Z_0^2\,:\,N^{\nu}\leq |k|\leq 2N^{\nu}\big\}\quad \text{ and }\quad 
\liminf_{\nu\downarrow 0}N^{\nu}= \infty.
\end{equation}
For passive scalars, a version of Theorem \ref{t:anomalous_diss_NS} holds in all dimensions.

\begin{theorem}[Anomalous dissipation of energy by transport noise -- Passive scalars]
\label{t:anomalous_diss_diff}
Let $d\in \N_{\geq 1}$, $N\geq 1$ and $\delta>0$ be fixed. Then,
for all $\mu>0$, there exists a family $(\theta^{\g})_{\g\in (0,1)}\subseteq \Set$ such that, for all mean-zero $\varrho_0\in H^{\delta}(\T^d)$ satisfying $ \|\varrho_0\|_{H^{\delta}}\leq N$, 
we have
$$
\liminf_{\g\downarrow 0} \E\int_0^1 \int_{\T^d}  \g \,|\nabla \varrho^{\g}|^2\,\dd x \,\dd t \geq \frac{1}{2} (1-e^{-\mu/(4\pi^2)})\,\|\varrho_0\|_{L^2(\T^d)}^2,
$$
where $\varrho^\g$ is the unique global smooth solution to \eqref{eq:diffusive_scalars}.
\end{theorem}

Global solutions to \eqref{eq:diffusive_scalars} are defined in Definition \ref{def:sol_passive}, and their existence as well as smoothness for $\theta^{\g}\in \Set$ is well-known, e.g., \cite[Theorem 4.2]{AV23_reactionI}. Under an additional assumption on $\varrho_0$, the $\liminf_{\g\downarrow 0}$ can replaced by $\inf_{\g\in (0,1)}$ in Theorem \ref{t:anomalous_diss_diff}, see Proposition \ref{prop:anomalous_diss_diff2}.

The comments below Theorem \ref{t:anomalous_diss_NS} extend to the above result. In particular, arguing in the comments below \eqref{eq:small_noise_transport}, Theorem \ref{t:anomalous_diss_diff} yields anomalous dissipation for passive scalars even in the presence of a \emph{small} transport noise. 

\begin{remark}
\label{r:anomalous_dissipation_and_convergence}
Here we collect some additional comments on the anomalous dissipation of passive scalars as in Theorem \ref{t:anomalous_diss_diff}. The same consideration extends verbatim to the 2D NSEs case analyzed in Theorem \ref{t:anomalous_diss_NS}. 
\begin{itemize}
\item {\rm (Limiting behavior of $(\varrho^\g)_{\g\in (0,1)}$)}.\ 
From the proof of Theorem \ref{t:anomalous_diss_diff}, it follows that the unique global solution $\varrho^\g$ to \eqref{eq:diffusive_scalars}  with $\theta^\g$ as in the latter satisfies
\begin{equation}
\label{eq:varrhog_convergence_to_varrhodet}
\varrho^\g \to \rhod^{0} \  \text{as $\g\downarrow 0$ in probability in }C([0,1];L^2(\T^d)),
\end{equation}
where $\rhod^0$ is the unique global solution to 
\begin{equation}
\label{eq:rhod_PDE_introduction_rev2}
\partial_t \rhod^{0} 
= \mu\Delta \rhod^{0}  \ \text{ on }\T^d,\qquad 
\rhod^{0}(0,\cdot)=\varrho_0\ \text{ on }\T^d.
\end{equation}
This fact is crucial in our proof of Theorem \ref{t:anomalous_diss_diff}, see Subsection \ref{sss:strategy} for more details. Furthermore, the appearance of the additional (eddy) dissipation $\mu\Delta \rhod^{0}$ in \eqref{eq:rhod_PDE_introduction_rev2} underlies the anomalous dissipation of energy for passive scalars. 
It is worth noticing that \eqref{eq:varrhog_convergence_to_varrhodet} implies that the extra dissipation in \eqref{eq:rhod_PDE_introduction_rev2} is produced by the gradients of $\varrho^\g$. More precisely, \eqref{eq:varrhog_convergence_to_varrhodet} and the energy balances for $\varrho^\g$ and $\rhod^0$ yield
\begin{align*}
2\int_0^\cdot\int_{\T^d} \g\, |\nabla \varrho^\g|^2\,\dd x \,\dd s
 =\|\varrho_0\|_{L^2(\T^d)}^2-\|\varrho^\g(\cdot)\|_{L^2(\T^d)}^2&\\
\stackrel{\g \downarrow 0}{\to}
\|\varrho_0\|_{L^2(\T^d)}^2 -\|\rhod^0(\cdot)\|_{L^2(\T^d)}^2& = 
2\int_0^\cdot\int_{\T^d} \mu\, |\nabla \rhod^0|^2\,\dd x \,\dd s
\end{align*} 
in probability in $C([0,1])$. The reader is referred to Remark \ref{r:lack_convergence_gradient} for an analogous result in the diffusive case $\g>0$.

Let us emphasize that the convergence in \eqref{eq:varrhog_convergence_to_varrhodet} is non-trivial due to the vanishing diffusive limit $\g\downarrow 0$. While compactness and the energy balance yield $\varrho^\g \to \rhod^0$ in $C([0,1]; H^{-\delta}(\T^d))$ for all $\delta > 0$, establishing the stronger convergence in \eqref{eq:varrhog_convergence_to_varrhodet} requires a more delicate argument. In particular, \eqref{eq:varrhog_convergence_to_varrhodet} relies on a careful choice of the coefficients $\theta^\g$, which is made possible by a scaling limit argument and the application of Meyers' estimates (proved in Appendix \ref{app:Meyers}). Further details are provided in Subsection \ref{sss:strategy}.

\item\label{it:anomalous_dissipation_and_convergence2} {\rm (Non necessity of $L^2$-convergence for anomalous dissipation)}.\
Although we will prove Theorem \ref{t:anomalous_diss_diff} using the convergence stated in \eqref{eq:varrhog_convergence_to_varrhodet}, we emphasize that this convergence is not required for anomalous dissipation to occur. In fact, it suffices that $\varrho^\g \approx \rhod^\g$ with high probability for all $\g > 0$. This observation leads to a slightly different version of the anomalous dissipation result found in Theorems \ref{t:anomalous_diss_NS} and \ref{t:anomalous_diss_diff}, which we will present for passive scalars in Proposition \ref{prop:anomalous_diss_diff2}, and for the 2D NSEs in Proposition \ref{prop:anomalous_diss_NS2}.
\end{itemize}
\end{remark}

Next, we provide a heuristic motivation for the $\nu$-dependence of the transport noise in the context of NSEs with transport noise \eqref{eq:NS_2d_vorticity}. We also offer a possible physical interpretation of our results.
The physical motivations for the $\gamma$-dependence of the transport noise in advection-diffusion equations \eqref{eq:diffusive_scalars} are not discussed here; we refer the reader to \cite[Section 2]{DE17_dissipation} and \cite[Section 2]{R23_kraichnan} for further discussion.

\subsection{Heuristics for $\nu$-dependent transport noise}
\label{ss:heuristics_viscosity_dependent}
Here, to motivate transport noise, we follow the heuristic motivations given in \cite[Section 1.2]{FL19} based on two scale arguments.
Rigorous justifications of the transport noise in fluid dynamics models are given in, e.g., \cite{DP22_two_scale,FlaPa21,H15_SVP,MR01,MR04}.

As in \cite[Subsection 1.2]{FL19}, let us assume that the vorticity field $\vor^{\nu}$ of a turbulent fluid decomposes into large and small scales, i.e.,  $\vor^{\nu}=\vor^{\nu}_\L+\vor_{\Sm}^{\nu}$ where $\vor_{\Sm}^{\nu}$ and $\vor_\L^{\nu}$ are the `small' and `large' scales, respectively; and 
\begin{equation}
\label{eq:large_small_scale_decomposition}
\left\{
\begin{aligned}
\partial_t \vor^{\nu}_{\L} +([u_{\Sm}^{\nu}+u_\L^{\nu}]\cdot\nabla)\vor_\L^{\nu}
&= \nu \Delta \vor^{\nu}_\L& \text{on }&\T^2,\\
\partial_t \vor_{\Sm}^{\nu} +([u_{\Sm}^{\nu}+u_\L^{\nu}]\cdot\nabla)\vor_{\Sm}^{\nu}
&=\nu \Delta \vor_{\Sm}^{\nu} & \text{on }&\T^2.
\end{aligned}
\right.
\end{equation}
Here, $u_\L^{\nu}=\K \vor_\L^\nu$ and $u_{\Sm}^{\nu}=\K \vor_{\Sm}^\nu$ are the corresponding small and large scales velocity, respectively.
Note that $\vor^{\nu}$ solves the 2D (deterministic) NSEs. 
For completeness, below we also include comments on the 3D NSEs, where in \eqref{eq:large_small_scale_decomposition} the transport term $([u^{\nu}_{\Sm}+ u^{\nu}_{\L}] \cdot\nabla)\vor_*^\nu$
 is replaced by $([u^{\nu}_{\Sm}+u^{\nu}_{\L}] \cdot\nabla)\vor_*^\nu-(\vor^{\nu}_* \cdot\nabla)[u^{\nu}_{\Sm}+u^{\nu}_{\L}] $ with $*\in\{\L,\Sm\}$, cf.\ again \cite[Subsection 1.2]{FL19}.
Heuristically, one could think that, in a turbulent regime, $u^{\nu}_{\Sm}$ varies in time very rapidly compared to $u^{\nu}_\L$. Therefore, $u^{\nu}_{\Sm}$ can be approximated, in time, by a white noise:
\begin{equation}
\label{eq:small_noise}
u^{\nu}_{\Sm}(t,x)\approx \sum_{k\in \Z^d_0}\sum_{1\leq \alpha\leq d-1}\wt{\theta}_k^{\nu}\,e^{2\pi \i k\cdot x} a_{k,\alpha} \dot{W}_t^{k,\alpha},
\end{equation}
where $(a_{k,\alpha})_{\alpha\in \{0,\dots,d-1\}}\subseteq \R^d$ is an orthonormal basis of $k^{\perp}=\{k'\in \R^d\,:\, k'\cdot k=0\}$ (this fact ensures $\nabla\cdot\sigma_{k,\alpha}=0$). 
Note that using \eqref{eq:small_noise} in the first of \eqref{eq:large_small_scale_decomposition}, one obtains \eqref{eq:NS_2d_vorticity}. The choice of the Stratonovich formulation in \eqref{eq:NS_2d_vorticity} is due to Wong-Zakai-type results, which, roughly, ensure that \eqref{eq:NS_2d_vorticity} can be obtained as a limit of regular approximations of $\dot{W}^{k,\alpha}_t$.
At this point, there are no physical motivations for the independence of the coefficients $\wt{\theta}^{\nu}_k$ on $\nu>0$ in \eqref{eq:small_noise}. 
In contrast, one expects that small scales are more active at the \emph{Kolmogorov length scale} $\sim \nu^{-\ell_d}$ where $\ell_2=\frac{1}{2}$ and $\ell_3=\frac{3}{4}$, c.f., \cite[pp.\ 350]{Vallis06}. As $u_{\Sm}^{\nu}$ in \eqref{eq:small_noise} is a model for small scales, then one expects that $k\mapsto \theta_k^{\nu}$ `accumulates' at (high) frequencies $\sim \nu^{-\ell_d}$.

Due to the previous facts and the support condition in our results \eqref{eq:supp_theta_high_frequencies}, we conjecture that the anomalous dissipation in Theorem \ref{t:anomalous_diss_NS}, which is expected in the KB theory of 2D turbulence, holds if and only if $\supp\,\theta^{\nu}\subseteq \{k\in \Z_0^2\,:\, |k|\eqsim \nu^{-\delta}\}$ for some $\delta\geq \frac{1}{2}$ independent of $\nu$. In the case of passive scalars, in agreement with the fact that anomalous dissipation can be obtained by a (deterministic) H\"older continuous flow \cite{AV23_hom,CCS23}, we instead expect that $ \supp\,\theta^\g \subseteq \{k \in \Z_0^d\,:\, |k|\eqsim \g^{-\delta}\}$ for $\delta>0$ is sufficient for obtaining anomalous dissipation for solutions to \eqref{eq:diffusive_scalars}. 
The proof of the above conjectures goes beyond the scope of this manuscript. 

\smallskip

Next, we explore a consequence of the energy conservation for the small scales in combination with \eqref{eq:small_noise}. It is reasonable to postulate that, although the small scales are more excited as $\nu\downarrow 0$, the corresponding `energy' stays \emph{uniformly} bounded in $\nu$. On the one hand, due to the white-in-time ansatz \eqref{eq:small_noise}, the energy cannot be defined directly. On the other hand, the kinetic energy $
\Kin_{\Sm}^{\nu,\phi}$ of the small scales $u^{\nu}_{\Sm}$ at an observable $\phi \in L^2_{{\rm loc}}([0,\infty))$ can be defined as
$$
\Kin_{\Sm}^{\nu,\phi}(t,x)\stackrel{{\rm def}}{=}\Big|\sum_{k\in \Z^d_0} \wt{\theta}_k^{\nu}\, e^{2\pi \i k\cdot x}
\sum_{1\leq \alpha\leq d-1}  a_{k,\alpha}  \int_{0}^t\phi (s)\,\dd W^{k,\alpha}_s\Big|^2 $$ 
for $(t,x)\in\R_+\times \T^d$, and satisfies,
\begin{equation}
\label{eq:energy_small_scales_two_sides}
\E\int_{\T^d}
\Kin_{\Sm}^{\nu,\phi}(t,x)\,\dd x \eqsim_d 
\|\wt{\theta}^{\nu}\|_{\ell^2}^2 \|\phi\|_{L^2(0,t)}^2\ \text{ for }t>0.
\end{equation}
The requirement that the energy of the small scales is bounded, as $\nu\downarrow 0$, therefore implies  
$$
\sup_{\nu\in (0,1)}\|\wt{\theta}^{\nu}\|_{\ell^2}<\infty.
$$ 
The above condition is indeed satisfied in Theorem \ref{t:anomalous_diss_NS} with $d=2$ and $\wt{\theta}^{\nu}_k=\sqrt{2\mu }\,\theta_k^{\nu}$.

\smallskip

To conclude, let us stress that some criticisms can be made about the decomposition \eqref{eq:large_small_scale_decomposition} and the corresponding ansatz \eqref{eq:small_noise}. Indeed, as commented in \cite[Section 1.2]{FL19}, the above scale separation has never been established so strictly in real fluids. However, simplified models such as \eqref{eq:large_small_scale_decomposition}-\eqref{eq:small_noise} can be used to understand basic features of certain phenomena. The reader is referred to \cite{MK99_simplified} for discussions.

\subsection{Strategy, novelty, physical interpretation and comparison}
\label{claim_strategy}
We begin by discussing the strategy used in the proofs of our main results. 

\subsubsection{Strategy}
\label{sss:strategy}
Here, we describe the strategy used to prove Theorem \ref{t:anomalous_diss_diff}. The proof of Theorem \ref{t:anomalous_diss_NS} is slightly more involved due to the presence of the convective nonlinearity. The main step in the proof of Theorem \ref{t:anomalous_diss_diff} can be roughly summarized as follows (cf.\ Corollary \ref{cor:scaling_limit_diffusive}):

\smallskip

{\sc Main Step.}
For each fixed $\mu>0$ and $\g>0$, there exists $\theta^{\g}\in \Set$ such that the unique global solution $\varrho^{\g}$ to \eqref{eq:diffusive_scalars} satisfies
\begin{equation}
\label{eq:main_step_strategy}
\E\|\varrho^{\g}(1)\|_{L^2}^2 \approx 
\|\rhod^{\g}(1)\|_{L^2}^2
\end{equation}
where $\rhod^{\g}$ solves 
\begin{equation}
\label{eq:strategy_varrhod}
\left\{
\begin{aligned}
\partial_t \rhod^{\g} 
&= (\g+\mu)\Delta \rhod^{\g}& \text{ on }&\T^d,\\
\rhod^{\g}(0,\cdot)&=\varrho_0& \text{ on }&\T^d.
\end{aligned}
\right.
\end{equation}
Here,  $\approx$ means that the difference between the two quantities in \eqref{eq:main_step_strategy} can be made as small as needed by choosing $\theta^{\g}$ appropriately. Below, we use that $\varrho_0$ has mean zero by assumption, see Theorem \ref{t:anomalous_diss_diff}.

The advantage of \eqref{eq:strategy_varrhod} compared to \eqref{eq:diffusive_scalars} is that the operator $\mu\Delta$ provides an additional dissipation/diffusion (sometimes referred to as \emph{eddy viscosity}). It is important to note that the same is \emph{not} true for the transport noise in \eqref{eq:diffusive_scalars}. Indeed, due to the incompressibility of the noise coefficients $\nabla \cdot \sigma_{k,\alpha}=0$ required below, the following energy balance holds:
\begin{equation}
\label{eq:energy_balance_diffusion_intro}
\frac{1}{2}\|\varrho^{\g}(t)\|_{L^2}^2 +\int_0^t \int_{\T^d}\g\, |\nabla\varrho^{\g} |^2\,\dd x\,\dd s
=\frac{1}{2}\|\varrho_0\|_{L^2}^2 \ \text{ a.s.\ for all }t>0,
\end{equation}
where $\varrho^\g$ solves \eqref{eq:diffusive_scalars}.
Therefore, the transport noise in \eqref{eq:diffusive_scalars} with intensity $\mu>0$ does not produce any additional dissipation. 

\smallskip
 
If \eqref{eq:main_step_strategy} holds, then the proof of Theorem \ref{t:anomalous_diss_diff} readily follows from the additional viscosity/diffusivity in \eqref{eq:strategy_varrhod}:
\begin{align*}
2\, \E\int_0^1 \int_{\T^d}  \g \,|\nabla \varrho^{\g}|^2\,\dd x \,\dd t 
& \stackrel{\eqref{eq:energy_balance_diffusion_intro}}{=}\|\varrho_0\|_{L^2(\T^d)}^2 - 
\E\|\varrho^{\g}(1)\|_{L^2(\T^d)}^2\\
&\stackrel{\eqref{eq:main_step_strategy}}{\approx}\|\varrho_0\|_{L^2(\T^d)}^2 - 
\|\rhod^{\g}(1)\|_{L^2(\T^d)}^2\\ 
&\ \ \stackrel{(i)}{\geq}(1-e^{-\mu/(4\pi^2)})\|\varrho_0\|_{L^2(\T^d)}^2 
\end{align*}
where in $(i)$ we used that from the energy inequality for \eqref{eq:strategy_varrhod} it holds that $\|\rhod^{\g}(t)\|_{L^2(\T^d)}^2\leq e^{-[(\mu+\g) t]/(4\pi^2)}\|\varrho_0\|_{L^2(\T^d)}^2\leq e^{-(\mu t)/(4\pi^2)}\|\varrho_0\|_{L^2(\T^d)}^2$, due to the increased viscosity/diffusivity in \eqref{eq:strategy_varrhod}. Here, we used $\int_{\T^d}\rhod^\g(t,\cdot)\,\dd x =0$ for all $t\geq 0$ as $\int_{\T^d}\varrho_0\,\dd x =0$, and that the Poincar\'e constant is $1/(2\pi)$ (the latter follows from the Plancherel inequality and $\T=\R/\Z$).
Hence, Theorem \ref{t:anomalous_diss_diff} follows by taking the $\liminf_{\g\downarrow 0}$ in the above.

\subsubsection{Novelty -- Improved scaling limits via Meyers' estimates}
\label{sss:novelty}
The key ingredient in the proof of \eqref{eq:main_step_strategy} is an improvement of a well-established scaling limit argument due to \cite{FGL21,FL19}.
Our improvement of such a scaling limit argument reads roughly as follows (cf.\ Proposition \ref{prop:scaling_diff}): 

For all $\mu,\g>0$ and sequences $(\theta^n)_{n\geq 1}\subseteq \Set$ (see \eqref{eq:symmetric_coefficients_set}) satisfying
\begin{equation}
\label{eq:sequence_goes_to_zero_linfty}
\displaystyle{\lim_{n\to \infty} \|\theta^n\|_{\ell^{\infty}}=0}
\end{equation}
it holds that
\begin{equation}
\label{eq:convergence_varrho_strong}
\lim_{n\to \infty} \varrho^n =\rhod^{\g} \text{ in probability in }C([0,1];L^2(\T^d)),
\end{equation}
where $\varrho^n$ is the unique global solution to \eqref{eq:diffusive_scalars} with $\theta^{\g}=\theta^n$.

\smallskip

A sequence in $\Set$ satisfying \eqref{eq:sequence_goes_to_zero_linfty} is given in  \cite[eq.\ (1.9)]{L23_enhanced} (see also \eqref{eq:choice_thetan}).

The main novelty and improvement compared to the above-mentioned works of \eqref{eq:convergence_varrho_strong} is that the limit is established in a space of \emph{zero smoothness} uniformly in time. More precisely, in previous works, limits as in \eqref{eq:convergence_varrho_strong} were only established with $L^2(\T^d)$ replaced by $H^{-\varepsilon}(\T^d)$ for some $\varepsilon>0$ (cf.\ \cite[Proposition 3.7]{FGL21} and \cite[Theorem 1.4]{FL19}). 
The validity of \eqref{eq:convergence_varrho_strong} requires \eqref{eq:diffusive_scalars} to be parabolic, i.e., $\gamma > 0$, and discriminates between the parabolic ($\gamma > 0$) and hyperbolic ($\gamma = 0$) regimes; see \cite{butori2024background,MR4238216} for some results in the hyperbolic setting. It is worth noticing that \eqref{eq:convergence_varrho_strong} implies the \emph{non-convergence} of $\nabla \varrho^n$ to $\nabla \rhod^\g$, see Remark \ref{r:lack_convergence_gradient}. 

The improvement in \eqref{eq:convergence_varrho_strong} is of central importance to establish \eqref{eq:main_step_strategy}, and thus for our approach to anomalous dissipation. Indeed, \eqref{eq:convergence_varrho_strong} and energy estimates readily imply \eqref{eq:main_step_strategy} for a sufficiently large $n\geq 1$; while this is not true if one uses $H^{-\varepsilon}(\T^d)$ instead of $L^2(\T^d)$ in \eqref{eq:convergence_varrho_strong}. 
Interestingly, we can also prove that \eqref{eq:convergence_varrho_strong} holds with $L^2(\T^d)$ replaced by $H^{r}(\T^d)$ for some $r(\g,\mu)>0$. The latter fact is needed to deal with, e.g., 2D NSEs \eqref{eq:NS_2d_vorticity} (see the proof of Corollary \ref{cor:strong_convergence_theta}) and in the velocity formulation in Section \ref{s:2D_NSE_velocity}. 

\smallskip

The key tool behind \eqref{eq:convergence_varrho_strong} is the \emph{stochastic Meyers' estimates} proven in Section \ref{app:Meyers}. In particular, they imply the existence of $r_0(\g,\mu)>0$ for which 
\begin{equation}
\label{eq:uniform_estimate_varrhon_strategy}
\sup_{n\geq1 }\,\E\sup_{t\in [0,1]}\|\varrho^{n}(t)\|_{H^{r_0}(\T^d)}^2<\infty
\end{equation} 
where $\varrho^{n}$ is as in \eqref{eq:convergence_varrho_strong}, i.e., the unique global solution to \eqref{eq:diffusive_scalars} with $\theta^{\g}=\theta^n$. After \eqref{eq:uniform_estimate_varrhon_strategy} is proved, then the claim \eqref{eq:convergence_varrho_strong} follows from a compactness argument. 

The main difficulty behind the proof of \eqref{eq:uniform_estimate_varrhon_strategy} is the \emph{lack} of uniformity of regularity of the noise coefficients $(\theta^n_k \sigma_{k,\alpha})_{k,\alpha}$ in \eqref{eq:diffusive_scalars} along the sequence $(\theta^n)_{n\geq 1}$ satisfying \eqref{eq:sequence_goes_to_zero_linfty}. 
More precisely, one has, for all $r>0$,
\begin{equation}
\label{eq:smoothness_coefficients}
\sup_{n\geq 1} \|(\theta^n_k \sigma_{k,\alpha})_{k,\alpha}\|_{L^{\infty}(\T^d;\ell^2)}<\infty\ \  \text{ while }\ \ 
\sup_{n\geq 1} \|(\theta^n_k \sigma_{k,\alpha})_{k,\alpha}\|_{H^{r}(\T^d;\ell^2)}=\infty;
\end{equation}
see \cite[Proposition 4.1]{Arole_25} for details. 
Let us stress that the supremum over $n\geq 1$ in the above is essential as it can hold that $\#\{k\,:\, \theta_k^n \neq 0\}<\infty$ and therefore $(\theta^n_k \sigma_{k,\alpha})_{k,\alpha}\in C^{\infty}(\T^d;\ell^2)$ for each fixed $n\geq 1$.
In particular, \eqref{eq:uniform_estimate_varrhon_strategy} cannot be derived from well-known results on $L^p$-theory of SPDEs, as it always requires some degree of regularity of the coefficients and provides a high improvement of the regularity of solutions to the SPDE under consideration, see e.g., \cite{AV_torus,Kry}. In contrast, Meyers' estimates provide a small improvement on the regularity of solutions of (S)PDEs with bounded, measurable, and parabolic coefficients, which is exactly the setting one encounters when dealing with the scaling limit \eqref{eq:sequence_goes_to_zero_linfty}-\eqref{eq:convergence_varrho_strong}.
Indeed, the boundedness is given by \eqref{eq:smoothness_coefficients}, while the parabolicity is uniform due to the Stratonovich formulation of the transport noise, see the proof of Lemma \ref{l:uniform_estimates_diffusive} for details. 
To conclude, let us stress that, due to the lack of uniform smoothness as described in \eqref{eq:smoothness_coefficients}, \eqref{eq:uniform_estimate_varrhon_strategy} \emph{cannot} hold with $r_0$ large. Counterexamples in the deterministic counterpart can be found in \cite{BMV24,M21_singularity}.

\subsubsection{Physical interpretation -- Connection with homogenization}
Arguing as in \cite[Subsection 2.2]{A23}, the scaling limit result of \eqref{eq:sequence_goes_to_zero_linfty}-\eqref{eq:convergence_varrho_strong} can be interpreted as a homogenization result for the SPDE \eqref{eq:diffusive_scalars}. Indeed, as discussed in Subsection \ref{ss:heuristics_viscosity_dependent}, the transport noise term in \eqref{eq:diffusive_scalars} (or, more precisely, for 2D NSEs \eqref{eq:NS_2d_vorticity} for which a similar discussion applies) can be interpreted as the contribution of the `small scales' of the fluid flow. 
Therefore, reasoning as in \cite[Subsection 2.2]{A23}, the limit \eqref{eq:sequence_goes_to_zero_linfty} can be interpreted as `zooming out' from small scales, and the PDE \eqref{eq:strategy_varrhod} can be thought of as the `effective equation' for the SPDE \eqref{eq:diffusive_scalars}. In our context, the `microscopic' parameter is $\varepsilon=\|\theta^n\|_{\ell^{\infty}}\downarrow 0$, cf.\ \eqref{eq:sequence_goes_to_zero_linfty}. 

The homogenization viewpoint is also interesting from a mathematical perspective. Indeed, the main tool to achieve \eqref{eq:convergence_varrho_strong} are the (stochastic) Meyers' estimates which also play a fundamental role in the context of homogenization of PDEs (see e.g., \cite{ABM18_stoc,AKM19,GNO15_stoc,S18_hom_book}) since, as it is well-known in homogenization, the only quantities remaining uniform along the process of zooming out from small scales (i.e., $\varepsilon\downarrow 0$) are measurability, boundedness, and parabolicity.

\subsubsection{Further comparison with the literature}
\label{ss:comparison}
As commented in Subsection \ref{sss:novelty}, the main novelty consists of an improvement of the scaling limit \eqref{eq:convergence_varrho_strong}. Let us note that such an improvement is also interesting on its own. Indeed, the latter can also be used to strengthen other results on the topic. For instance, arguing as in Section \ref{s:anomalous_NS},
one can check that the main results of \cite{FGL21} can be extended also to the case of $L^2(\T^d)$ being critical for the corresponding SPDE (here, we use criticality in the sense of \cite[Section 3]{AV24_variational} and are implicitly assuming that \cite[(H4)]{FGL21} holds with $L^2$ replaced by $H^r$ with $r>0$). For instance, this covers the case of the Keller-Segel model system in four dimensions, cf. \cite[Subsection 2.1]{FGL21}.

\smallskip

We are now in a position to give a comparison between our results and those of 
\cite[Section 3]{HPZZ23} and \cite{R23_kraichnan}.
Our results are, in spirit, similar to the ones in \cite{R23_kraichnan} which, roughly speaking, ensure that anomalous dissipation can happen if the noise is chosen according to the viscosity and/or diffusivity (cf.\ the comments below \cite[Theorem 1.1]{R23_kraichnan}). On the one hand, in our results, due to the application of the scaling limit in \eqref{eq:sequence_goes_to_zero_linfty}-\eqref{eq:convergence_varrho_strong}, we cannot provide an explicit dependence on $\supp\,\theta^\g$ on $\g$ while those in \cite{R23_kraichnan} are rather explicit. 
On the other hand, in our results, we have a sharp control of how the intensity of the noise affects anomalous dissipation, and our methods are robust enough for applications to nonlinear SPDEs, e.g., 2D NSEs in vorticity formulation leading to anomalous dissipation of enstrophy as conjectured in the KB theory.  
From a technical point of view, the approaches are completely different and thus a comparison is not possible.

In \cite{HPZZ23}, the authors showed total dissipation for passive scalars and 2D and (even) 3D weak solutions to the NSEs. Here, total dissipation means that all the initial energy $\|\varrho_0\|_{L^2}>0$ is absorbed by the energy dissipation rates, i.e., $\lim_{t\uparrow 1}\|\varrho^{\g}(t)\|_{L^2} =0$ a.s.\ for all $\g>0$. 
Our results share some similarities with those of \cite[Section 3]{HPZZ23}, where the authors considered transport noise. 
For instance, to the best of our knowledge, they were the first to use the constructions in \cite{FGL21,G20_convergence} in the context of anomalous dissipation.
However, there are several differences between our results and those in \cite[Section 3]{HPZZ23}, and in particular in the mechanism for creating anomalous dissipation. Firstly, instead of assuming a viscosity/diffusivity dependent $\theta$ in either \eqref{eq:NS_2d_vorticity} or \eqref{eq:diffusive_scalars}, they considered a piecewise constant in time noise coefficient $\theta(t)=(\theta_{k}(t))_{k\in \Z^d_0}$ such that $\|\theta(t)\|_{\ell^2}\to \infty$ as $t\uparrow 1$, see \cite[Subsection 2.1]{HPZZ23}. In particular, following the argument in Subsection \ref{ss:heuristics_viscosity_dependent}, extending the identity \eqref{eq:energy_small_scales_two_sides} to the case of piecewise constant in time $\theta$, one sees that their noise has `infinite energy' as $t\uparrow 1$. This is not the case in our construction where, due to the requirement $\theta^{\nu}\in \Set$ for all $\nu>0$, we obtain finite energy of the noise as $\nu\downarrow 0$ and uniformly for $t\in [0,1]$. However, our method cannot deal with 3D NSEs in contrast to \cite{HPZZ23} (see also comments below Theorem \ref{t:anomalous_diss_NS}). Secondly,  from a technical point of view, the proofs are completely different. Indeed, in \cite[Section 3]{HPZZ23}, the authors rely on a careful study of the mixing properties of the transport noise and subsequently upgrade them to an anomalous dissipation result. This approach in particular requires $\sup_{t\in [0,1)}\|\theta(t)\|_{\ell^2}= \infty$. In our manuscript, instead, we employ the above-discussed refinement of the scaling limit arguments of \cite{G20_convergence,FGL21} via Meyers' estimates. 
To conclude, it would be interesting to see if the arguments used in \cite[Section 4]{HPZZ23} to derive anomalous dissipation by advection via randomly forced NSEs can benefit from the use of Meyers' estimates.
However, this goes beyond the scope of the current manuscript.

\section{Preliminaries}

\subsection{Notation}
\label{ss:notation}
Here, we collect the basic notation used in the manuscript.  
For given parameters $p_1,\dots,p_n$, we write $R(p_1,\dots,p_n)$ if the quantity $R$ depends only on $p_1,\dots,p_n$.
For two quantities $x$ and $y$, we write $x\lesssim y$, if there exists a constant $C$ such that $x\le Cy$. If such a $C$ depends on above-mentioned parameters $p_1,\dots,p_n$ we either mention it explicitly or indicate this by writing $C{(p_1,\dots,p_n)}$ and correspondingly $x\lesssim_{p_1,\dots,p_n}y$ whenever $x\le C{(p_1,\dots,p_n)}y$. We write $x\eqsim_{p_1,\dots,p_n} y$, whenever $x\lesssim_{p_1,\dots,p_n} y$ and $y\lesssim_{p_1,\dots,p_n}x$.

Below, $(\Omega, \mathscr{A},(\mathscr{F}_t)_{t\geq 0}, \P)$ denotes a filtered probability space carrying a sequence of independent standard Brownian motions which changes depending on the SPDE under consideration. We write $\E$ for the expectation on $(\Omega, \mathscr{A}, \P)$ and $\mathscr{P}$ for the progressive $\sigma$-field. A process $\phi:[0,\infty)\times \O\to X$ is progressively measurable if $\phi|_{[0,t]\times \O}$ is $\Borel([0,t])\otimes \F_t$ measurable for all $t\geq 0$, where $\Borel$ is the Borel $\sigma$-algebra on $[0,t]$ and $X$ a Banach space. Moreover, a stopping time $\tau$ is a measurable map $\tau:\O\to [0,\infty]$ such that $\{\tau\leq t\}\in \F_t$ for all $t\geq 0$. 
For a stopping time $\tau$, we define the stochastic interval $[0,\tau)\times \O$ as 
\begin{equation*}
[0,\tau)\times \O\stackrel{{\rm def}}{=}\{(t,\om)\in[0,\infty)\times \O\,:\,\,0\leq t<\tau(\om)\}.
\end{equation*} 
A similar definition is employed for $[0,\tau]\times \O$, $(0,\tau]\times \O$ etc.
Finally, a stochastic process $\phi:[0,\tau)\times \O\to X$ is progressively measurable if $\one_{[0,\tau)\times \O}\,\phi$ is progressively measurable where $\one_{[0,\tau)\times \O}$ stands for the extension by zero outside the set $[0,\tau)\times \O$, and $\tau$ is a stopping time. 

Next, we collect the notation for function spaces. 
We write $L^p(S,\mu;X)$ for the Bochner space of strongly measurable, $p$-integrable $X$-valued functions for a measure space $(S,\mu)$ and a Banach space $X$ as defined in \cite[Section 1.2b]{Analysis1}. 
If $X=\R$, we write $L^p(S,\mu)$, and if it is clear which measure we refer to, we also leave out $\mu$. Finally, $\one_A$ denotes the indicator function of $A\subseteq S$.
Bessel-potential spaces are indicated as usual by $H^{s,q}(\T^d)$ where $s\in \R$ and $q\in (1,\infty)$. 
We also use the standard shorth-hand notation $H^{s,2}(\T^d)$ for $H^{s}(\T^d)$. We sometimes also employ Besov spaces $B^{s}_{q,p}(\T^d)$ which can be defined as the real interpolation space $(H^{-k,q}(\T^d),H^{k,q}(\T^d))_{(s+k)/2k,p}$ for $s\in \R$, $\N\ni k>|s|$ and $1<q,p<\infty$. The reader is referred to \cite{BeLo,Analysis1} for details on interpolation and to \cite[Section 3.5.4]{schmeisser1987topics} for details on function spaces over $\T^d$. Finally, we let $\mathcal{A}(\T^d;\R^k)\stackrel{{\rm def}}{=}(\mathcal{A}(\T^d))^k$ and $\mathcal{A}(\cdot)\stackrel{{\rm def}}{=}\mathcal{A}(\T^d;\cdot)$ for $\mathcal{A}\in \{L^q,H^{s,q},B^{s}_{q,p}\}$ and $k\in \N$.

\subsection{Structure of the noise}
\label{ss:structure_noise}
Here, we describe the quantities $\theta^{\nu}$, $\theta^{\g}$, $\sigma_{k,\alpha}$, $W^{k,\alpha}$ and $W^k$ appearing in the stochastic perturbation of \eqref{eq:NS_2d_vorticity} and \eqref{eq:diffusive_scalars}, while $c_d\stackrel{{\rm def}}{=}d/(d-1)$. In this subsection, we follow \cite{FGL21,FL19}. 
In the following, we only describe the stochastic perturbation in \eqref{eq:diffusive_scalars}, for the case of NSEs \eqref{eq:NS_2d_vorticity}, it is enough to take $d=2$ in the following construction and omit the index $\alpha$. Moreover, since $\g>0$ is fixed here, we simply write $\theta$ instead of $\theta^\g$.

To begin, set $\Z_0^d\stackrel{{\rm def}}{=}\Z^d\setminus\{0\}$.  
Throughout the manuscript $\theta=(\theta_k)_{k\in \Z^d_0}\in \ell^2(\Z^d_0)$ is normalized and radially symmetric, i.e.\
\begin{equation}
\label{eq:theta_normalized_symmetric}
\|\theta\|_{\ell^2(\Z^d_0)}=1 \quad \text{ and } \quad \theta_{j}=\theta_k \  \text{ for all $j,k\in\Z_0^d$ \  such that }|j|=|k|.
\end{equation}
Next, we define the family of vector fields $(\sigma_{k,\alpha})_{k,\alpha}$. 
Here and in the following, we use the shorthand subscript `$k,\alpha$' to indicate $k\in \Z^d_0,\, \alpha\in\{1,\dots,d-1\}$. 
Let $\Z_{+}^d$ and $\Z_-^d$ be a  partition of $\Z_0^d$ such that $-\Z_+^d=\Z_-^d$. For any $k\in \Z_+^d$, let $\{a_{k,\alpha}\}_{\alpha\in \{1,\dots,d-1\}}$ be a complete orthonormal basis of the hyperplane $k^{\bot}=\{k'\in \R^d\,:\, k\cdot k'=0\}$, 
and set $a_{k,\alpha}\stackrel{{\rm def}}{=}a_{-k,\alpha} $ for $k\in \Z^d_-$. Finally, we let
\begin{equation*}
\sigma_{k,\alpha}\stackrel{{\rm def}}{=} a_{k,\alpha} e^{2\pi \i k\cdot x}  \ \  \text{ for all } \ \ x\in \Tor^d,\ k\in \Z^d_0,\ \alpha\in \{1,\dots,d-1\}.
\end{equation*}
By construction, the vector field $\sigma_{k,\alpha}$ is smooth and divergence-free for all $k,\alpha$.

Finally, we introduce the family of complex Brownian motions $(W^{k,\alpha})_{k,\alpha}$. 
Let $(B^{k,\alpha})_{k,\alpha}$ be a family of independent standard (real) Brownian motions on the above-mentioned filtered probability space $(\O,\A,(\F_t)_{t\geq 0},\P)$. Then, we set
\begin{equation}
\label{eq:complex_BM}
W^{k,\alpha}\stackrel{{\rm def}}{=}
\left\{
\begin{aligned}
&B^{k,\alpha}+ \i B^{-k,\alpha},& \qquad& k\in \Z^d_+,\\
&B^{-k,\alpha}- \i B^{k,\alpha}, &\qquad& k\in \Z^d_-.
\end{aligned}
\right.
\end{equation}
In particular $\overline{W^{k,\alpha}}= W^{-k,\alpha}$ for all $k,\alpha$.

As in \cite[Section 2.3]{FL19} or \cite[Remark 1.1]{FGL21}, by \eqref{eq:theta_normalized_symmetric} and the definition of the vector fields $\sigma_{k,\alpha}$, at least formally, one has
\begin{equation}
\begin{aligned}
\label{eq:Ito_stratonovich_change}
\sqrt{c_d \ellip}\sum_{k,\alpha}\theta_k (\sigma_{k,\alpha}\cdot \nabla) \varrho\circ \dot{W}_t^{k,\alpha}
=
\ellip \Delta \varrho+  \sqrt{c_d \ellip}\sum_{k,\alpha}\theta_k (\sigma_{k,\alpha}\cdot \nabla) \varrho\,  \dot{W}_t^{k,\alpha},
\end{aligned}
\end{equation}
where we recall that $c_d=d/(d-1)$.
To see the above, one uses that $\nabla\cdot\sigma_{k,\alpha}=0$, \eqref{eq:theta_normalized_symmetric} and the elementary identity (cf.\ \cite[eq.\ (2.3)]{FL19} or \cite[eq.\ (3.2)]{FGL21})
\begin{equation}
\label{eq:ellipticity_noise}
\sum_{k,\alpha} \theta_k^2\, \sigma_{k,\alpha}\otimes \overline{\sigma_{k,\alpha}}
=
\sum_{k,\alpha} \theta_k^2 \,a_{k,\alpha}\otimes a_{k,\alpha}
 =\frac{1}{c_d}\mathrm{Id}_{d\times d}\ \   \text{ on }\T^d.
\end{equation}
Let us remark that the stochastic integration on the RHS\eqref{eq:Ito_stratonovich_change} is understood in the It\^{o}--sense.
In the paper, we will always understand the Stratonovich noise on the LHS\eqref{eq:Ito_stratonovich_change} as the RHS\eqref{eq:Ito_stratonovich_change}, namely an It\^o noise plus a diffusion term. 
However, note that the diffusion term $\ellip\Delta \varrho$ does \emph{not} provide any additional diffusion, as in energy estimates it balances the It\^{o} correction coming from the It\^{o}-noise, cf.\ \eqref{eq:energy_balance_diffusion_intro}.

\subsection{Solution concept and well-posedness}
Here, we define solutions to 2D NSEs \eqref{eq:NS_2d_vorticity} and to the scalar equation \eqref{eq:diffusive_scalars}. 
We begin by recalling that the sequence of complex Brownian motions $(W^{k,\alpha})_{k,\alpha}$ induces an $\ell^2$-cylindrical Brownian motion $\mathcal{W}_{\ell^2}$ via the formula
\begin{equation}
\label{eq:def_cylindrical_noise}
\mathcal{W}_{\ell^2}(f)=\sum_{k,\alpha}\int_{\R_+} f_{k,\alpha}(t)\,\dd W^{k,\alpha}_t 
\ \ \  \text{ for } \ \ \
f=(f_{k,\alpha})_{k,\alpha}\in L^2(\R_+;\ell^2).
\end{equation}
Note that $
\mathcal{W}_{\ell^2}(f)$ is real-valued if $f_{k,\alpha}=f_{-k,\alpha}$ as $\overline{W^{k,\alpha}}= W^{-k,\alpha}$ for all $k,\alpha$. Using this and the symmetry of $\sigma_{k,\alpha}$ under the reflection $k\mapsto -k$, one can check that the stochastic perturbation in \eqref{eq:NS_2d_vorticity} and \eqref{eq:diffusive_scalars} can be rewritten only using real-valued coefficients and real-valued Brownian motions $(B^{k,\alpha})_{k,\alpha}$ in \eqref{eq:complex_BM}. For instance,
\begin{align}
\label{eq:equivalence_complex_real_noise}
\sum_{k,\alpha} \theta^{\g}_k(\sigma_{k,\alpha}\cdot\nabla) \varrho\, \dot{W}^{k,\alpha}
&= \sum_{k\in \Z^d_+,\alpha\in \{0,\dots,d-1\}} 2\,\theta^{\g}_k \,(\Re \sigma_{k,\alpha}\cdot\nabla) \varrho\, \dot{B}^{k,\alpha}\\
\nonumber
&+ \sum_{k\in \Z^d_-,\alpha\in \{0,\dots,d-1\}} 2\,\theta^{\g}_k\,(\Im \sigma_{k,\alpha}\cdot\nabla) \varrho \, \dot{B}^{k,\alpha}
\end{align}
for any real-valued function $\varrho$.
Hence, solutions to \eqref{eq:NS_2d_vorticity} and \eqref{eq:diffusive_scalars} are naturally real-valued. 
However, the complex formulation of the noise is more convenient for computations and has been widely employed in related works (e.g., \cite{FGL21,FGL21_quantitative,FL19,L23_enhanced}).  

We begin by defining solutions to the passive scalar equation \eqref{eq:diffusive_scalars}. Below, we use the reformulation of the Stratonovich noise in \eqref{eq:diffusive_scalars}.

\begin{definition}[Solutions -- Passive scalars]
\label{def:sol_passive}
Fix $\g>0$, $\theta^{\g}\in \ell^2$ and $\varrho_0\in L^2(\T^d)$. 
Let $\tau^{\g}:\O\to [0,\infty]$ and $
\varrho^{\g}:[0,\tau^{\g})\times \O\to H^1(\T^d) 
$ 
be a stopping time and a progressive measurable process, respectively.
\begin{itemize}
\item We say that $(\varrho^{\g},\tau^{\g})$ is a \emph{local solution} to \eqref{eq:diffusive_scalars} if there exists a sequence of stopping times $(\tau^\g_n)_{n\geq 1}$ such that $\tau^\g_n\leq \tau^\g$ a.s.\ for all $n\geq 1$, and $\lim_{n\to \infty}\tau^\g_n=\tau^\g$ a.s., for which the following conditions hold for all $n\geq 1$:
\begin{itemize}
\item $\varrho^{\g}\in L^2(0,\tau_n^{\g};H^1(\T^d))\cap C([0,\tau_n^{\g}];L^2(\T^d))$ a.s.;
\item a.s.\ for all $t\in [0,\tau^{\g}_n]$ it holds that 
\begin{align*}
\varrho^{\g}(t)-\varrho_0
= &\,(\g+\mu) \int_0^t \Delta \varrho^{\g}(s)\,\dd s\\
+ &\, \sqrt{c_d\mu}\int_0^t \one_{[0,\tau^{\g}_n]}\Big((\theta^{\g}_k\,\sigma_{k,\alpha}\cdot\nabla)\varrho^{\g}(s)\Big)_ {k,\alpha}\,\dd 
\mathcal{W}_{\ell^2}.
\end{align*}

\end{itemize}
\item A local solution $(\varrho^{\g},\tau^{\g})$ to \eqref{eq:diffusive_scalars} is a said to be a \emph{unique maximal (local) solution} to \eqref{eq:diffusive_scalars} if for any other local solution $(\chi^{\g},\lambda^{\g})$ we have $\lambda^{\g}\leq \tau^{\g}$ a.s.\ and $\chi^{\g}=\varrho^{\g}$ a.e.\ on $[0,\lambda^{\g})\times \O$. 
\item A unique maximal solution $(\varrho^{\g},\tau^{\g})$ to \eqref{eq:diffusive_scalars} is said to be a \emph{unique global solution} to \eqref{eq:diffusive_scalars} if $\tau^{\g}=\infty$ a.s.\ 
\end{itemize}
\end{definition}
 
Note that the deterministic and stochastic integral in Definition \ref{def:sol_passive} are well-defined as a $H^{-1}(\T^d)$-valued Bocher and $L^2(\T^d)$-valued It\^o integrals, respectively. 
The solutions of Definition \ref{def:sol_passive} are \emph{weak} in the PDE sense, but \emph{strong} in the probabilistic sense. In case of global solutions, we only write $\varrho^{\g}$ instead of $(\varrho^{\g},\tau^{\g})$ if no confusion seems likely.

Existence and uniqueness of global solutions to \eqref{eq:diffusive_scalars} in the sense of Definition \ref{def:sol_passive} is standard, see e.g., \cite[Chapter 4]{LR15}. In case $\theta^{\g}_k$ decays sufficiently fast as $|k|\to \infty$ (e.g., $\#\{k\,:\, \theta_k^{\g}\neq 0\}<\infty$), then the unique global solution $\varrho^{\g}$ instantaneously improves its regularity in time and space, cf.\  \cite[Theorems 2.7 and 4.2]{AV23_reactionI}. 

\smallskip

In the case of NSEs \eqref{eq:NS_2d_vorticity}, the definition is analogous. As usual, we interpret the convective term $(\K \vor^{\nu}\cdot\nabla)\vor^{\nu}$ in its conservative form $\nabla \cdot(\K \vor^{\nu}\otimes \vor^{\nu})$ due to the divergence-free condition. 
Note that, as in \eqref{eq:equivalence_complex_real_noise}, one can rearrange the sum in its definition so that only $\Re \sigma_{k,\alpha}$ and $\Im \sigma_{k,\alpha}$ appear. 

\begin{definition}[Solutions -- 2D NSEs vorticity formulation]
\label{def:sol_NSE}
Fix $\nu>0$, $\theta^{\nu}\in \ell^2$ and $\vor_0\in L^2(\T^2)$. Let $\tau^{\nu}:\O\to [0,\infty]$ and $\vor^{\nu}:[0,\tau^\nu)\times \O\to H^1(\T^2)$ be a stopping time and a progressive measurable process, respectively.
\begin{itemize}
\item We say that $(\vor^{\nu},\tau^{\nu})$ is a \emph{local solution} to \eqref{eq:NS_2d_vorticity} if there exists a sequence of stopping times $(\tau^\nu_n)_{n\geq 1}$ such that $\tau^\nu_n\leq \tau^\nu$ a.s.\ for all $n\geq 1$, and $\lim_{n\to \infty}\tau^\nu_n=\tau^\nu$ a.s., for which the following conditions hold for all $n\geq 1$:
\begin{itemize}
\item $\vor^{\nu}\in L^2(0,\tau^\nu_n;H^1(\T^2))\cap C([0,\tau^\nu_n];L^2(\T^2))$ a.s.; 
\item $\K\vor^{\nu} \vor^{\nu}\in L^2(0,\tau^\nu_n;L^2(\T^2;\R^2))$ a.s.;
\item a.s.\ for all $t\in [0,\tau^\nu_n]$ it holds that 
\begin{align*}
\vor^{\nu}(t)-\vor_0
&= \int_0^t \Big[(\nu+\mu) \Delta \vor^{\nu}(s) - \nabla \cdot(\K \vor^{\nu}(s) \vor^{\nu}(s))\Big]\,\dd s\\
&
+ \sqrt{2\mu}\int_0^t \one_{[0,\tau^{\nu}_n]}\Big( \theta^{\nu}_k\,(\sigma_{k}\cdot\nabla)\vor^{\nu}(s)\Big)_ {k}\,\dd 
\mathcal{W}_{\ell^2}.
\end{align*}
\end{itemize}
\item A local solution $(\vor^{\nu},\tau^{\nu})$ to \eqref{eq:NS_2d_vorticity} is said to be a \emph{unique maximal (local) solution} to \eqref{eq:NS_2d_vorticity} if for any other local solution $(\xi^{\nu},\lambda^{\nu})$ we have $\lambda^{\nu}\leq \tau^{\nu}$ a.s.\ and $\xi^{\nu}=\vor^{\nu}$ a.e.\ on $[0,\lambda^{\nu})\times \O$. 
\item A unique maximal solution $(\vor^{\nu},\tau^{\nu})$ to \eqref{eq:NS_2d_vorticity}  is said to be a \emph{unique global solution} to \eqref{eq:NS_2d_vorticity} if $\tau^{\nu}=\infty$ a.s.\ 
\end{itemize}
\end{definition}

As above, the integrals in Definition \ref{def:sol_NSE} are well-defined and, in the case of global solutions, we only write $\vor^{\nu}$ instead of $(\vor^{\nu},\tau^{\nu})$ if no confusion seems likely.

Existence and uniqueness of global solutions to \eqref{eq:NS_2d_vorticity} is standard, see e.g., \cite[Theorem 3.4]{AV24_variational}, \cite[Theorem 7.10]{agresti2025nonlinear} or \cite{FL32_book}. Instantaneous regularization results in case $\theta^{\nu}_k$ decays sufficiently fast as $|k|\to \infty$ can be proven as in \cite[Theorems 2.7 and 2.12]{AV21_NS}.

\section{Anomalous dissipation of energy -- Passive scalars}
\label{s:diffusive_proofs}
This section is devoted to the proof of Theorem \ref{t:anomalous_diss_diff}, and will also serve as a guideline for that of Theorem \ref{t:anomalous_diss_NS}. The strategy used is outlined in Subsection \ref{claim_strategy}.

\subsection{Scaling limit at a fixed diffusivity}
\label{ss:scaling_fixed_diffusivity}
In this subsection, we prove the following scaling limit result for passive scalars \eqref{eq:diffusive_scalars} with fixed diffusivity $\g>0$. 

\begin{proposition}[Scaling limit -- Passive scalars]
\label{prop:scaling_diff}
Fix $\g\in (0,1)$ and $\mu>0$. Let $(\theta^n)_{n\geq 1}\subseteq \ell^2$ be a sequence of normalized radially symmetric coefficients (i.e., satisfying \eqref{eq:theta_normalized_symmetric}) such that
$$
\lim_{n\to \infty} \|\theta^n\|_{\ell^{\infty}}=0.
$$
Let $\delta>0$ and $\varrho_0\in H^\delta(\T^d)$. Then, for all sequences $(\varrho_{0,n})_{n\geq 1}\subseteq H^{\delta}(\T^d)$ such that $\varrho_{0,n}\to \varrho_0$ in $ H^{\delta}(\T^d)$, 
we have 
\begin{equation}
\label{eq:limit_varrho_scaling}
\lim_{n\to \infty} \P\Big(\sup_{t\in [0,1]}\|\varrho_n^{\g}(t)-\rhod^{\g}(t)\|_{L^2(\T^d)}\geq \varepsilon \Big)=0 \  \text{ for all } \varepsilon>0,
\end{equation}
where $\varrho_n^{\g}$ and $\rhod^{\g}$ denotes the unique global solution to \eqref{eq:diffusive_scalars} with $\theta^{\g}=\theta^n$ and initial data $\varrho_{0,n}$ and the unique global solution to
\begin{equation}
\label{eq:diffusion_deterministic}
\left\{
\begin{aligned}
\partial_t \rhod^{\g} 
&= (\g+\mu)\Delta \rhod^{\g}& \text{ on }&\T^d,\\
\rhod^{\g}(0,\cdot)&=\varrho_0& \text{ on }&\T^d;
\end{aligned}
\right.
\end{equation}
respectively.
\end{proposition}

Solutions to \eqref{eq:diffusion_deterministic} are understood as in Definition \ref{def:sol_passive} with trivial noise. 

As outlined in Subsections \ref{sss:strategy} and \ref{sss:novelty}, the main novelty and key point in the above result is the presence of the $L^2$-norm in \eqref{eq:limit_varrho_scaling}, and the key behind such improvement is the use of the stochastic Meyers' estimates of Section \ref{app:Meyers}. In the proof of Proposition \ref{prop:scaling_diff} below, we actually prove a stronger version of Proposition \ref{prop:scaling_diff}. More precisely, we show the existence of $\delta_0=\delta_0(\mu,\g)>0$ such that, for all $\delta\in (0,\delta_0]$ and sequences $(\varrho_{0,n})_{n\geq 1}\subseteq H^{\delta}$ such that $\varrho_{0,n}\to \varrho_0$ in $H^\delta$, it holds that 
\begin{equation}
\label{eq:limit_varrho_scaling_strong}
\lim_{n\to \infty} \P\Big(\sup_{t\in [0,1]}\|\varrho^\g_n(t)-\rhod^\g(t)\|_{H^{r }}\geq \varepsilon \Big)=0
\end{equation}
for all $\varepsilon>0$ and $r<\delta<\delta_0$. We emphasize that the norm appearing in \eqref{eq:limit_varrho_scaling} is well-defined, since--as shown in the proof of Proposition \ref{prop:scaling_diff} below--for all $r<\delta\leq \delta_0$ and $n\geq 1$,
$$
\varrho^{\g}_n \in C([0,\infty);H^{r})\text{ a.s.\ }\quad  \text{ and } \quad \rhod^{\g}\in C([0,\infty);H^{r}).
$$
In contrast to the 2D NSEs analysed in Section \ref{s:anomalous_NS} below, the stronger convergence in \eqref{eq:limit_varrho_scaling_strong} will not be needed.

\smallskip

Before going further, let us mention that a prototype example of a sequence satisfying the assumption of Proposition \ref{prop:scaling_diff} is given by 
\begin{equation}
\label{eq:choice_thetan}
\theta^{n}=\frac{\Theta^{n}}{\|\Theta^{n}\|_{\ell^2}} \quad \text{ where }\quad 
\Theta^n \stackrel{{\rm def}}{=} \one_{\{n\leq |k|\leq 2n\}}\frac{1}{|k|^{r}} \ \ \text{ for some }r>0.
\end{equation}
For details, see eq.\ (1.9) in \cite{L23_enhanced} and the comments below it.
In particular, the above sequence satisfies $\# \{k\,:\, \theta_k^n \neq 0\}<\infty$ for all $n\geq 1$.

\smallskip

Now, we turn to the proof of Proposition \ref{prop:scaling_diff}. 
The key ingredient is the following estimates that are uniform in the class normalized radially symmetric $\theta^\g$.

\begin{lemma}[Uniform in $\theta$-estimates]
\label{l:uniform_estimates_diffusive}
Fix $\g\in (0,1)$ and $\mu>0$. Then there exist $p_0(\g,\mu)>2$ and $C_0(\g,\mu)>0$ such that, for all $p\in [2,p_0]$ and all normalized radially simmetric $\theta^{\g}\in \ell^2$ (i.e., satisfying \eqref{eq:theta_normalized_symmetric}) and $\varrho_0\in B^{1-2/p}_{2,p}$, the unique global solution $\varrho$ to \eqref{eq:diffusive_scalars} satisfies
$$
\E\sup_{t\in [0,1]}\|\varrho^{\g}(t)\|_{B^{1-2/p}_{2,p}}^p
+\E\int_0^1 \|\varrho^{\g}(t)\|_{H^1}^p\, \dd t
\leq C_0\|\varrho_0\|_{B^{1-2/p}_{2,p}}^p.
$$
\end{lemma}

\begin{proof}
The claimed estimate follows from the stochastic Meyers' inequality of Theorem \ref{t:Meyers_parabolic} and the comments below it, see in particular \eqref{eq:time_regularity_estimates_1} and the comments on non-trivial initial data for \eqref{eq:parabolic_linear}. 
For clarity, let us check the assumptions \eqref{eq:boundedness_linear}-\eqref{eq:parabolicity_linear} with constants that are \emph{uniform} in the class of normalized radially symmetric $\theta^{\g}\in \ell^2$. To this end, we employ the real reformulation \eqref{eq:equivalence_complex_real_noise} of the complex transport noise in \eqref{eq:diffusive_scalars}.  
To this end, for $\alpha\in \{0,\dots,d-1\}$ and $x\in\T^d$, we let 
\begin{align*}
\xi_{k,\alpha}(x)\stackrel{{\rm def}}{=}2\sqrt{c_d\mu} \, \theta_k^{\g}
\left\{
\begin{aligned}\Re [\sigma_{k,\alpha}(x)]=
\cos(2\pi k\cdot x)\, a_{k,\alpha}, &\quad& k\in \Z^d_+,\\
\Im [\sigma_{k,\alpha}(x)]=\sin(2\pi k\cdot x)\, a_{k,\alpha}, &\quad& k\in \Z^d_-. 
\end{aligned}
\right.
\end{align*}
From the normalized condition $\|\theta^{\g}\|_{\ell^2}=1$, it immediately follows that 
$
\|(\xi_{k,\alpha})_{k,\alpha}\|_{\ell^2} \leq 2\sqrt{c_d\mu}
$
which yields \eqref{eq:boundedness_linear} with constant uniform in $\theta^{\g}$. 
Finally, to check \eqref{eq:parabolicity_linear}, note that \eqref{eq:ellipticity_noise} and $a_{-k,\alpha}=a_{k,\alpha}$ imply, for all $\eta\in \R^d$, 
$$
\frac{1}{2}\sum_{k,\alpha} (\xi_{k,\alpha}\cdot \eta)^2 =
c_d\mu \sum_{k,\alpha}(\theta_k^{\g})^2\, (a_{k,\alpha}\cdot \eta)^2 =\mu |\eta|^2.
$$ 
Thus, the assumptions \eqref{eq:boundedness_linear}-\eqref{eq:parabolicity_linear} with $(\xi_n)_{n\geq1}$ replaced by an enumeration of $(\xi_{k,\alpha})_{k,\alpha}$ are uniform w.r.t.\ normalized radially simmetric $\theta^{\g}\in \ell^2$, and therefore the claim of Lemma \ref{l:uniform_estimates_diffusive} follows from Theorem \ref{t:Meyers_parabolic} with $\kappa=\g+\mu$ and $\kappa_0=\mu$. Here we also used that the leading differential operator in the It\^o formulation of \eqref{eq:diffusive_scalars} is $(\g+\mu)\Delta $, see Definition \ref{def:sol_passive}. 
\end{proof}

To proceed further, note that, by the It\^o formula (e.g., \cite[Chapter 4]{LR15}) and $\nabla\cdot\sigma_{k,\alpha}=0$, the following energy equality holds for the global solution $\varrho^{\g}$ of \eqref{eq:diffusive_scalars}: 
\begin{equation}
\label{eq:energy_balance_diffusive}
\frac{1}{2}
\|\varrho^{\g}(t)\|_{L^2}^2 +\int_0^t \int_{\T^d}\g\, |\nabla \varrho^{\g}|^2\,\dd x \,\dd s= \frac{1}{2}\|\varrho_0\|^2_{L^2} \quad \text{a.s.\ for all }t>0.
\end{equation}
The above will be used frequently below. 

The following is the last ingredient in the proof of Proposition \ref{prop:scaling_diff}.

\begin{lemma}[Time-regularity estimate]
\label{l:time_regularity_estimate}
Let $\g\in (0,1)$, $\mu>0$ and $a\in (1,\infty)$ be fixed. Fix $\varrho_0\in L^2$, and assume that $\theta^{\g}\in \ell^2$ is normalized and radially symmetric. Let $\varrho^\g$ be the unique global solution to \eqref{eq:diffusive_scalars}, and set  
$$
\mathcal{M}(t)\stackrel{{\rm def}}{=}\sqrt{c_d \mu}\sum_{k} \theta^{\g}_{k,\alpha}\int_0^t (\sigma_{k,\alpha}\cdot\nabla) \varrho^\g \,\dd W^{k,\alpha}_t \qquad \forall t\geq 0.
$$
Then there exists $s_0,r_1,K_0>0$ independent of $(\varrho_0,\theta)$ such that 
\begin{align}
\label{eq:martingale_goes_zero}
\E\big[\|\mathcal{M}\|_{C^{s_0}(0,1;H^{-r_1})}^{2a}\big]
&\leq K_0\|\theta^{\g}\|_{\ell^{\infty}}^{2a}\|\varrho_0\|_{L^2}^{2a},\\
\E\big[\|\varrho^{\g}\|_{C^{s_0}(0,1;H^{-r_1})}^{2a}\big]
&\leq K_0\|\varrho_0\|_{L^2}^{2a}.
\end{align}
\end{lemma}

The key point in the above result is the presence of $\|\theta\|_{\ell^{\infty}}$ on the RHS\eqref{eq:martingale_goes_zero}.

Lemma \ref{l:time_regularity_estimate} is well-known to experts, and it is a consequence of the structure of the noise described in Subsection \ref{ss:structure_noise} and of the energy estimate \eqref{eq:energy_balance_diffusive}, see e.g.,  \cite[Proposition 3.6]{FGL21} or \cite[Lemma 6.3]{A23}. For brevity, we omit the proof.

\begin{proof}[Proof of Proposition \ref{prop:scaling_diff}]
The argument used here is by now well-established; the reader is referred to e.g.,  \cite[Theorem 1.4]{FL19}, \cite[Proposition 3.7]{FGL21} or \cite[Theorem 6.1]{A23}.
We partially repeat it to highlight the fundamental role of the estimate in Lemma \ref{l:uniform_estimates_diffusive} with $p>2$. Moreover, we prove the stronger result in \eqref{eq:limit_varrho_scaling_strong}.

\smallskip

\emph{Step 1: For all $s_0,r_1>0$ and $0\leq r<r_0$, set
\begin{align*}
\Y\stackrel{{\rm def}}{=}C([0,1];H^{r_0})\cap C^{s_0}(0,1;H^{-r_1})\quad\text{ and }\quad
\X\stackrel{{\rm def}}{=}C([0,1];H^{r}).
\end{align*}
The embedding $\Y\embed \X$ is compact. 
Moreover, for any $K\geq 1$, the following set is closed}
$$
\X_K\stackrel{{\rm def}}{=}
\Big\{f\in \X\,:\, \sup_{t\in [0,1]}\|f(t)\|_{L^2}^2+\int_0^1\|\nabla f(t)\|_{L^2}^2\,\dd t\leq K \Big\}.
$$

The final assertion follows from the first one and Fatou's lemma. It remains to prove that $\Y\embed_{{\rm c}}\X$. By interpolation, for all $\wt{r}\in (r,r_0)$ there exists $\wt{s}>0$ such that 
$$
C([0,1];H^{r_0})\cap C^{s_0}(0,1;H^{-r_1})\embed C^{\wt{s}}(0,1;H^{\wt{r}})\embed_{{\rm comp}} C([0,1];H^{r}),$$ 
where we used the Ascoli-Arzelà theorem in the last embedding.
 
 \smallskip

\emph{Step 2: Conclusion}. We begin by collecting some useful facts. Let $p_0>2$ be as in Lemma \ref{l:uniform_estimates_diffusive}. Fix $\delta_0\in (0,1-2/p_0)$, $\delta\in (0,\delta_0]$ and $r,r_0\in [0,\delta)$ such that $r<r_0$. 
Finally, let us select $p\in (2,p_0]$ such that $\delta>1-2/p>r_0$. In particular, we have the following embeddings at our disposal:
\begin{equation}
\label{eq:embedding_initial_data}
H^{\delta}\embed B^{1-2/p}_{2,p}\embed H^{r_0}.
\end{equation}
Now, since $\varrho_{0,n}\to \varrho_0$ in $H^{\delta}$ by assumption, we have $\sup_{n\geq 1}\|\varrho_{0,n}\|_{H^{\delta}}<\infty$. 
Hence, by Lemma \ref{l:uniform_estimates_diffusive}--\ref{l:time_regularity_estimate} and the embedding \eqref{eq:embedding_initial_data},
\begin{equation}
\label{eq:a_priori_estimates}
\sup_{n\geq 1}
\E \Big[ \|\varrho^\g_n\|_{C([0,1];H^{r_0})}^2 + \|\varrho^\g_n\|_{C^{s_0}(0,1;H^{-r_1})}^2 \Big]<\infty.
\end{equation}
Moreover, by the energy equality \eqref{eq:energy_balance_diffusive}, we obtained the quenched estimate:
\begin{equation}
\label{eq:energy_balance_diffusive_2}
\sup_{t\in [0,1]}\|\varrho_n^{\g}(t)\|_{L^2}^2 +\int_0^1 \int_{\T^d} |\nabla \varrho_n^{\g}|^2\,\dd x \,\dd t \leq N_{\g} \ \text{ a.s., }
\end{equation}
for some deterministic constant $N_{\g}\geq 1$ independent of $n$.

To prove \eqref{eq:limit_varrho_scaling}, or even \eqref{eq:limit_varrho_scaling_strong}, it suffices to prove that for each sub-sequence $(n_k)_{k\geq 1}$ we can find a further sub-sequence $(n_{k_j})_{j\geq 1}$ for which 
\begin{equation}
\label{eq:claim_conclusion_diff_eq}
\lim_{n_{k_j}\to \infty} \P\Big(\sup_{t\in [0,1]}\|\varrho^\g_{n_{k_j}}(t)-\rhod^\g(t)\|_{H^{r}}\geq \varepsilon \Big)=0\ \text{ for all }  \varepsilon>0.
\end{equation}
For notational convenience, below, we do not relabel sub-sequences. To begin, let $\X_{N_\g}$ be as in Step 1 with $N_\g$ is as in \eqref{eq:energy_balance_diffusive_2}. Moreover, denote by $\mathcal{L}_n$ the law of $\varrho^{\g}_n$ on $\X_{N_{\g}}$. 
By combining Prokhorov's theorem, Step 1, \eqref{eq:a_priori_estimates} and \eqref{eq:energy_balance_diffusive_2}, it follows that the sequence $(\mathcal{L}_{n})_{n\geq 1}\subseteq \mathcal{P}(\mathcal{X}_{N_{\g}})$ is tight (here $\mathcal{P}$ the space of probability measures on the polish space $\mathcal{X}_{N_{\g}}$).
In particular, there exists $\mu\in \mathcal{P}(\mathcal{X}_{N_{\g}})$ such that 
$
\mathcal{L}_n$ converges weakly to
$ 
\mu.
$
It remains to prove that 
\begin{equation}
\label{eq:mu_equal_delta_det}
\mu=\delta_{\rhod^{\g}}.
\end{equation}
Indeed, by the Portmanteau theorem and the fact that the limit is deterministic, the previous implies \eqref{eq:claim_conclusion_diff_eq}.

To prove \eqref{eq:mu_equal_delta_det}, arguing as either \cite[Step 2, Theorem 6.1]{A23} or \cite[Proposition 3.7]{FGL21}, it follows that 
$$
\mu\big(f\in \X_{N_{\g}}\,:\, f \text{ is a solution to \eqref{eq:diffusion_deterministic} with life-time $\geq 1$}\big)=1.
$$
In the above, we mean that the couple $(f,1)$ is a local solution to \eqref{eq:diffusion_deterministic}, cf.\ Definition \ref{def:sol_passive} and the comments below Proposition \ref{prop:scaling_diff}.
Now, the uniqueness and maximality of solutions to \eqref{eq:diffusion_deterministic} in the class $\X_{N_{\g}}$ yields \eqref{eq:mu_equal_delta_det}.
\end{proof}

We conclude by discussing an interesting consequence of Proposition \ref{prop:scaling_diff}.

\begin{remark}[Lack of $L^2$ gradients convergence]
\label{r:lack_convergence_gradient}
In the setting of Proposition \ref{prop:scaling_diff}, it holds that, in probability in $C([0,1])$,
$$
\lim_{n\to \infty} \int_0^{\cdot}\int_{\T^d}\g\,|\nabla \varrho^{\g}_n|^2\,\dd x \,\dd s 
= \int_0^{\cdot}\int_{\T^d}(\g+\mu)\,|\nabla \rhod^{\g}|^2\,\dd x \,\dd s .
$$
The above follows by arguing as in Subsection \ref{sss:strategy} and using Proposition \ref{prop:scaling_diff}: 
\begin{align*}
2\int_0^{\cdot} \int_{\T^d}\g\,|\nabla \varrho^{\g}_n|^2\,\dd x \,\dd s 
 =&\,\|\varrho_0\|_{L^2(\T^d)}^2-\|\varrho_n^\g(\cdot)\|_{L^2(\T^d)}^2\\
\stackrel{n\to \infty}{\to} &\,\|\varrho_0\|_{L^2(\T^d)}^2-\|\rhod^\g(\cdot)\|_{L^2(\T^d)}^2
=2\int_0^{\cdot}\int_{\T^d}(\g+\mu)\,|\nabla \rhod^{\g}|^2\,\dd x \,\dd s,
\end{align*}
where the last equality follows from the energy balance for \eqref{eq:diffusion_deterministic}. In the above, we also used that $\|\varrho^\g(t)\|_{L^2(\T^d)},\|\rhod(t)\|_{L^2(\T^d)}\leq \|\varrho_0\|_{L^2(\T^d)}$ for all $\g,t>0$ a.s.

In particular, if $\varrho_0\neq 0$ (and therefore $\varrho_0\not\in \R$ as $\int_{\T^d}\varrho_0(x)\,\dd x=0$), then the process $\varrho_n^\g$ does \emph{not} convergence to $ \rhod^\g$ in probability in $L^2(0,1;H^1(\T^d))$.
\end{remark}

\subsection{Proof of Theorem \ref{t:anomalous_diss_diff}}
\label{ss:proofs_diffusive}
The key ingredient in the proof of  Theorem \ref{t:anomalous_diss_diff} is the following consequence of Proposition \ref{prop:scaling_diff} and the choice \eqref{eq:choice_thetan}.

\begin{corollary}
\label{cor:scaling_limit_diffusive}
Let $\g,\varepsilon,\delta\in (0,1)$ and $N,\mu>0$ be fixed. Then there exists a normalized radially symmetric $\theta^{\gamma}\in \ell^2$ (i.e., satisfying \eqref{eq:theta_normalized_symmetric})
for which the following assertion holds. For all $\varrho_0\in H^{\delta}$ satisfying $\|\varrho_0\|_{H^{\delta}}\leq N$, the unique global solution $\varrho^{\g}$ to \eqref{eq:diffusive_scalars} satisfies  
\begin{equation}
\label{eq:choice_varepsilon_diff}
\P\Big(\sup_{t\in [0,1]}\|\varrho^{\g}-\rhod^{\g}\|_{L^2}\leq \varepsilon\Big)>1-\varepsilon,
\end{equation}
where $\rhod^{\g}$ is the unique global solution to \eqref{eq:diffusion_deterministic}.
\end{corollary}

As the proof below shows, by using \eqref{eq:limit_varrho_scaling_strong} instead of \eqref{eq:limit_varrho_scaling} in Proposition \ref{prop:scaling_diff}, we can replace the $L^2$-norm in \eqref{eq:choice_varepsilon_diff} by $H^r$ for some $r>0$ depending only on $(\delta,\mu,\g)$.

\begin{proof}
Let $\delta_0>0$ be as in Proposition \ref{prop:scaling_diff}. Without loss of generality, we can assume that $\delta<\delta_0$. Next, we set
$$
\B_{N,\delta} \stackrel{{\rm def}}{=}\{\varrho_0\in H^{\delta}\,:\,\|\varrho_0\|_{H^{\delta}}\leq N\}.
$$
Fix $\varepsilon>0$.
Let $(\theta_n)_{n\geq 1}$ be as in \eqref{eq:choice_thetan}, and due to the comments below it, we have $\lim_{n\to \infty}\|\theta^n\|_{\ell^{\infty}}=0$. To conclude, it is enough to show that, for all $\varepsilon>0$,
\begin{equation}
\label{eq:claim_diffusive_proof}
\lim_{n\to \infty} \sup_{\varrho_0 \in 
\B_{N,\delta}} \P\Big(\sup_{t\in [0,1]}\|\varrho^\g_n(\varrho_0)-\rhod^\g(\varrho_0)\|_{L^2}\geq \varepsilon\Big)=0.
\end{equation}
where $\varrho^\g_n(\varrho_0)$ and $\rhod^\g(\varrho_0)$ denote the solution to \eqref{eq:diffusive_scalars} with $\theta^{\g}=\theta^n$ and \eqref{eq:diffusion_deterministic} both with initial data $\varrho_0$, respectively. 
We prove \eqref{eq:claim_diffusive_proof} by contradiction. Indeed, if \eqref{eq:claim_diffusive_proof} is not true, then that there exists $\varepsilon_0>0$ and a (not-relabeled) sub-sequence $(\varrho_{0,n})_{n\geq 1}\subseteq \B_{N,\delta}$ such that  
\begin{equation}
\label{eq:contradiction_claim_diffusive}
\lim_{n\to \infty}  \P\Big(\sup_{t\in [0,1]}\|\varrho^\g_{n}(t)-\rhodn^\g(t)\|_{L^2}\geq \varepsilon_0\Big)>0,
\end{equation}
where $\varrho^\g_{n}\stackrel{{\rm def}}{=}\varrho^\g_{n}(\varrho_{0,n})$ and $\rhodn^\g\stackrel{{\rm  def}}{=}\rhod^\g(\varrho_{0,n})$.
In the remaining part of this proof, we show that \eqref{eq:contradiction_claim_diffusive} leads to a contradiction with Proposition \ref{prop:scaling_diff}. To this end, note that, for all $\delta_0\in (0,\delta)$, there exists a (again, not-relabeled) sub-sequence $(\varrho_{0,n})$ such that $\varrho_{0,n}\to \varrho_0$ in $H^{\delta_0}$. 
From the latter, it readily follows that 
\begin{equation}
\label{eq:convergence_1_proof_corollary_rho}
\rhodn^\g\to \rhod^\g \text{ in }C([0,1];L^2).
\end{equation}
Recall from Proposition \ref{prop:scaling_diff} that 
\begin{equation}
\label{eq:convergence_2_proof_corollary_rho}
\lim_{n\to \infty} \P\Big(\sup_{t\in [0,1]}\|\varrho^\g_n(t)-\rhod^\g(t)\|_{L^2}\geq \frac{\varepsilon_0}{2}\Big)=0.
\end{equation} 
Hence, from the triangle inequality, \eqref{eq:convergence_1_proof_corollary_rho} and \eqref{eq:convergence_2_proof_corollary_rho} contradict \eqref{eq:contradiction_claim_diffusive}. Thus, \eqref{eq:claim_diffusive_proof} is proved, and this concludes the proof.
\end{proof}

\begin{proof}[Proof of Theorem \ref{t:anomalous_diss_diff}]
We begin by collecting some useful facts. 
Let $N$ and $\delta$ be as in the statement of Theorem \ref{t:anomalous_diss_NS}. In particular, 
\begin{equation}
\label{eq:data_varrho_proof_improved}
 \|\varrho_0\|_{H^{\delta}}\leq N.
\end{equation}
Moreover, by the mean-zero assumption on $\varrho_0$ and the Poincar\'e inequality $\|\varrho\|_{L^2}\leq \frac{1}{2\pi}\|\nabla \varrho\|_{L^2}$, it follows that the unique global solution $\rhod^{\g}$ of \eqref{eq:diffusion_deterministic} satisfies
\begin{equation}
\label{eq:decay_rhod_L2}
\|\rhod^{\g}(t)\|_{L^2}^2\leq e^{-(\mu t)/(4\pi^2)}\|\varrho_0\|_{L^2}^2 \  \text{ for all }t>0.
\end{equation}
We now split the proof into two steps.

\smallskip

\emph{Step 1: For all $\mu>0$ and $\g,\varepsilon\in (0,1)$, there exists $\theta_{\varepsilon}^\g\in \Set$ such that the unique global solution to \eqref{eq:diffusive_scalars} (see Definition \ref{def:sol_passive}) satisfies} 
\begin{equation}
\label{eq:claim_step_1_proof_anomalous_dissipation_varrho_g}
2\,\E \int_0^1 \int_{\T^d} \g\, |\nabla \varrho^\g|^2\,\dd x \,\dd t \geq (1-\varepsilon-e^{-\mu/(4\pi^2)})\|\varrho_0\|_{L^2}^2
- 2\varepsilon\|\varrho_0\|_{L^2} .
\end{equation}
Recall that $\Set$ is defined in \eqref{eq:symmetric_coefficients_set}. 
From Corollary \ref{cor:scaling_limit_diffusive} applied with $\mu,N,\delta$ and $\g$ as above, there exists $\theta_\varepsilon^\g\in \Set$ (depending also on the reamining parameters $\mu,N$ and $\delta$, which are not displayed as they are considered fixed) for which the unique global solution $\varrho^{\g}$ to \eqref{eq:diffusive_scalars} with $\theta^\g=\theta^\g_{\varepsilon}$ satisfies 
\begin{equation}
\label{eq:choice_varepsilon_diff_application}
\P\Big(\sup_{t\in [0,1]}\|\varrho^{\g}-\rhod^{\g}\|_{L^2}\leq \varepsilon\Big)>1-\varepsilon,
\end{equation}
where $\rhod^{\g}$ is the unique global solution of \eqref{eq:diffusion_deterministic}, with $\mu>0$ as above.

Let $\tau$ be the stopping time defined as
\begin{equation}
\label{eq:def_tau_diff}
\tau\stackrel{{\rm def}}{=}\inf\big\{t\in [0,1]\,:\, \|\varrho^{\g}(t)-\rhod^{\g}(t)\|_{L^2}\geq \varepsilon\big\}
\ \ \text{ where } \ \ \inf\emptyset\stackrel{{\rm def}}{=}1. 
\end{equation}
It follows from \eqref{eq:choice_varepsilon_diff_application} that 
\begin{equation}
\label{eq:tau_larger_than1_with_large_prob}
\P(\tau\geq 1)>1-\varepsilon.
\end{equation}
Applying the It\^o's formula to compute $\|\varrho(1)\|_{L^2}^2$ and using $\nabla\cdot \sigma_{k,\alpha}=0$, we have
\begin{align}
\label{eq:energy_balance_varrho_proof_g}
2\int_0^t \int_{\T^d} \g\, |\nabla \varrho^\g|^2\,\dd x \,\dd t
= \|\varrho_0\|_{L^2}^2- \|\varrho(t)\|_{L^2}^2\ \ \text{ a.s.\ for all }t>0.
\end{align}
In particular, 
\begin{align}
\label{eq:integrated_energy_inequality_varrho_proof_g}
2\,\E\int_0^1 \int_{\T^d} \g\, |\nabla \varrho^\g|^2\,\dd x \,\dd t
&= \|\varrho_0\|_{L^2}^2-\E \|\varrho^\g(1)\|_{L^2}^2\\ 
\nonumber
&= \|\varrho_0\|_{L^2}^2-\E[\one_{\{\tau<1\}} \|\varrho^\g(1)\|_{L^2}^2 ]-\E[\one_{\{\tau\geq 1\}}\|\varrho^\g(1)\|_{L^2}^2].
\end{align}
Next, we estimate each term separately. Firstly, as $\|\varrho^\g(1)\|_{L^2}\leq \|\varrho_0\|_{L^2}$ a.s., due to the energy balance \eqref{eq:energy_balance_varrho_proof_g}, it holds that 
\begin{equation}
\label{eq:integrated_energy_inequality_varrho_proof_g1}
\E[\one_{\{\tau<1\}} \|\varrho^\g(1)\|_{L^2}^2 ]\leq \P(\tau<1) \|\varrho_0\|_{L^2}^2\stackrel{\eqref{eq:tau_larger_than1_with_large_prob}}{\leq} \varepsilon  \,\|\varrho_0\|_{L^2}^2.
\end{equation}
Secondly, by the definition of $\tau$, and once again invoking the energy inequalities for $\varrho^\g$ and $\rhod^\g$, we have, a.s.\ on $\{\tau\geq 1\}$,
\begin{align*}
\|\varrho^\g(1)\|_{L^2}^2
&\leq 
\|\rhod^\g(1)\|_{L^2}^2
+\big|
\|\varrho^\g(1)\|_{L^2}^2-\|\rhod^\g(1)\|_{L^2}^2\big|\\
&\leq
\|\rhod^\g(1)\|_{L^2}^2
+\big(
\|\varrho^\g(1)\|_{L^2}+\|\rhod^\g(1)\|_{L^2}\big)
\|\varrho^\g(1)-\rhod^\g(1)\|_{L^2}\\
&\leq e^{-\mu/(4\pi^2)}\|\varrho_0\|_{L^2}^2
+
2\varepsilon\|\varrho_0\|_{L^2}.
\end{align*}
Therefore,
\begin{equation}
\label{eq:integrated_energy_inequality_varrho_proof_g2}
\E [\one_{\{\tau\geq 1\}}\|\varrho^\g(1)\|_{L^2}^2\big]
\leq  e^{-\mu/(4\pi^2)}\|\varrho_0\|_{L^2}^2
+
2\varepsilon \|\varrho_0\|_{L^2}.
\end{equation}
Hence, \eqref{eq:claim_step_1_proof_anomalous_dissipation_varrho_g} follows by putting together the estimates \eqref{eq:integrated_energy_inequality_varrho_proof_g}, \eqref{eq:integrated_energy_inequality_varrho_proof_g1} and \eqref{eq:integrated_energy_inequality_varrho_proof_g2}.

\smallskip

\emph{Step 2: Conclusion}. Let $(\varepsilon_\g)_{\g\in (0,1)}$ be any family of positive numbers in $(0,1)$ satisfying $\lim_{\g\downarrow 0} \varepsilon_\g=0$. Let
\begin{equation}
\label{eq:choice_theta_g_second_step_proof2}
\theta^\g\stackrel{{\rm def}}{=} \theta^{\g}_{\varepsilon^\g}\in \Set
\end{equation} 
where $\theta^\g_{\varepsilon}$ are as in Step 1. 
From the inequality \eqref{eq:claim_step_1_proof_anomalous_dissipation_varrho_g} and $\|\varrho_0\|_{L^2}\leq N$ by \eqref{eq:data_varrho_proof_improved}, we have 
\begin{align*}
2\liminf_{\g\downarrow 0}
\E \int_0^1 \int_{\T^d} \g\, |\nabla \varrho^\g|^2\,\dd x \,\dd t 
&\geq
 \liminf_{\g\downarrow 0} \big[
(1-\varepsilon_\g-e^{-\mu/(4\pi^2)})\|\varrho_0\|_{L^2}^2
-2  N \varepsilon_\g\big]\\
&= (1-e^{-\mu/(4\pi^2)})\|\varrho_0\|_{L^2}^2,
\end{align*}
where $\varrho^{\g}$ is the unique global solution to \eqref{eq:diffusive_scalars} with $\theta^\g$ as in \eqref{eq:choice_theta_g_second_step_proof2}. 
\end{proof}

As announced in Remark \ref{r:anomalous_dissipation_and_convergence}, the argument used to prove Theorem \ref{t:anomalous_diss_diff}  yields the following slightly different version of the latter.

\begin{proposition}[Anomalous dissipation of energy by transport noise -- Passive scalars II]
\label{prop:anomalous_diss_diff2}
Let $d\in \N_{\geq 1}$, $N\geq 1$ and $\delta>0$. For all $\mu>0$ and $\eta\in (0,1-e^{-\mu/(4\pi^2)})$, there exists a family $(\theta^{\g})_{\g\in (0,1)}\subseteq \Set$ such that, for all mean-zero $\varrho_0\in H^{\delta}(\T^d)$ satisfying $N^{-1}\leq \|\varrho_0\|_{L^2}$ and $ \|\varrho_0\|_{H^{\delta}}\leq N$, it holds that 
$$
\inf_{\g\in ( 0,1)} \E\int_0^1 \int_{\T^d}  \g \,|\nabla \varrho^{\g}|^2\,\dd x \,\dd t \geq \frac{\eta}{2}\,\|\varrho_0\|_{L^2(\T^d)}^2,
$$
where $\varrho^\g$ is the unique global smooth solution to \eqref{eq:diffusive_scalars}.
\end{proposition}

The proof requires only a minor modification of that of Theorem \ref{t:anomalous_diss_diff}.

\begin{proof}
From \eqref{eq:claim_step_1_proof_anomalous_dissipation_varrho_g} in the above proof of Theorem \ref{t:anomalous_diss_diff}, for each $\g,\varepsilon\in (0,1)$, we can find $\theta_\varepsilon^\g\in \Set$ such that 
\begin{align*}
2\E \int_0^1 \int_{\T^d} \g\, |\nabla \varrho^\g|^2\,\dd x \,\dd t 
&\geq (1-\varepsilon-e^{-\mu/(4\pi^2)})\|\varrho_0\|_{L^2}^2
- 2  \varepsilon \|\varrho_0\|_{L^2}\\
&\stackrel{(i)}{\geq}\big[(1-\varepsilon-e^{-\mu/(4\pi^2)})- 2 N\varepsilon\big]\,\|\varrho_0\|_{L^2}^2,
\end{align*}
where in $(i)$ we used that $\|\varrho_0\|_{L^2}\geq N^{-1}$ by assumption.
Since $\eta<1-e^{-\mu/(4\pi^2)}$, the claimed estimate follows by choosing $\varepsilon_\eta$ depending only on $N$ and $\mu$ such that $\big[(1-\varepsilon_\eta-e^{-\mu/(4\pi^2)})- 2 N\varepsilon_\eta\big]\geq \eta$, with a corresponding choice of $\theta^\g=\theta^\g_{\varepsilon_\eta}$.
\end{proof}

\section{Anomalous dissipation of enstrophy -- 2D NSEs}
\label{s:anomalous_NS}
This section is devoted to the proof of Theorem \ref{t:anomalous_diss_NS}. Its proof follows the one of Theorem \ref{t:anomalous_diss_diff} given in Section \ref{s:diffusive_proofs} as the convective nonlinearity creates only minor difficulties. This section is organized as follows. In Subsection \ref{ss:scaling_NS_2D}, we first show the global well-posedness of the 2D NSEs with cut-off, where the cut-off is used to tame the nonlinearity. Secondly, in Subsection \ref{ss:proof_anomalous_NS}, we prove Theorem \ref{t:anomalous_diss_NS} by removing the cut-off and arguing as in the case of passive scalars.

\subsection{Scaling limit with cut-off at a fixed Reynolds number}
\label{ss:scaling_NS_2D}
We begin by introducing the 2D NSEs in vorticity formulation with a cut-off. For $\phi\in C_{{\rm }}^{\infty}([0,\infty))$ such that  $\supp\,\phi\subseteq [0,2]$ and $\phi=1$ on $[0,1]$, and 
parameters $r\in (0,1)$ and $R>0$, consider
\begin{equation}
\label{eq:NS_2d_cutoff}
\left\{
\begin{aligned}
\partial_t \vorc^{\nu} &+\phi_{R,r}(\vorc^{\nu})\,\big[ (\K \vorc^{\nu}\cdot \nabla) \vorc^{\nu} \big]
=\nu \Delta \vorc^{\nu} \\
&\qquad \qquad\qquad + \sqrt{2\mu }\sum_{k\in \Z^2_0} \theta_k \,(\sigma_k \cdot\nabla) \vorc^{\nu} \circ \dot{W}_t^k & \text{ on }&\T^2,\\
\vorc^{\nu}(0,\cdot)&=\vor_0& \text{ on }&\T^2,
\end{aligned}
\right.
\end{equation}
where $\theta=(\theta_k)_{k\in \Z^2_0}\in \ell^2$ and
$$
\phi_{R,r}(\vorc^{\nu})\stackrel{{\rm def}}{=}\phi\big(R^{-1}\|\vorc^{\nu}\|_{H^{r}}\big).
$$
Other possible choices of the norm in the cut-off function are possible for the 2D NSEs in vorticity formulation \eqref{eq:NS_2d_vorticity}. However, the above cut-off has the highest smoothness that is compatible with Meyers' estimate in Theorem \ref{t:Meyers_parabolic}. Moreover, the choice of $H^r$ with $r>0$ will be crucial in Section \ref{s:2D_NSE_velocity} below, where we deal with 2D NSEs in velocity formulation, which are critical in $L^2$. 

With a slight abuse of notation, in this subsection, we do not display the dependence of $v^{\nu}$
on $(R,r,\g)$ as we will argue for a fixed value of such parameters.

\smallskip

In contrast to the previous section, to combine Meyers' estimates with the nonlinearity in the 2D NSEs in vorticity formulation (and looking ahead with the above-mentioned $L^2$-criticality of the 2D NSEs in velocity form), we need to consider solutions to \eqref{eq:NS_2d_cutoff} that are more regular than the one in Definition \ref{def:sol_NSE} depending on a parameter $p>2$ ruling the time integrability of solutions.

\begin{definition}[$p$-solutions -- 2D NSEs with cut-off]
\label{def:sol_NSE_cut_off}
Let $p\in [2,\infty)$. Fix $\nu>0$, $\theta^{\nu}\in \ell^2$ and 
$\vor_0\in B^{1-2/p}_{2,p}(\T^2)$. Let $\tau^{\nu}:\O\to [0,\infty]$ and $\vorc^{\nu}:[0,\tau^{\nu})\times \O\to H^1(\T^2) $ be a stopping time and a progressive measurable process, respectively.
\begin{itemize}
\item We say that $(\vorc^{\nu},\tau^{\nu})$ is a \emph{local $p$-solution} to \eqref{eq:NS_2d_cutoff} if there exists a sequence of stopping times $(\tau^\nu_{n})_{n\geq 1}$ such that $\tau^\nu_n\leq \tau^\nu$ a.s.\ for all $n\geq 1$ and $\lim_{n\to \infty} \tau^\nu_n=\tau^\nu$ a.s., for which the following conditions hold for all $n\geq 1$:
\begin{itemize}
\item $\vorc^{\nu}\in L^p(0,\tau^\nu_n;H^1(\T^2))\cap C([0,\tau^\nu_n];B^{1-2/p}_{2,p}(\T^2))$ a.s.; 
\item $\K\vorc^{\nu} \vorc^{\nu}\in L^p(0,\tau^\nu_n;L^2(\T^2;\R^{2}))$ a.s.;
\item a.s.\ for all $t\in [0,\tau^\nu_n]$ it holds that 
\begin{align*}
\vorc^{\nu}(t)-\vor_0
&= (\nu+\mu)\int_0^t \Big( \Delta \vorc^{\nu}(s) - \phi_{R,r}(\vorc^{\nu}(s))\big[  \nabla\cdot(\K \vorc^{\nu}(s) \vorc^{\nu}(s))\big]\Big)\,\dd s\\
&
+ \sqrt{2\mu}\int_0^t \one_{[0,\tau^{\nu}_n]}\Big((\theta^{\nu}_k\,\sigma_{k}\cdot\nabla)\vorc^{\nu}(s)\Big)_ {k}\,\dd 
\mathcal{W}_{\ell^2}.
\end{align*}
\end{itemize}
\item A local $p$-solution $(\vorc^{\nu},\tau^{\nu})$ to \eqref{eq:NS_2d_cutoff} is said to be a \emph{unique maximal (local) $p$-solution} to \eqref{eq:NS_2d_cutoff} if for any other local $p$-solution $(\vorcc^{\nu},\lambda^{\nu})$ we have $\lambda^{\nu}\leq \tau^{\nu}$ a.s.\ and $\vorcc^{\nu}=\vorc^{\nu}$ a.e.\ on $[0,\lambda^{\nu})\times \O$. 
\item A unique maximal $p$-solution $(\vorc^{\nu},\tau^{\nu})$ to \eqref{eq:NS_2d_cutoff} is said to be a \emph{unique global $p$-solution} to \eqref{eq:NS_2d_cutoff} if $\tau^{\nu}=\infty$ a.s.\ 
\end{itemize}
\end{definition}

In the above, $\mathcal{W}_{\ell^2}$ is as in \eqref{eq:def_cylindrical_noise}. For the optimality of the regularity of the initial data, the reader is referred to the comments below Theorem \ref{t:Meyers_parabolic}.
Similar to Definition \ref{def:sol_NSE}, if $p=2$, then we simply write `solution' instead of `$p$-solution'.

The aim of this subsection is to prove the following result.

\begin{proposition}[Scaling limit -- 2D NSEs with cut-off]
\label{prop:scaling_NS}
Fix $\nu\in (0,1)$ and $\mu>0$. 
Let $(\theta^n)_{n\geq 1}\subseteq \ell^2$ be a sequence of normalized radially symmetric coefficients (i.e., satisfying \eqref{eq:theta_normalized_symmetric}) such that
$$
\lim_{n\to \infty} \|\theta^n\|_{\ell^{\infty}}=0.
$$
Then there exists $p_0(\nu,\mu)>2$ for which the following assertion holds. 
If for $p\in (2,p_0]$, $r\in (0,1-2/p)$ and $R>0$ we have  
\begin{enumerate}[{\rm(1)}]
\item $\vor_0\in B^{1-2/p}_{2,p}$;
\item\label{it:scaling_NS_2} there exists a unique solution $\vordd^{\nu}\in C([0,1];H^r)\cap L^2(0,1;H^1)$ on $[0,1]$ to 
\begin{equation}
\label{eq:vdet_proof}
\left\{
\begin{aligned}
\partial_t \vordd^{\nu} +\phi_{R,r}(\vordd^\nu)\,[(\K\vordd^{\nu}\cdot\nabla)\vordd^{\nu}]
&= (\nu+\mu)\Delta \vordd^{\nu}& \text{ on }&\T^2,\\
\vordd^{\nu}(0,\cdot)&=\vor_0& \text{ on }&\T^2;
\end{aligned}
\right.
\end{equation}
\end{enumerate}
then for all $r_0\in (0,1-2/p)$ and all sequences $(\vor_{0,n})\subseteq B^{1-2/p}_{2,p}$ such that $\vor_{0,n}\to \vor_0$ in $B^{1-2/p}_{2,p}$ it holds that 
\begin{equation}
\label{eq:un_scaling_NS}
\lim_{n\to \infty} \P\Big(\sup_{t\in [0,1]}\|\vorcn^{\nu}(t)-\vordd^{\nu}(t)\|_{H^{r_0}}\geq \varepsilon \Big)=0\ \text{ for all } \varepsilon>0,
\end{equation}
where $\vorcn^{\nu}$ is the unique global $p$-solution to \eqref{eq:NS_2d_cutoff} with $\theta=\theta^n$ and initial data $\vor_{0,n}$.
\end{proposition}

As for Proposition \ref{prop:scaling_diff}, the key point is that the Sobolev space over $\T^2$ in which the convergence in \eqref{eq:un_scaling_NS} takes place has positive smoothness. Moreover, repeating the argument in Remark \ref{r:lack_convergence_gradient}, one sees that $\nabla\vorcn^\nu\not\to\nabla\vordd^{\nu}$ as $n\to \infty$ in probability in $L^2(0,1;L^2)$ in case $\vor_0\not\equiv 0$.  
Moreover, in contrast to the linear case, the uniqueness of (weak) solutions to \eqref{eq:vdet_proof} is not immediate, and therefore we include it as an assumption.

The existence of a unique global $p$-solution to \eqref{eq:NS_2d_cutoff} is proven in Lemma \ref{prop:uniform_theta_estimate_NSE} below. 
(Local) solutions to \eqref{eq:vdet_proof} on $[0,1]$ are as in Definition \ref{def:sol_NSE_cut_off} with trivial noise, $p=2$ and life-time $\geq 1$. Note that the condition $\vord^{\nu}\in C([0,1];H^r)$ is not a-priori given in the requirements for being local solutions, this is why it has been displayed in \eqref{it:scaling_NS_2}. 
Similarly, $\vord^{\nu}\in L^2(0,1;H^1)$ is needed as life-time $\geq 1$ does not imply that $t\mapsto \|\vord^\nu(t)\|_{L^2}^2$ is integrable near $t=1$. Let us emphasize that uniqueness for \eqref{eq:vdet_proof} is only required among all maps on $[0,1]$ which belong to $C([0,1];H^r)\cap L^2(0,1;H^1)$.

\smallskip

The proof of Proposition \ref{prop:scaling_NS} is given at the end of this subsection. 
Following the argument of Subsection \ref{ss:scaling_fixed_diffusivity}, the key ingredients are the following estimates, which are uniform in the class of normalized radially symmetric $\theta^{\nu}$.

\begin{lemma}[Uniform in $\theta$-estimate -- 2D NSEs with cut-off]
\label{prop:uniform_theta_estimate_NSE}
Let $\nu,r\in (0,1)$ and $\mu,R>0$ be fixed. Then there exist $p_0(\nu,\mu)>2$ and $C_0(\nu,r,\mu,R)>0$ for which the following assertion holds. For all $p\in [2,p_0]$, $\vor_0\in B^{1-2/p}_{2,p}$, $r\in (0,1-2/p)$ and all normalized radially symmetric $\theta^{\nu}\in \ell^2$ (i.e., satisfying \eqref{eq:theta_normalized_symmetric}),
there exists a unique global $p$-solution $\vorc^{\nu}$ to \eqref{eq:NS_2d_cutoff} and
\begin{equation}
\label{eq:estimate_uniform_theta_NS}
\E \sup_{t\in [0,1]}\|\vorc^{\nu}(t)\|_{B^{1-2/p}_{2,p}}^p + \E\int_{0}^1 \|\vorc^{\nu}(t)\|_{H^1}^p\,\dd t 
\leq C_0(1+\|\vor_0\|_{B^{1-2/p}_{2,p}}^p).
\end{equation}
\end{lemma}

As in the proof of Lemma \ref{l:uniform_estimates_diffusive}, the key point is the independence of  $(p_0,C_0)$ on $\theta$. 
As in the latter result, the key tools are the stochastic Meyers' estimates of Theorem \ref{t:Meyers_parabolic}. To handle the nonlinearity, we use the sub-criticality of $H^{r}$ for all $r>0$ in the case of 2D NSEs, both in vorticity and velocity form (see Section \ref{s:2D_NSE_velocity} below for the latter). The sub-criticality of $H^r$ is exploited via the following inequality (cf.\ \cite[Lemmas 4.3 and 4.5]{A23}): For all $r<\frac{1}{2}$,
\begin{align}
\label{eq:sub_criticality_L2}
\big\|\nabla\cdot([\K\vor]\, \vor)\big\|_{H^{-1}}
\lesssim \|\vor\|_{L^4}^2
\lesssim\|\vor\|_{H^{1/2}}^2
\lesssim \|\vor\|_{H^{r}}^{1+\kappa_r}\|\vor\|_{H^1}^{1-\kappa_r}
\end{align}
where $\kappa_r=\frac{r}{1-r}>0$. For later use (see Section \ref{s:2D_NSE_velocity} below), let us note that the \eqref{eq:sub_criticality_L2} also holds if $[\K\vor] \vor$ is replaced by $v\otimes v$ with $v\in H^1$.

The above follows from the Sobolev embedding $H^{1/2}(\T^2)\embed L^4(\T^2)$ as well as standard interpolation inequality. The sub-criticality of $H^r$ is encoded in the power $1-\kappa_r<1$ for the $H^1$-norm on the RHS\eqref{eq:sub_criticality_L2}. This fact will allow us to absorb the corresponding term via Young's inequality.

\begin{proof}[Proof of Lemma \ref{prop:uniform_theta_estimate_NSE}]
As we argue at a fixed viscosity $\nu>0$, we remove it from the notation and we simply write $\vorc$ instead of $\vorc^{\nu}$ and similar.
Let us begin by collecting some useful facts. 
Arguing as in the proof of Lemma \ref{l:uniform_estimates_diffusive}, the assumptions of the stochastic Meyers' inequalities of Theorem \ref{t:Meyers_parabolic} holds with constant independent of the choice of normalized radially symmetric $\theta\in \ell^2$, and we let $p_0=p_0(\mu,\nu)>2$ be the corresponding integrability parameter for which the estimate of Theorem \ref{t:Meyers_parabolic} holds for all $p\in [2,p_0)$. 
Below $p\in (2,p_0]$ is fixed. 

The proof is now split into three steps. We first prove the existence of a unique maximal $p$-solution $(\vorc,\tau)$ with a corresponding blow-up criterion. Secondly, we prove that a localized (in time) version of \eqref{eq:estimate_uniform_theta_NS} holds for such a maximal solution. Finally, in the third step, we prove that $\tau=\infty$ a.s.\ (i.e., $(\vorc,\tau)$ is global) and that \eqref{eq:estimate_uniform_theta_NS} holds.

\smallskip

\emph{Step 1: There exists a unique maximal (unique) $p$-solution $(\vorc,\tau)$ to \eqref{eq:NS_2d_cutoff}, and for all $T\in (0,\infty)$, the following blow-up criterion holds}
\begin{equation}
\label{eq:blow_up_truncated_NS}
\P\Big(\tau<T,\, \sup_{t\in [0,\tau)}\|\vorc(t)\|_{B^{1-2/p}_{2,p}}^p+\int_0^t \|\vorc(s)\|_{H^1}^p\,\dd s <\infty\Big)=0.
\end{equation}
To prove the existence of a unique maximal (local) $p$-solution to \eqref{eq:NS_2d_cutoff}, we employ \cite[Theorem 4.8]{AV19_QSEE_1} (or \cite[Theorem 4.7]{agresti2025nonlinear}). To this end, we rewrite \eqref{eq:NS_2d_cutoff} as a stochastic evolution equation on $X_0\stackrel{{\rm def}}{=}H^{-1}$:
\begin{equation}
\label{eq:SEE}
\dd \vorc + A \vorc \, \dd t = F(\vorc)\,\dd t + B \vorc\,\dd \mathcal{W}_{\ell^2}, \qquad \vorc(0)=\vor_0
\end{equation}
where $\mathcal{W}_{\ell^2}$ is as in \eqref{eq:def_cylindrical_noise}, and for $\vorc\in X_1\stackrel{{\rm def}}{=}H^{1}$,
\begin{align*}
A \vorc = -(\nu+\mu)\Delta \vorc & , \qquad  F(\vorc)=\phi_{R,r}(\vorc)\, \nabla\cdot ([\K \vorc]\, \vorc), \\ 
 B \vorc&=((\theta_k\sigma_{k,\alpha}\cdot \nabla )\vorc )_{k,\alpha}.
\end{align*}
Comparing Definition \ref{def:sol_NSE_cut_off} with \cite[Definition 4.4]{AV19_QSEE_1}, one readily see that unique maximal $p$-solution to \eqref{eq:NS_2d_cutoff} are equivalent to maximal $L^p_0$-solution to \eqref{eq:SEE} (see also \cite[Remark 5.6]{AV19_QSEE_2}). Here, we are implicitly using the reformulation \eqref{eq:equivalence_complex_real_noise} of the noise in terms of real-valued Brownian motions. Set 
\begin{equation}
\label{eq:definition_Xbeta_proof_cut_off}
X_{\beta}\stackrel{{\rm def}}{=} [X_0,X_1]_{\beta}= H^{-1+2\beta}\  \text{ for } \ \beta\in (0,1),
\end{equation} 
where $[\cdot,\cdot]_{\beta}$ denotes the complex interpolation functor, see e.g., \cite[Chapter 4]{BeLo}. By \cite[Theorem 4.8]{AV19_QSEE_1} or \cite[Theorem 4.7]{agresti2025nonlinear}, to show the existence of a unique maximal $p$-solution, it suffices to show the existence of $1-1/p<\beta\leq \varphi < 1$ such that $\varphi+\beta<2-1/p$, and 
for all $\vor,\vor'\in X_1$,
\begin{align}
\label{eq:F_estimate_NS2D}
\|F(\vor)-F(\vor')\|_{X_0}
&\lesssim (\|\vor\|_{X_{3/4}}+\|\vor'\|_{X_{3/4}})\|\vor-\vor'\|_{X_{3/4}}\\
\nonumber
&+(\|\vor\|_{X_{\varphi}}+\|\vor'\|_{X_{\varphi}})\|\vor-\vor'\|_{X_{\beta}}.
\end{align}
Next, we prove \eqref{eq:F_estimate_NS2D}. Without loss of generality, we assume that $\|\vor'\|_{H^r}\leq \|\vor\|_{H^r}$. Under the latter condition, we write $F(\vor)-F(\vor')= I_{\vor,\vor'}+ J_{\vor,\vor'}$ where 
\begin{equation}
\label{eq:definition_Jvorvorprime_Ivorvorprime}
\begin{aligned}
I_{\vor,\vor'}&=\phi_{R,r}(\vor)\, \nabla\cdot ([\K \vor]\, \vor)-\nabla \cdot ([\K \vor'] \vor'),\\
J_{\vor,\vor'}&= \nabla\cdot ([\K\vor'] \,\vor')\,\big( \phi_{R,r}(\vor)- \phi_{R,r}(\vor')\big).
\end{aligned}
\end{equation}
The first term on the RHS\eqref{eq:F_estimate_NS2D} clearly follows by estimating $I_{\vor,\vor'}$ and using $X_{3/4}=H^{1/2}(\T^2)\embed L^4(\T^2)$. Arguing as in \eqref{eq:sub_criticality_L2}, by interpolaton, for all $r<\frac{1}{2}$, we have
\begin{align}
\label{eq:interpolation_inequality_proof_p_solutions_cutoff}
\big\|\nabla\cdot([\K\vor']\,\vor')\big\|_{H^{-1}}
\lesssim \|\vor'\|_{H^{1/2}}^{2}
\lesssim \|\vor'\|_{H^{r}}\|\vor'\|_{H^{1-r}}.
\end{align}
We emphasize that the choice of the space $H^{1-r}$ for the interpolation in \eqref{eq:interpolation_inequality_proof_p_solutions_cutoff} is, to a certain extent, arbitrary. For our purposes, we only need the spaces appearing on the RHS\eqref{eq:interpolation_inequality_proof_p_solutions_cutoff} to be of lower order compared to $X_1=H^1$.

Now, coming back to \eqref{eq:F_estimate_NS2D}, as are assuming $\|\vor\|_{H^r}\geq \|\vor'\|_{H^{r}}$, it follows that  
\begin{equation}
\label{eq:estimate_phivorvorprime}
|\phi_{R,r}(\vor)-\phi_{R,r}(\vor')|=0 \ \ \text{ if }\ \ \|\vor'\|_{H^r}\geq 2R.
\end{equation}
In particular, 
\begin{align*}
\|J_{\vor,\vor'}\|_{H^{-1}}
&\stackrel{\eqref{eq:interpolation_inequality_proof_p_solutions_cutoff}}{\lesssim}
\|\vor'\|_{H^r}\|\vor'\|_{H^{1-r}}\,
|\phi_{R,r}(\vor)-\phi_{R,r}(\vor')|\\
&\stackrel{\eqref{eq:estimate_phivorvorprime}}{\lesssim_R}
\|\vor'\|_{H^{1-r}}\,
|\phi_{R,r}(\vor)-\phi_{R,r}(\vor')| \\
&\ \lesssim_R\|\vor'\|_{H^{1-r}}\,
\|\vor-\vor'\|_{H^r}.
\end{align*}
Thus, it follows from \eqref{eq:definition_Xbeta_proof_cut_off} that, for all $\beta>1-\frac{1}{p}$ and $\varphi\in [(1-\frac{r}{2})\vee \beta,1)$, 
\begin{equation}
\label{eq:Jvorvorprime_estimate}
\|J_{\vor,\vor'}\|_{H^{-1}}\lesssim \|\vor'\|_{X_{\varphi}}\,\|\vor-\vor'\|_{X_{\beta}}.
\end{equation}
Hence, the condition $\beta+\varphi<2-\frac{1}{p}$ follows by choosing $\beta=1-\frac{1}{p}+\varepsilon$ and $\varphi=1-2\varepsilon$ with $\varepsilon$ sufficiently small depending only on $r$ and $p$. Hence, \eqref{eq:F_estimate_NS2D} follows from \eqref{eq:Jvorvorprime_estimate} and the comments below \eqref{eq:definition_Jvorvorprime_Ivorvorprime}. 

Finally, \eqref{eq:blow_up_truncated_NS} follows from \cite[Theorem 4.10(3)]{AV19_QSEE_2} applied to \eqref{eq:SEE}.

\smallskip

\emph{Step 2: Proof of $\tau=\infty$.}
We begin with a localization argument. Fix $T<\infty$. For all $m\geq 1$, let $\tau_m$ be the stopping time given by
$$
\tau_m\stackrel{{\rm def}}{=}\inf\Big\{t\in [0,\tau\wedge T)\,:\, \|\vorc(t)\|^p_{B^{1-2/p}_{2,p}}+ \int_0^t \|\vorc(s)\|_{H^1}^p\,\dd s \geq m\Big\}
$$ 
where $\inf\emptyset\stackrel{{\rm def}}{=}\tau\wedge T$.
The Meyers' estimates of Theorem \ref{t:Meyers_parabolic} and a standard localization argument (see e.g., \cite[Proposition 3.12(b)]{AV19_QSEE_1}) ensure that existence of $C_1(\nu,\mu)>0$ such that, for all $m\geq 1$, $p\in [2,p_0]$ and $u_0\in B^{1-2/p}_{2,p}$,  
\begin{align*}
\E\sup_{t\in [0,\tau_m]}\|\vorc(t)\|_{B^{1-2/p}_{2,p}}^p 
&+ \E \int_0^{\tau_m}\|\vorc(t)\|^p_{H^1}\,\dd t
\leq C_1 \|\vor_0\|_{B^{1-2/p}_{2,p}}^p \\
&+ C_1\E \int_0^{\tau_m} \phi_{R,r}(\vorc)\,\big\|\nabla\cdot ([\K \vorc] \,\vorc)\big\|_{H^{-1}}^p\, \dd t .
\end{align*}
Here $p_0$ is as at the beginning of the proof.

Now, the sub-critical estimate \eqref{eq:sub_criticality_L2} yields, for all $\vor\in H^1(\T^2)$, 
\begin{align*}
\phi_{R,r}(\vor)\,\big\|\nabla ([\K \vor]\, \vor)]\big\|_{H^{-1}} 
&\leq K_{r}
\phi_{R,r}(\vor) \|\vor\|_{H^{r}}^{1+\kappa_r} \|\vor\|_{H^1}^{1-\kappa_r}\\
&\leq K_{r}
(2R)^{1+\kappa_r}\|\vor\|_{H^{1}}^{1-\kappa_r}\\
&\leq 
\frac{1}{2C_1}\|\vor\|_{H^{1}} + K_{r,R,\nu,\mu}
\end{align*}
where we used the Young inequality and that $\phi_{R,r}(\cdot,\vor)=0$ if $\|\vor\|_{H^{r}}\geq 2R$.

Hence, combining the previous estimates, we obtain, for all $m\geq 1$,
$$
\E\sup_{t\in [0,\tau_m]}\|\vorc(t)\|_{B^{1-2/p}_{2,p}}^p 
+ \E \int_0^{\tau_m}\|\vorc(t)\|^p_{H^1}\,\dd t 
\leq 2C_1 \|\vor_0\|_{B^{1-2/p}_{2,p}}^p +K_{r,R,\nu,\mu}.
$$
Let $C_0\stackrel{{\rm def}}{=} (2C_1)\vee K_{r,R,\nu,\mu}$. Note that $C_0$ is independent of $m\geq 1$ as are $C_1$ and $K_{r,R,\nu,\mu}$. Thus, since $\lim_{m\to \infty} \tau_m=\tau\wedge T$ a.s., the above implies
\begin{equation}
\label{eq:estimate_uniform_theta_NS_stopped}
\E\sup_{t\in [0,\tau\wedge T)}\|\vorc(t)\|_{B^{1-2/p}_{2,p}}^p 
+ \E \int_0^{\tau\wedge T}\|\vorc(s)\|^p_{H^1}\,\dd s 
\leq C_0 (1+\|\vor_0\|_{B^{1-2/p}_{2,p}}^p) .
\end{equation}
Hence, $\|\vorc(t)\|_{B^{1-2/p}_{2,p}}^p 
+  \int_0^{\tau\wedge T}\|\vorc(s)\|^p_{H^1}\,\dd s<\infty$ a.s.\ on $\{\tau<T\}$, and therefore $\tau\geq T$ a.s.\ by \eqref{eq:blow_up_truncated_NS}. The arbitrariness of $T$ implies $\tau=\infty$ a.s., as desired.

\smallskip

\emph{Step 3: Proof of \eqref{eq:estimate_uniform_theta_NS}}. The estimate \eqref{eq:estimate_uniform_theta_NS} follows from \eqref{eq:estimate_uniform_theta_NS_stopped} and the fact that $\tau=\infty$ a.s.\ by Step 2. 
\end{proof}

\begin{proof}[Proof of Proposition \ref{prop:scaling_NS}]
The proof follows almost verbatim the one of Proposition \ref{prop:scaling_diff} as the content of Lemma \ref{l:time_regularity_estimate} also holds with $(\varrho^{\g},\g)$ replaced by $(\vorc^{\nu},\nu)$. 
\end{proof}

\subsection{Proof of Theorem \ref{t:anomalous_diss_NS}}
\label{ss:proof_anomalous_NS}
Parallel to the proof of Theorem \ref{t:anomalous_diss_diff}, the key ingredient in Theorem \ref{t:anomalous_diss_NS} is the following consequence of the above scaling limit.

\begin{corollary}
\label{cor:strong_convergence_theta}
Let $\delta,\nu,\varepsilon\in (0,1)$ and $N,\mu>0$ be fixed. Then there exists $\theta^{\nu}\in \ell^2$ such that 
$$
\|\theta^{\nu}\|_{\ell^2}=1\quad \text{ and }\quad \#\{k\in \Z^2_0\,:\, \theta^{\nu}_k \neq 0\}<\infty,
$$
and that for all $\|\vor_0\|_{H^{\delta}}\leq N$ the unique global solution $\vor^{\nu}$ to \eqref{eq:NS_2d_vorticity} satisfies  
$$
\P\Big(\sup_{t\in [0,1]}\|\vor^{\nu}-\vord^{\nu}\|_{L^2}\leq \varepsilon\Big)>1-\varepsilon
$$
where $\vord$ is the unique global solution to
\begin{equation}
\label{eq:deterministic_NS_proof}
\left\{
\begin{aligned}
\partial_t \vord^{\nu} +(\K\vord^{\nu} \cdot\nabla) \vord^{\nu}
&=  (\nu+\mu)\Delta \vord^{\nu}& \text{ on }&\T^2,\\
\vord^{\nu}(0,\cdot)&=\vor_0& \text{ on }&\T^2.
\end{aligned}
\right.
\end{equation}
\end{corollary}

For the proof of the above, we need the following result.

\begin{proposition}[$H^{\delta}$-estimates -- 2D NSEs in vorticity formulation]
\label{prop:Hdelta_estimates}
Let $\kappa>0$, $\delta\in (0,1]$ and let $\vor$ be unique global solution to the 2D NSEs in vorticity formulation on $\T^2$, i.e.,
\begin{equation}
\label{eq:2D_NS_appendix}
\partial_t \vor +(\K\vor\cdot\nabla) \vor=\kappa\Delta \vor, \ \ \  \vor(0)=\vor_0\in H^{\delta}(\T^2)\  \text{ with }\ \ \textstyle{\int_{\T^2}}\vor_0(x)\,\dd x =0.
\end{equation}
Then there exists a non-decreasing mapping $N:[0,\infty)\to [1,\infty)$ such that 
\begin{equation}
\label{eq:claimed_estimates_Hdelta}
\sup_{t\in [0,\infty)}\|\vor(t,\cdot) \|_{H^{\delta}}+
\|\vor\|_{L^2(\R_+;H^{1+\delta})} \leq N(\|\vor_0\|_{L^2})\,\|\vor_0\|_{H^{\delta}}.
\end{equation}
\end{proposition}

The above is a standard consequence of $L^2$-maximal regularity estimates on the space $H^{-1+\delta}$ and the subcriticality of the 2D NSEs in vorticity formulation. For brevity, we omit the details.
We are now ready to prove Corollary \ref{cor:strong_convergence_theta}. 

\begin{proof}[Proof of Corollary \ref{cor:strong_convergence_theta}]
Let $p_0(\nu,\mu)>2$ be as in Proposition \ref{prop:scaling_NS}. Fix $p\in (2,p_0)$ such that $1-2/p<\delta$ and $r,r_0$ such that $r<r_0<1-2/p$. Note that the choice of $p$ yields $\vor_0\in B^{1-2/p}_{2,p}$, cf.\ \eqref{eq:embedding_initial_data}. 
Next, by Proposition \ref{prop:Hdelta_estimates}, there exists $K_0=K_0(N,r_0)>0$ such that the unique global solution $\vord^\nu$ to \eqref{eq:deterministic_NS_proof} satisfies
\begin{equation}
\label{eq:a_priori_vd_proof}
\sup_{t\in [0,1]}\|\vord^{\nu}(t)\|_{H^{r_0}}\leq K_0.
\end{equation}

Next, to apply Proposition \ref{prop:scaling_NS} with $R=K_0+1$ and $(r,r_0)$ as above, we check that \eqref{eq:vdet_proof} has a unique solution on $[0,1]$ and it is given by $\vord^{\nu}$. To begin, as $\vord^{\nu}$ is a unique global solution to \eqref{eq:deterministic_NS_proof} satisfying the bound \eqref{eq:a_priori_vd_proof}, then it is also a local solution to \eqref{eq:vdet_proof} on $[0,1]$ with $\vord^{\nu}\in C([0,1];H^r)\cap L^2(0,1;H^1)$. It remains to discuss the uniqueness (see the comments below Proposition \ref{prop:scaling_NS} for the uniqueness concept used here). To this end, we mimic a stopping-time argument. Let $\vordd^{\nu,\prime}$ be another global solution to \eqref{eq:vdet_proof} with $\vordd^{\nu,\prime}\in C([0,1];H^r)$. Let
$$
e^\prime \stackrel{{\rm def}}{=}\inf\{t\in [0,1]\,:\, \|\vordd^{\nu,\prime}(t,\cdot)\|_{H^r}\geq R\}\quad \text{ where }\quad \inf\emptyset\stackrel{{\rm def}}{=}1.
$$
If $
e^\prime=1$, then $\vord^{\nu}=\vordd^{\nu,\prime}$ by uniqueness of solutions to \eqref{eq:deterministic_NS_proof}. If $
e^\prime<1$, then still using the uniqueness of solutions to \eqref{eq:deterministic_NS_proof} we get $\vord^{\nu}=\vord^{\nu,\prime}$ on $[0,e^\prime]$. As $\sup_{t\in [0,e^\prime]}\|\vord^{\nu}(t)\|_{H^r}\leq 
\sup_{t\in [0,1]}\|\vord^{\nu}(t)\|_{H^{r_0}}\leq K_0$ and $R>K_0$ by construction, it follows that 
$\|\vordd^{\nu,\prime} (e^\prime,\cdot)\|_{H^{r}}<R$  
which contradicts the definition of $e^\prime $ and $\vordd^{\nu,\prime} \in C([0,1];H^r)$. This proves $
e^\prime=1$ and hence uniqueness of \eqref{eq:vdet_proof} follows from the one of solutions to \eqref{eq:deterministic_NS_proof}.

Now, by Proposition \ref{prop:scaling_NS} applied with $R=K_0+1$, $(p,r,r_0)$ as above and $\theta^n$ as in \eqref{eq:choice_thetan}, there exists $\theta^{\nu}\in \ell^2$ with the required properties such that 
$$
\P\Big(\sup_{t\in [0,1]}\|\vorc^{\nu}-\vord^{\nu}\|_{H^{r}} \leq \varepsilon\Big)>1-\varepsilon
$$
where $\vorc^{\nu}$ is the unique global $p$-solution to \eqref{eq:NS_2d_cutoff} with the above choice of $(R,r)$.
Set 
$\displaystyle{
\O_0\stackrel{{\rm def}}{=}\big\{\sup_{t\in [0,1]}\|\vorc^{\nu}-\vord^{\nu}\|_{H^{r}}\leq \varepsilon\big\}.}$
Now, to conclude, it remains to prove that
\begin{equation}
\label{eq:v_u_nu_equivalent_claim}
\vor^{\nu}(t)=\vorc^{\nu}(t) \text{ for all $t\in [0,1]$ a.s.\ on }\O_0.
\end{equation}
To see \eqref{eq:v_u_nu_equivalent_claim}, define the following stopping time 
$$
\tau_0\stackrel{{\rm def}}{=}
\inf\Big\{t\in [0,1]\,:\, \sup_{t\in [0,1]}\|\vorc^{\nu}(t)\|_{H^r}\geq K_0+1\Big\}.
$$ 
and $\inf\emptyset \stackrel{{\rm def}}{=}1$. 
Note that, by \eqref{eq:a_priori_vd_proof} and $\varepsilon<1$, we have 
\begin{equation}
\label{eq:lambda_1_on_big_set}
\tau_0=1 \ \ \text{ on }\O_0.
\end{equation} 
Moreover, $\phi_{R,r}(\vorc^{\nu})=1$ on $[0,\tau_0]$. Hence, $(\vorc^{\nu},\tau_0)$ is a local solution to the 2D NSEs  \eqref{eq:NS_2d_vorticity}. By uniqueness and maximality of the global solution $\vor^{\nu}$ to \eqref{eq:NS_2d_vorticity} (see the comments below Definition \ref{def:sol_NSE}), we obtain 
\begin{equation}
\label{eq:uniqueness_uv_nu_2dNS}
\vor^{\nu}(t)=\vorc^{\nu}(t)  \ \text{ a.s.\ for all $t\in [0,\tau_0]$.}
\end{equation}
Hence, the claim \eqref{eq:v_u_nu_equivalent_claim} follows from 
\eqref{eq:lambda_1_on_big_set} and \eqref{eq:uniqueness_uv_nu_2dNS}.
\end{proof}

\begin{proof}[Proof of Theorem \ref{t:anomalous_diss_NS}]
Due to Corollary \ref{cor:strong_convergence_theta}, the proof of Theorem \ref{t:anomalous_diss_NS} now follows almost verbatim from that of Theorem \ref{t:anomalous_diss_diff} given at the end of Subsection \ref{ss:proofs_diffusive}. 
\end{proof}

Similar to Proposition \ref{prop:anomalous_diss_diff2}, from Corollary \ref{cor:strong_convergence_theta} one can deduce the following variant of Theorem \ref{t:anomalous_diss_NS}.

\begin{proposition}[Anomalous dissipation of enstrophy by transport noise -- 2D NSEs II]
\label{prop:anomalous_diss_NS2}
Let $N\geq 1$ and $\delta>0$. For all $\mu>0$ and $\eta\in (0,1-e^{-\mu/(4\pi^2)})$, there exists a family $(\theta^{\nu})_{\nu\in (0,1)}\subseteq \Set$ such that, for all mean-zero $\vor_0\in H^{\delta}(\T^2)$ satisfying $N^{-1}\leq \|\vor_0\|_{L^2}$ and $ \|\vor_0\|_{H^{\delta}}\leq N$, it holds that 
$$
\inf_{\nu\in ( 0,1)} \E\int_0^1 \int_{\T^d}  \nu \,|\nabla \vor^{\nu}|^2\,\dd x \,\dd t \geq \frac{\eta}{2}\,\|\vor_0\|_{L^2(\T^d)}^2,
$$
where $\vor^\nu$ is the unique global smooth solution to \eqref{eq:NS_2d_vorticity}.
\end{proposition}

\section{Anomalous dissipation for 2D NSEs in velocity formulation}
\label{s:2D_NSE_velocity}
In this section, we study the anomalous dissipation of energy of the 2D NSEs with transport noise and in velocity formulation:
\begin{equation}
\label{eq:NS_2d_velocity}
\left\{
\begin{aligned}
\partial_t u^{\nu} +(u^{\nu}\cdot\nabla) u^{\nu} 
&=-\nabla p^{\nu} +\nu \Delta u^{\nu} \\
&+ \sqrt{2\mu }\sum_{k\in \Z^2_0}  \big[-\nabla \wt{p}_n^{\,\nu} +\theta_k^{\nu}	\,
(\sigma_k \cdot\nabla) u^{\nu} \big]\circ \dot{W}_t^k & \text{on }&\T^2,\\
\nabla \cdot u^{\nu}&=0& \text{on }&\T^2,
\\
u^{\nu}(0,\cdot)&=u_0& \text{on }&\T^2.
\end{aligned}
\right.
\end{equation}
where $(W^k)_{k\in \Z^2_0}$ is are complex-valued Brownian motions as in Subsection \ref{ss:structure_noise}. 
In Remark \ref{r:3D_NSE} below, we also discuss the three-dimensional case.
Let us point out that taking the curl operator $(u_x,u_y)\mapsto \partial_xu_y-\partial_y u_x$ in \eqref{eq:NS_2d_velocity} does \emph{not} give the vorticity formulation \eqref{eq:NS_2d_vorticity} (see \cite[Remark 2.1]{L23_enhanced} for details). Therefore, the results in this section are not related to Theorem \ref{t:anomalous_diss_NS} besides the proof strategy.  
The stochastic NSEs \eqref{eq:NS_2d_velocity} can be derived by mimicking the two-scale decomposition as in Subsection \ref{ss:heuristics_viscosity_dependent}. 
The reader is referred to \cite{DP22_two_scale,MR01,MR04} for more rigorous physical derivations of \eqref{eq:NS_2d_velocity}.

As usual, to analyse \eqref{eq:NS_2d_velocity}, we employ the Helmholtz projection $\p$ and its complement projection $\q$. For an $\R^d$-valued distribution $f=(f^i)_{i=1}^{d}\in \D'(\T^d;\R^d)$ on $\T^d$, let $\widehat{\,f^i\,}(k) =\langle e_k ,f^i\rangle $ be $k$-th Fourier coefficients, where $k\in \Z^d$, $i\in \{1,\dots,d\}$ and $e_k(x)=e^{2\pi \i k\cdot x}$. The Helmholtz projection $\p f$ for $f\in \D'(\T^d;\R^d)$ is given by
$$
(\widehat{ \p f})^i (k)\stackrel{{\rm def}}{=}\widehat{\,f^i\,}(k)- \sum_{1\leq j\leq d} \frac{k_i k_j }{|k|^2} \widehat{\,f^j\,}(k), \qquad 
(\widehat{ \p f})^i (0)\stackrel{{\rm def}}{=}\widehat{\,f^i\,}(0).
$$
Formally, $\p f$ can be written as $f- \nabla \Delta^{-1}(\nabla \cdot f)$. We set 
$
\q\stackrel{{\rm def}}{=} \mathrm{Id}- \p.
$  
From standard Fourier analysis, it follows that $\q$ and $\p$ restrict to bounded linear operators on $H^{s,q}$ and $B^s_{q,p}$ for $s\in \R$ and $q,p\in (1,\infty)$.
Finally, we can introduce function spaces of divergence-free vector fields: For $\mathcal{A}\in \{L^q,H^{s,q},B^{s}_{q,p}\}$, we set
\begin{align*}
\|f\|_{\mathbb{A}(\T^d)}\stackrel{{\rm def}}{=} \p (\mathcal{A}(\T^d;\R^d)),
\end{align*}
endowed with the natural norm. $\ell^2$-valued function spaces are defined analogously.

By using the Helmholtz projection to the first in \eqref{eq:NS_2d_velocity}, the 2D NSEs \eqref{eq:NS_2d_velocity} are (formally) equivalent to
\begin{align}
\label{eq:NS_2d_velocity_projection}
\partial_t u^{\nu} +\p[(u^{\nu}\cdot\nabla) u^{\nu} ]
=\nu \Delta u^{\nu} + \sqrt{2\mu }\sum_{k\in \Z^2_0} \theta_k^{\nu}	\, \p\big[
(\sigma_k \cdot\nabla) u^{\nu} \big]\circ \dot{W}_t^k 
\end{align}
with initial data $u^{\nu}(0,\cdot)=u_0$, both on $\T^2$.
Following the arguments in Subsection \ref{ss:structure_noise}, we reformulate the Stratonovich noise in \eqref{eq:NS_2d_velocity_projection} as an It\^o term plus a correction term.
As in \cite[Section 2]{FL19}, by \eqref{eq:ellipticity_noise}, it holds that 
\begin{equation}
\begin{aligned}
\label{eq:Ito_stratonovich_stokes_change}
\sqrt{2 \ellip}\sum_{k\in \Z^2_0}\theta_k\,\p [ (\sigma_{k}\cdot \nabla) u]\circ \dot{W}_t^{k}
=
\ellip \Delta u+Q_{\theta} (u) +  \sqrt{2 \ellip}\sum_{k\in \Z^2_0}\theta_k \,\p[ (\sigma_{k}\cdot \nabla) u]\,  \dot{W}_t^{k}
\end{aligned}
\end{equation}
where 
\begin{equation}
\label{eq:Ito_correction_Q}
Q_{\theta} (u)\stackrel{{\rm def}}{=}-
2\mu \sum_{k\in \Z^2_0}\theta_k^2\, \p \big[(\sigma_{-k}\cdot \nabla) \q [(\sigma_k \cdot \nabla) u]\big].
\end{equation}
Further comments on the transformation \eqref{eq:Ito_stratonovich_stokes_change} can be found in \cite[Section 1]{AV21_NS}. Arguing as in \eqref{eq:equivalence_complex_real_noise}, we can reformulate the noise terms in \eqref{eq:Ito_stratonovich_stokes_change} and \eqref{eq:Ito_correction_Q} using the real-valued coefficients $\Re \sigma_k$ and $\Im \sigma_k$. 
Finally, solutions to \eqref{eq:NS_2d_velocity} in the form \eqref{eq:NS_2d_velocity_projection}-\eqref{eq:Ito_correction_Q} can be defined as in Definition \ref{def:sol_NSE} by interpreting the convective term $(u^\nu\cdot\nabla)u^\nu$ in the conservative form $\nabla \cdot (u^\nu\otimes u^\nu)$, and replacing the usual Lebesgue and Sobolev spaces by the Lebesgue and Sobolev spaces of divergence-free vector fields, i.e., $\Hs^1(\T^2)$ and $\Ls^2(\T^2)$, respectively. We omit the details for brevity.
For completeness, we mention that the existence of a unique global solution to \eqref{eq:NS_2d_velocity} with $u_0\in \Ls^2(\T^2)$ with finite energy (i.e., $u\in L^2(0,T;\Hs^1(\T^2))\cap C([0,T];\Ls^2(\T^2))$ for all $T<\infty$) is well-known and can be, for instance, obtained via either \cite[Theorem 3.4]{AV24_variational} or \cite[Theorem 7.10]{agresti2025nonlinear}.

\smallskip

The main result of this section reads as follows.

\begin{theorem}[Anomalous dissipation of energy by transport noise -- 2D NSEs]
\label{t:anomalous_diss_NS_velocity}
Let $N\geq 1$ and $\delta>0$ be fixed. Then the following assertions hold.
\begin{itemize}
\item
For all $\mu>0$, there exists a family $(\theta^{\nu})_{\nu\in (0,1)}\subseteq \Set$ such that, for all mean-zero $u_0\in \Hs^{\delta}(\T^2)$ satisfying $ \|u_0\|_{H^{\delta}}\leq N$, 
we have
$$
\liminf_{\nu\downarrow 0} \E\int_0^1 \int_{\T^2}  \nu \,|\nabla u^{\nu}|^2\,\dd x \,\dd t \geq \frac{1}{2} (1-e^{-\mu/(16\pi^2)})\, \|\vor_0\|_{L^2(\T^2)}^2.
$$
\item For all $\mu>0$ and $\eta\in (0,1-e^{-\mu/(16\pi^2)})$, there exists a family $(\theta^{\nu})_{\nu\in (0,1)}\subseteq \Set$ such that, for all mean-zero $u_0\in \Hs^{\delta}(\T^2)$ satisfying $N^{-1}\leq \|u_0\|_{L^2}$ and $ \|u_0\|_{H^{\delta}}\leq N$, it holds that 
$$
\inf_{\nu\in ( 0,1)} \E\int_0^1 \int_{\T^d}  \nu \,|\nabla u^{\nu}|^2\,\dd x \,\dd t \geq \frac{\eta}{2}\,\|u_0\|_{L^2(\T^d)}^2.
$$
\end{itemize}
In the above statements, $u^\nu$ is the unique global smooth solution to \eqref{eq:NS_2d_velocity}.
\end{theorem}

The appearence of the factor $\mu/(16\pi^2)$ in Theorem \ref{t:anomalous_diss_NS_velocity}, rather than $\mu/(4\pi^2)$ as in Theorem \ref{t:anomalous_diss_NS}, is due to the asymptotic behaviour of the It\^o-Stratonovich corrector $Q_\theta$ under the scaling limit which will be carried out in Proposition \ref{prop:scaling_NS_velocity} below.

Before going into the proof of Theorem \ref{t:anomalous_diss_NS_velocity}, let us discuss its physical interpretation. 
In the context of the KB theory of 2D turbulence \cite{B69_2D_spectrum,R67_2D_turbulence,Frisch95}, anomalous dissipation of energy does \emph{not} hold. Thus, the result of Theorem \ref{t:anomalous_diss_NS_velocity} is not physically relevant.
Let us emphasize that this does not conflict with the physical interests of Theorem \ref{t:anomalous_diss_NS} as the latter is concerned with the \emph{stochastic} vorticity formulation, which cannot be obtained from the velocity one, see the comments below \eqref{eq:NS_2d_velocity}. 
Arguing as in Subsection \ref{ss:heuristics_viscosity_dependent}, we conjecture that this is due to the incorrect scaling for $\supp\,\theta^{\nu}$ in Theorem \ref{t:anomalous_diss_NS_velocity}. Physically speaking, following the two-scale terminology of Subsection \ref{ss:heuristics_viscosity_dependent}, the former corresponds to an overestimation of the `excitement' of the small scales. 
Thus, we conjecture that $\lim_{\nu\downarrow 0} \nu^{1/2-\varepsilon} N^{\nu}=\infty$ for all $\varepsilon>0$ where $N^{\nu}$ satisfies $\supp\, \theta^\nu \subseteq \{k\in \Z^2_0\,:\, |k|\eqsim N^{\nu}\}$ and $(\theta^\nu)_{\nu\in (0,1)}$ is such that anomalous dissipation of energy for \eqref{eq:NS_2d_velocity} holds.

\begin{remark}[Anomalous dissipation for the 3D NSEs]
\label{r:3D_NSE}
A version of Theorem \ref{t:anomalous_diss_NS} also holds for 3D NSEs in velocity formulation, provided, at a \emph{fixed} viscosity $\nu>0$, well-posedness and a-priori bounds for the 3D NSEs hold in the sub-critical space $\Ls^r(\T^3)$ with $r>3$. We refrain from formulating the latter assumption, as it implies the regularization by noise for 3D NSEs, which is one of the major open problems in stochastic fluid dynamics. The reader is referred to \cite{A24_hyper} for a recent result.
\end{remark}

As for the 2D NSE in vorticity formulation \eqref{eq:NS_2d_vorticity}, the proof of Theorem \ref{t:anomalous_diss_NS_velocity} is analogous to that of Theorem \ref{t:anomalous_diss_NS} and Proposition \ref{prop:anomalous_diss_diff2}. As for the vorticity formulation \eqref{eq:NS_2d_vorticity} analyzed in Section \ref{s:anomalous_NS}, the key ingredients are the scaling limit as in Proposition \ref{prop:scaling_NS} and the $H^{\delta}$-estimate of Proposition \ref{prop:Hdelta_estimates}.

To extend the Proposition \ref{prop:scaling_NS} to the stochastic 2D NSEs in velocity form \eqref{eq:NS_2d_velocity}, we consider the following truncated SPDE:
\begin{equation}
\label{eq:NS_2d_cutoff_velocity}
\left\{
\begin{aligned}
\partial_t \uc^{\nu} &+\phi_{R,r}(\uc^{\nu})\,[ ( \uc^{\nu}\cdot \nabla) \uc^{\nu} ]
=\nu \Delta \uc^{\nu} \\
&\qquad \qquad\qquad + \sqrt{2\mu }\sum_{k\in \Z^2_0} \theta_k \,(\sigma_k \cdot\nabla) \uc^{\nu} \circ \dot{W}_t^k & \text{ on }&\T^2,\\
\uc^{\nu}(0,\cdot)&=u_0& \text{ on }&\T^2,
\end{aligned}
\right.
\end{equation}
where $\theta=(\theta_k)_{k\in \Z^2_0}\in \ell^2$,
$
\phi_{R,r}(\uc^{\nu})\stackrel{{\rm def}}{=}\phi\big(R^{-1}\|\uc^{\nu}\|_{H^{r}}\big)
$ and $\phi\in C_{{\rm }}^{\infty}([0,\infty))$ is a cutoff function such that  $\supp\,\phi\subseteq [0,2]$ and $\phi=1$ on $[0,1]$.

\smallskip

Next, we prove an analogue of Proposition \ref{prop:scaling_NS} for the 2D NSEs with transport noise in the velocity formulation \eqref{eq:NS_2d_velocity}.

\begin{proposition}[Scaling limit -- 2D NSEs in velocity form with cut-off]
\label{prop:scaling_NS_velocity}
Fix $\nu\in (0,1)$ and $\mu>0$. Let $(\theta^n)_{n\geq 1}\subseteq \Set$ be as in \eqref{eq:choice_thetan}. 
Then there exists $p_0(\nu,\mu)>2$ for which the following assertion holds. 
If for $p\in (2,p_0]$, $r\in (0,1-2/p)$ and $R>0$ we have  
\begin{enumerate}[{\rm(1)}]
\item $u_0\in \Bs^{1-2/p}_{2,p}$;
\item\label{it:scaling_NS_2} there exists a unique solution $u^{\nu}\in C([0,1];\Hs^r)\cap L^2(0,1;\Hs^1)$ on $[0,1]$ to 
\begin{equation}
\label{eq:vdet_proof_velocity}
\left\{
\begin{aligned}
\partial_t \udd^{\nu} +\phi_{R,r}(\udd^\nu)\,[\nabla \pdd^\nu +(\udd^{\nu}\cdot\nabla)\udd^{\nu}]
&= \Big(\nu+\frac{\mu}{4}\Big)\Delta \udd^{\nu}&\text{on }&\T^2,\\
\udd^{\nu}(0,\cdot)&=u_0& \text{on }&\T^2;
\end{aligned}
\right.
\end{equation}
\end{enumerate}
then for all $r_0\in (0,1-2/p)$ and all sequences $(u_{0,n})\subseteq \Bs^{1-2/p}_{2,p}$ such that $u_{0,n}\to u_0$ in $\Bs^{1-2/p}_{2,p}$ it holds that 
\begin{equation*}
\lim_{n\to \infty} \P\Big(\sup_{t\in [0,1]}\|\ucn^{\nu}(t)-\udd^{\nu}(t)\|_{H^{r_0}}\geq \varepsilon \Big)=0\ \text{ for all } \varepsilon>0,
\end{equation*}
where $\ucn^{\nu}$ is the unique global $p$-solution to \eqref{eq:NS_2d_cutoff_velocity} with $\theta=\theta^n$ (whose definition is analogous to the one in Definition \ref{def:sol_NSE_cut_off}) and initial data $u_{0,n}$.
\end{proposition}

There are two minor differences between Proposition \ref{prop:scaling_NS_velocity} and Proposition \ref{prop:scaling_NS}. First, the sequence chosen for the scaling limit, $(\theta^n)_{n\geq 1} \in \Set$, is more restricted. Second, the limiting deterministic PDE \eqref{eq:vdet_proof_velocity} features a reduced enhanced dissipation coefficient, namely $\mu/4$. Both differences stem from the asymptotic behavior of $Q_{\theta^n}$ as $n \to \infty$, cf.\ \eqref{eq:asymptotic_behavior_Q} below.

\begin{proof}
The proof of Proposition \ref{prop:scaling_NS_velocity} closely follows that of Proposition \ref{prop:scaling_NS}. We therefore comment only on the key modifications. Firstly, the nonlinear term $\p[\nabla \cdot (u^\nu \otimes u^\nu)]$ admits estimates analogous to those used in the 2D NSEs with transport noise in vorticity formulation \eqref{eq:NS_2d_vorticity}. For instance, parallel to \eqref{eq:sub_criticality_L2}, for all $u\in \Hs^1(\T^2)$, for all $r>0$, it holds that 
$$
\|\p[\nabla \cdot(u\otimes u)]\|_{H^{-1}}\lesssim \|u\|_{L^4}^2 \lesssim \|u\|_{H^{r}}^{1+\kappa_r}
\|u\|_{H^{1}}^{1-\kappa_r}
$$
where $\kappa_r=\frac{r}{1-r}>0$. Similarly, the estimate \eqref{eq:interpolation_inequality_proof_p_solutions_cutoff} also holds for the nonlinearity in the velocity form. Hence, from the previous facts, one can prove an analogue of Lemma \ref{prop:uniform_theta_estimate_NSE}, where in corresponding Step 2, one employs the Meyers estimates for the turbulent Stokes system of Theorem \ref{t:Meyers_Stokes} instead of Theorem \ref{t:Meyers_parabolic}.

Second, the factor $\frac{\mu}{4}$ follows by repeating Step 2 of Proposition \ref{prop:scaling_diff} and using that, for all divergence-free vector field $\phi\in C^{\infty}(\T^d;\R^d)$ and $(\theta_n)_{n\geq 1}$ as in \eqref{eq:choice_thetan} it holds that 
\begin{equation}
\label{eq:asymptotic_behavior_Q}
\lim_{n\to \infty} Q_{\theta^n} (\varphi)= \frac{3}{4}\mu\, \Delta \varphi \  \text{ in }\ H^{-\delta}(\T^d;\R^d) \ \text{ for all }\delta>0;
\end{equation}
see \cite[Theorem 3.1]{L23_enhanced}, and also \cite[Theorem 5.1]{FL19} for the three-dimensional case.

The remainder of the proof follows by repeating the arguments in Proposition \ref{prop:scaling_NS} with minimal modification.
\end{proof}

Finally, similar to Proposition \ref{prop:Hdelta_estimates}, the following result holds.

\begin{proposition}[$H^{\delta}$-estimates -- 2D NSEs velocity formulation]
\label{prop:Hdelta_estimates_velocity}
Let $\kappa>0$, $\delta\in (0,1]$ and let $u$ be unique global solution to the 2D NSEs in velocity form on $\T^2$, i.e.,
\begin{equation}
\label{eq:2D_NS_appendix}
\partial_t u +(u\cdot\nabla) u=-\nabla p+\kappa\Delta u,\quad \ \nabla \cdot u=0, \quad\  u(0)=u_0\in \Hs^{\delta}(\T^2).
\end{equation}
Then there exists a non-decreasing mapping $N:[0,\infty)\to [1,\infty)$ such that
\begin{equation}
\label{eq:claimed_estimates_Hdelta}
\sup_{t\in [0,\infty)}\|u(t,\cdot) \|_{H^{\delta}}+
\Big\|\,u-\int_{\T^2}u_0(x)\,\dd x\,\Big\|_{L^2(\R_+;H^{1+\delta})} \leq N(\|u_0\|_{L^2})\|u_0\|_{H^{\delta}}.
\end{equation}
\end{proposition}
 
 The above is well-known to experts and can be obtained by using an interpolation result, as in \cite{GP02}, and the fact that 
 $H^{\delta}_{\mf}=(L_{\mf}^2,H_\mf^1)_{\delta,2}$ where $(\cdot,\cdot)_{\delta,2}$ is the real interpolation functor (see e.g., \cite{BeLo} for details) and the subscript $\mf$ stands for mean-zero.
The case $\delta>1$ of Proposition \ref{prop:Hdelta_estimates_velocity} also holds and it follows from the case $\delta<1$ and the sub-criticality of $H^{\delta}$ with $\delta>0$ for \eqref{eq:2D_NS_appendix}. 

\begin{proof}[Proof of Theorem \ref{t:anomalous_diss_NS_velocity}]
By Propositions \ref{prop:scaling_NS_velocity} and \ref{prop:Hdelta_estimates_velocity}, the conclusion of Corollary \ref{cor:strong_convergence_theta} remains valid when $\vor^\nu,\vordd^\nu$ and $\vord^{\nu}$ are replaced by $u^\nu,\udd^\nu$ and $\ud^\nu$, respectively; where $\ud^\nu$ is the unique global solution to \eqref{eq:2D_NS_appendix} with $\kappa=\nu+\mu/4$.
Consequently, the assertions of Theorem \ref{t:anomalous_diss_NS_velocity} can be established using the same arguments as in the proofs of Theorem \ref{t:anomalous_diss_diff} and Proposition \ref{prop:anomalous_diss_diff2}.
\end{proof}

\appendix

\section{Stochastic Meyers' estimates}
\label{app:Meyers}
In this appendix, we prove maximal $L^p_{t}(L_x^q)$-regularity estimates for parabolic SPDEs under only boundedness, measurability, and parabolicity assumptions with $p$ and $q$ close to $2$. In the deterministic case, such estimates have been first proven by Meyers' \cite{Meyers63} via a perturbation argument. 
In the stochastic setting, such estimates have been proven in  
\cite{VB23_extrapolation} where, due to the abstract setting, no extrapolation in space is given, i.e., $q=2$.
Here, we provide direct proof of the stochastic Meyers' estimates following Meyers' original idea, which, interestingly, 
also provides an improvement in the integrability in space.
For simplicity, we mainly consider SPDEs on $\T^d$. However, extensions to domains and systems of parabolic SPDEs are possible. 

\smallskip

Below, $(W^n)_{n\geq 1}$ denotes a family of independent standard Brownian motions on a filtered probability space $(\O,\mathscr{A},(\F_t)_{t\geq0},\P)$ and $\E\stackrel{{\rm def}}{=}\int_{\O}\cdot\,\dd \P$. We emphasize that these Brownian motions may differ from those introduced in Subsection \ref{ss:structure_noise}.

\subsection{Parabolic stochastic Meyers' estimates}
In this subsection, we consider the following parabolic SPDEs
\begin{equation}
\label{eq:parabolic_linear}
\left\{
\begin{aligned}
\partial_t v& =\kappa \Delta  v+f  + \sum_{n\geq 1} \big[(\xi_n\cdot\nabla) v+g_n\big]\,\dot{W}_t^n 
 &\text{ on }&\T^d,\\
v(0,\cdot)&=0 &\text{ on }&\T^d,
\end{aligned}
\right.
\end{equation}
where $v$ is the unknown process and $f$, $\xi_n$ and $g_n$ are specified below. 
In the following, for a stopping time $\tau:\O\to [0,\infty]$, we say that a progressively measurable process $v:[0,\tau)\times \O\to H^{1}(\T^d)$ is a \emph{strong solution} to \eqref{eq:parabolic_linear} (on $(0,\tau)$) if a.s.\ $v\in L^2_{{\rm loc}}([0,\tau);H^1(\T^d))\cap C([0,\tau);L^2(\T^d))$, and for all $t\in [0,\tau)$,
$$
v(t)=\int_{0}^t (\kappa \Delta  v(s)+f(s))\,\dd s + \sum_{n\geq 1} \int_0^t \big[(\xi_n\cdot\nabla) v(s)+g_n(s)\big]\,\dd{W}_s^n .
$$
Here, the above equality is understood in $H^{-1}(\T^d)$. Similar definitions are applied to the other SPDEs considered in this section.

The following complements \cite[Theorem 1.2]{AV_torus} in case of $L^{\infty}$-transport noise.

\begin{theorem}[Parabolic stochastic Meyers' estimates]
\label{t:Meyers_parabolic}
Let $\kappa>0$. Assume that 
$$
\xi_n=(\xi_j^n)_{j=1}^d:\R_+\times \O\times \T^d\to \R^d \text{ is $\Progress\otimes \Borel(\T^d)$-measurable 
for all $n\geq 1$,}
$$ 
and that there exist $M_0>0$ and $\kappa_0\in [0,\kappa)$ such that, a.e.\ on $\R_+\times \O\times\T^d$,
\begin{align}
\label{eq:boundedness_linear}
\|(\xi_n)_{n\geq 1}\|_{\ell^2}&\leq M_0,& &(\text{boundedness})\\
\label{eq:parabolicity_linear}
\frac{1}{2}\sum_{n\geq 1}\big( \xi_n\cdot \eta\big)^2  
&\leq \kappa_0 |\eta|^2\  \text{ for all } \eta\in \R^d. & &(\text{parabolicity})
\end{align}
Then there exists $p_0(d,\kappa,\kappa_0,M_0)>2$ such that, for all stopping times $\tau$ with values in $[0,T]$ with $T<\infty$, $p\in [2,p_0]$, $q\in [2,p]$ and progressively measurable processes $f$ and $g$ satisfying
$$
f\in L^p((0,\tau)\times \O;H^{-1,q}(\T^d)) \quad \text{ and }\quad g=(g_n)_{n\geq 1} \in L^p((0,\tau)\times \O;L^{q}(\T^d;\ell^2)),
$$
there exists a unique strong solution $v\in L^p((0,\tau)\times \O;H^{1,q}(\T^d))$ to \eqref{eq:parabolic_linear} and
\begin{align}
\label{eq:maximal_regularity_estimate_meyers_parabolic}
\|v\|_{L^p((0,\tau)\times\O;H^{1,q}(\T^d))}
\lesssim_{d,\kappa,\kappa_0,M_0,T,p} &\,\|f\|_{L^p((0,\tau)\times \O;H^{-1,q}(\T^d))}\\
\nonumber
+ &\,\|g\|_{L^p((0,\tau)\times \O;L^{q}(\T^d;\ell^2))}.
\end{align}
\end{theorem}

Some remarks are in order. Firstly, it is well-known that the condition \eqref{eq:parabolicity_linear} with $\kappa_0<\kappa$ is optimal in the parabolic regime. Secondly, by \cite[Proposition 3.8]{AV19_QSEE_1}, \cite[Theorem 1.2]{MaximalLpregularity} and the fact that the operator $v\mapsto -\Delta v$ on $H^{-1,q}(\T^d)$ with domain $H^{1,q}(\T^d)$ has a bounded $H^{\infty}$-calculus of angle $0$ (due to the periodic version of \cite[Theorem 10.2.25]{Analysis2}), it follows that in case $\tau\equiv T$ the LHS\eqref{eq:maximal_regularity_estimate_meyers_parabolic} can be replaced by 
\begin{align}
\label{eq:time_regularity_estimates_1}
&\|v\|_{L^p(\O;C([0,T];B^{1-2/p}_{q,p}(\T^d)))},  \ \ \text{and}\\
\label{eq:time_regularity_estimates_2}
&\|v\|_{L^p( \O;H^{\theta,p}(0,T;H^{1-2\theta,q}(\T^d)))}  \text{ for $\theta\in [0,\tfrac{1}{2})$ provided $p>2$}.
\end{align}
We emphasize that \eqref{eq:time_regularity_estimates_1} also holds in the case $p=2$ (and hence $q=2$). In the latter case, \eqref{eq:maximal_regularity_estimate_meyers_parabolic}-\eqref{eq:time_regularity_estimates_1} coincide with the usual $L^2$-estimate for parabolic SPDEs as $L^2(\T^d)=B^0_{2,2}(\T^d)$, see e.g., \cite[Chapter 4]{LR15}.

The reader is referred to \cite[Chapter 10]{Analysis2} and \cite[Subsection 2.2]{AV19_QSEE_1} for details on the $H^{\infty}$-calculus and Banach-valued fractional Sobolev spaces, respectively. 
The appearance of Besov spaces in \eqref{eq:time_regularity_estimates_1} is optimal in light of the trace method, see e.g.\ \cite[Section 3.12]{BeLo} and \cite[Theorem 1.2]{ALV23_trace}. Finally, by \cite[Proposition 3.10]{AV19_QSEE_1}, the estimates \eqref{eq:maximal_regularity_estimate_meyers_parabolic}-\eqref{eq:time_regularity_estimates_2} also hold with non-trivial initial data $v(\cdot,0)\in L^p_{\F_0}(\O;B^{1-2/p}_{q,p}(\T^d))$ provided on the RHS\eqref{eq:maximal_regularity_estimate_meyers_parabolic} one also add $\|v(\cdot,0)\|_{L^p(\O;B^{1-2/p}_{q,p}(\T^d))}$.

\begin{proof}
By a perturbation argument \cite[Theorem 3.2]{AV_torus}, it is enough to prove the claim with $T=\infty$ and with $\Delta v$ replaced by $\Delta v -v$. 
The advantage is that the latter operator is invertible. This allows us to work on the half-line $\R_+=(0,\infty)$ instead of intervals of finite length. Moreover, as will be evident from the proof below, this also gives the \emph{independence} of $p_0>2$ on $T$ in the maximal $L^p_t(L^q_x)$-regularity estimate for \eqref{eq:parabolic_linear}. Hence, below, we consider
\begin{equation}
\label{eq:parabolic_linear2}
\left\{
\begin{aligned}
\partial_t v& =\kappa (\Delta  v-v)  +f+ \sum_{n\geq 1}\big(\big[(\xi_n\cdot\nabla) v \big]+g_n\big)\,\dot{W}_t^n 
&\text{ on }&\T^d,\\
v(0,\cdot)&=0 &\text{ on }&\T^d.
\end{aligned}
\right.
\end{equation}
The main idea is to regard \eqref{eq:parabolic_linear2} as a perturbation of the case $\xi_n\equiv 0$ and show the smallness of the transport noise part as $(q,p)$ approaches $p=q=2$. To this end, we first analyze the constant in the maximal $L^p_t(L^q_x)$-regularity estimate when $\xi_n\equiv 0$, i.e., for the stochastic heat equation with additive noise.

\smallskip

\emph{Step 1: (Analysis of constants -- case $\xi_n\equiv 0$ and $f\equiv 0$)
For each $p>2$ there exists $K_p>0$ satisfying 
\begin{equation}
\label{eq:asymptotic_behaviour_parabolic}
\limsup_{p\downarrow 2}K_p< 1/\sqrt{2\kappa_0}
\end{equation}
such that, for all $q\in [2,p]$, all stopping times $\tau:\O\to [0,\infty]$, and all progressively measurable process $g\in L^p((0,\tau)\times \O;L^q(\ell^2))$, there exists a unique strong solution $v$ to \eqref{eq:parabolic_linear2} with $\xi_n\equiv 0$ and $f\equiv 0$ such that $v\in L^{p}((0,\tau)\times \O;H^{1,q})$ and}
\begin{equation}
\label{eq:a_priori_estimate_1_eq1}
\| v\|_{L^{p}((0,\tau)\times \O;H^{1,q})}\leq K_p\|g\|_{L^{p}((0,\tau)\times \O;L^{q}(\ell^2))}.
\end{equation}

We emphasize that the case $\tau=\infty$ a.s.\ is allowed. 
The existence of a strong solution $v$ to \eqref{eq:parabolic_linear2} with $\xi_n\equiv 0$ for which \eqref{eq:a_priori_estimate_1_eq1} holds follows \cite[Theorem 7.1]{NVW13} and the above mentioned boundedness of the  $H^{\infty}$-calculus and invertibility of the operator $v\mapsto -\Delta v+v$ on $H^{-1,q}$ with domain $H^{1,q}$. 
For $p\geq 2$ and $q\in [2,p]$, let $N_{q,p}$ be the optimal constant constant for which \eqref{eq:a_priori_estimate_1_eq1} holds. By complex interpolation (see, e.g., \cite[Theorem 2.2.6]{Analysis1}), it follows that 
\begin{equation}
\label{eq:asymptotic_behaviour_Rstarp_step1_parabolic0}
N_{p,q}\leq N_{p,2}^{1-\eta}N_{p,p}^{\eta} \  \text{ where $\eta\in (0,1)$ satisfies }\ \tfrac{1}{q}=\tfrac{1-\eta}{2}+\tfrac{\theta}{p}.
\end{equation} 
Recall that $\kappa_0<\kappa$. To prove the claim of Step 1, it remains to show that \eqref{eq:asymptotic_behaviour_parabolic} with $K_p=\sup_{2\leq q\leq p} N_{p,q}$. To this end, as $\kappa_0<\kappa$, it is enough to show that 
\begin{equation}
\label{eq:asymptotic_behaviour_Rstarp_step1_parabolic}
\lim_{p\downarrow 2} N_{p,2}=\lim_{p\downarrow 2} N_{p,p}= 1/\sqrt{2\kappa}.
\end{equation}
We prove only that $\lim_{p\downarrow 2} N_{p,p}= 1/\sqrt{2\kappa}$ as the other is similar. 
Again, by complex interpolation \cite[Theorem 2.2.6]{Analysis1}, it follows that, for all $p_0>p$,
\begin{equation*}
N_{p,p}\leq N_{2,2}^{1-\theta}N_{p_0,p_0}^\theta 
\  \text{ where $\theta\in (0,1)$ satisfies }\ \tfrac{1}{p}=\tfrac{1-\theta}{2}+\tfrac{\theta}{p_0}.
\end{equation*}
We claim that 
\begin{equation}
\label{eq:N22_constant_parabolic}
N_{2,2}\leq 1/\sqrt{2\kappa}
\end{equation}
It is clear that \eqref{eq:N22_constant_parabolic} yields second in \eqref{eq:asymptotic_behaviour_Rstarp_step1_parabolic} as $p\to 2$ implies $\theta\to 0$. 
Next, we prove \eqref{eq:N22_constant_parabolic}. To this end, note that from It\^o's formula to $v\mapsto \|v\|_{L^2}^2$ (see e.g.\ \cite[Theorem 4.2.5]{LR15}) and an integration by parts, it follows that, a.s.\ for all $t>0$,
$$
\|v(t)\|_{L^2}^2 + 2\kappa\int_0^t\int_{\T^d}(|v|^2+ |\nabla v|^2)\,\dd x \,\dd t 
= \int_0^t \int_{\T^d} \|g\|_{\ell^2}^2,\dd x \,\dd t +\mathcal{M}_t,
$$
where $\mathcal{M}_t$ is an $L^1(\O)$-martingale starting at $0$. Thus, taking expected values in the previous identity, discarding the quantity $\E\|v(t)\|_{L^2}^2$, and using that $\E\mathcal{M}_t=0$, we obtain the (quenched) energy inequality: 
\begin{align*}
 2\kappa\,\E\int_0^t\int_{\T^d}(|v|^2+ |\nabla v|^2)\,\dd x \,\dd t 
\leq \E\int_0^t \int_{\T^d} \|g\|_{\ell^2}^2\,\dd x \,\dd t \ \ \text{ for all }t>0.
\end{align*}
Hence, \eqref{eq:N22_constant_parabolic} follows from the above by taking $t\to \infty$.

\smallskip

\emph{Step 2: Conclusion}. 
To prove the claim of Theorem \ref{t:Meyers_parabolic}, we employ the method of continuity \cite[Proposition 3.13]{AV19_QSEE_2}. 
The method of continuity is a powerful tool in the theory of (S)PDEs (see, e.g., \cite{AV_torus} and \cite[p.\ 113--115]{TayPDE3}). In this context, one considers a family of problems that vary continuously with respect to a parameter $\lambda\in [0,1]$. The method allows one to deduce the existence and uniqueness of solutions at $\lambda=1$, provided these properties are known at $\lambda=0$, and a uniform \emph{a priori} estimate for solutions holds for all $\lambda\in [0,1]$ (here, uniform means that the constants are independent of $\lambda$).
In applications, the problem with $\lambda=0$ is typically simpler. In our case, it corresponds to equation \eqref{eq:parabolic_linear2} with $\xi_n\equiv 0$, as analyzed in Step 1. The original problem typically corresponds to $\lambda=1$. The key advantage of this approach is that, for $\lambda\in (0,1]$, one can assume \emph{a priori} the existence of a solution to the $\lambda$-problem.
In the context of linear (S)PDEs, such as those considered here, the existence and uniqueness of solutions can also be obtained a posteriori by leveraging the a priori estimate through a contraction principle (see \cite[Proposition 3.13]{AV19_QSEE_2}).

Coming back to the proof, from the above discussion, for $\lambda\in[0,1]$, we consider the SPDE:
\begin{equation}
\label{eq:parabolic_linear2_lambda}
\left\{
\begin{aligned}
\partial_t v_{\lambda}& =\kappa (\Delta  v_{\lambda}-v_{\lambda} )+f + \sum_{n\geq 1}\big(\lambda\,\big[(\xi_n\cdot\nabla) v_{\lambda} \big]+g_n\big)\,\dot{W}_t^n 
&\text{ on }&\T^d,\\
v_{\lambda}(0,\cdot)&=0 &\text{ on }&\T^d.
\end{aligned}
\right.
\end{equation}
Note that, for $\lambda=1$, the above reduces to \eqref{eq:parabolic_linear2}, while for $\lambda =0$ it reduces to the one considered in Step 1.
 
We claim that there exist constants $p_0>2$ and $C_0>0$ such that, for all $p\in[2,p_0]$, $q\in [2,p]$, $\lambda\in [0,1]$, stopping time $\tau:\O\to [0,\infty]$, progressively measurable processes $f\in L^p((0,\tau)\times \O;H^{-1,q})$ and  $g=(g_n)_{n\geq 1}\in L^p((0,\tau)\times \O;L^q(\ell^2))$, and \emph{any} strong solutions $v_{\lambda} $ to \eqref{eq:parabolic_linear2_lambda} in $L^p((0,\tau)\times \O;H^{1,q})$ and $\lambda\in[0,1]$, it holds that 
\begin{equation}
\label{eq:claimed_estimate_independent_lambda}
\|v_{\lambda}\|_{L^p((0,\tau)\times \O;H^{1,q})}
\leq C_0(
\|f\|_{L^p((0,\tau)\times \O;H^{-1,q})}
+
\|g\|_{L^p((0,\tau)\times \O;L^{q}(\ell^2))}).
\end{equation}
The key point is the independence of $C_0,p_0$ on $\lambda\in [0,1]$. As discussed above, the existence of such 
$v_\lambda$ is taken for granted, based on the method of continuity, which will be rigorously applied at the end of this step.

Let $\tau$ be a fixed stopping time. To prove \eqref{eq:claimed_estimate_independent_lambda}, we begin with a reduction to the case $f\equiv 0$. Therefore, let us assume that \eqref{eq:claimed_estimate_independent_lambda} holds with $f\equiv 0$. Since the operator $v\mapsto -\Delta v +v$ has a bounded $H^{\infty}$-calculus of angle $0$, it also has deterministic maximal $L^p_t(L^q_x)$-regularity by \cite[Theorem 4.4.5]{pruss2016moving}. Arguing as in \cite[Theorem 3.9]{VP18} (or by approximation), we can find a progressively measurable process $w\in L^p(\O;W^{1,p}(0,\tau;H^{-1,q})\cap L^p(0,\tau;H^{1,q}))$ which satisfies, a.s.\ and a.a.\ $t \in (0,\tau)$,
\begin{equation}
\label{eq:remove_f_g_meyers_parabolic}
\partial_t w (t,\cdot)=\kappa(\Delta w(t,\cdot)-w(t,\cdot)) + f(t,\cdot), \qquad w(0,\cdot)=0,
\end{equation}
on $\T^d$.
Since $\wt{v}_{\lambda}\stackrel{{\rm def}}{=} v_{\lambda}-w$ solves \eqref{eq:claimed_estimate_independent_lambda} with $f\equiv 0$ and $g_n$ replaced by $g_n + (\xi_n\cdot\nabla)w$, it is enough to prove \eqref{eq:claimed_estimate_independent_lambda} with $f\equiv 0$.

To prove \eqref{eq:claimed_estimate_independent_lambda} with $f\equiv 0$, note that, for all $v\in H^{1,q}$,
\begin{align*}
\|((\xi_n\cdot\nabla)v)_{n\geq 1}\|_{L^{q}(\ell^2)}
= \Big(\int_{\T^d} \Big( \sum_{n\geq 1}| (\xi_n\cdot \nabla) v |^2\Big)^{q/2}\,\dd x \Big)^{1/q}
\leq \sqrt{2\kappa_0}\, \|\nabla v\|_{L^q} ,
\end{align*}
where the last inequality follows by applying \eqref{eq:parabolicity_linear} pointwise and with $\eta=\nabla v(x)$.

Hence, in case $f\equiv 0$, for all $\lambda\in [0,1]$, $p\in [2,\infty]$ and $q\in [2,p]$,
\begin{align*}
\| v_{\lambda}\|_{L^{p}(\R_+\times\O;	H^{1,q})}
&\leq K_p \big(\|g\|_{L^p(\R_+\times\O;L^{q}(\ell^2))}
+
\|((\xi_n\cdot\nabla)v_\lambda)_{n\geq 1}\|_{L^p(\R_+\times \O;L^{q}(\ell^2))}
\big)\\
&\leq K_p \|g\|_{L^p(\R_+\times\O;L^{q}(\ell^2))}+ \sqrt{2\kappa_0}K_p \|v_{\lambda}\|_{L^{p}(\R_+\times\O;H^{1,q})}.
\end{align*}
where $K_p$ is the constant in \eqref{eq:a_priori_estimate_1_eq1} of Step 1. From the latter, there exists $p_0(\kappa_0,\kappa)>2$ such that $\sup_{2\leq p\leq p_0}K_p<1/\sqrt{2\kappa_0}$. Thus, the previous estimate yields, for all $2\leq p\leq p_0$ and $q\in [2,p]$,
$$
\|v_{\lambda}\|_{L^{p}(\R_+\times\O;H^{1,q})}\leq (1-\sqrt{2\kappa_0}K_p)^{-1}K_p\|g\|_{L^p(\R_+\times\O;L^{q}(\ell^2))}
$$
Hence, \eqref{eq:claimed_estimate_independent_lambda} with $f\equiv 0$ holds with a constant independent of $\lambda$ for $2\leq p\leq p_0$.

Now, we are in a position to apply the method of continuity \cite[Proposition 3.13]{AV19_QSEE_2}. Note that, from Step 1 and the reduction to the case $f\equiv 0$ as performed below \eqref{eq:claimed_estimate_independent_lambda}, we have existence and uniqueness for the SPDE \eqref{eq:parabolic_linear2_lambda} with $\lambda=0$ for solutions in the class $L^p((0,\tau)\times \O;H^{1,q})$, where $\tau$ is any stopping time with values in $[0,\infty]$. Thus, the conclusion of this step follows from the method of continuity in \cite[Proposition 3.13]{AV19_QSEE_2} and the a priori estimate \eqref{eq:claimed_estimate_independent_lambda} that is uniform in $\lambda\in [0,1]$.
\end{proof}

\begin{remark}
\label{rem:time_independence_p}
Below are some comments on the proof of Theorem \ref{t:Meyers_parabolic}.
\begin{itemize}
\item The additional dissipative term $-\kappa v$ in \eqref{eq:parabolic_linear2} is not essential for the proof but ensures that the constant $K_p$ in \eqref{eq:a_priori_estimate_1_eq1} remains independent of $T$. Without it, $K_p$ would depend on $T$, leading to a time-dependent $p_0$ in the second step and, consequently, to \eqref{eq:claimed_estimate_independent_lambda} holding only for $p \in [2, p_0]$, $q \in [2, p]$. Introducing $-\kappa v$ and working on the half-line avoids this artificial dependence, as anticipated at the beginning of the proof of Theorem \ref{t:Meyers_parabolic}.

\item Theorem \ref{t:Meyers_parabolic} also implies the existence of strong solutions to \eqref{eq:parabolic_linear2}, and estimate \eqref{eq:claimed_estimate_independent_lambda} even holds for $\tau = \infty$. This is due to the term $-\kappa v$, which—as previously noted—ensures that $p_0$ remains independent of the time interval. One can check that the estimate \eqref{eq:parabolic_linear} holds if the processes $f$ and $g_n$ have mean zero, as it restores invertibility of the Laplacian on $\T^d$. 
\end{itemize}
\end{remark}

\subsection{Stochastic Meyers' estimates for the turbulent Stokes system}
In this subsection, we prove the Meyers' estimates for the turbulent Stokes system on $\T^d$, i.e.,
\begin{equation}
\label{eq:stokes_turbulent}
\left\{
\begin{aligned}
\partial_t v& =\kappa \Delta  v +\qxi v+ f  + \sum_{n\geq 1}\big(\p\big[(\xi_n\cdot\nabla) v \big]+g_n\big)\,\dot{W}_t^n 
&\text{ on }&\T^d,\\
v(0,\cdot)&=0 &\text{ on }&\T^d,
\end{aligned}
\right.
\end{equation}
here $v$ is the unknown process, $f$, $\xi_n$ and $g_n$ are specified below. Finally, 
$\qxi: \Hs^{1}(\T^d) \to (\Hs^{1}(\T^d))^{*}$ is given by 
\begin{align*}
\langle w ,\qxi v\rangle \stackrel{{\rm def}}{=}
-\frac{1}{2}\sum_{n\geq 1}\int_{\T^d}  \q[(\xi_n\cdot\nabla)v]\cdot\q[(\xi_n\cdot\nabla) w]\,\dd x
\ \text{ for }\  w\in \Hs^{1}(\T^d).
\end{align*}
Here and below, we use the notation introduced at the beginning of Section \ref{s:2D_NSE_velocity}.
Although it holds that $(\Hs^{1})^{*}=\Hs^{-1}$, here we will not employ such identification as it is important to keep track of constants. Note that $\qxi$ corresponds to the It\^o-correction for the Stratonovich formulation of the transport noise for the turbulent Stokes system, see Section \ref{s:2D_NSE_velocity} and \cite[Section 1]{AV21_NS}.

Strong solutions to \eqref{eq:stokes_turbulent} can be defined similarly to those of \eqref{eq:parabolic_linear}.
The following complements \cite[Theorem 3.2]{AV21_NS} in case of $L^{\infty}$-transport noise.

\begin{theorem}[Meyers' estimates -- Turbulent Stokes system]
\label{t:Meyers_Stokes}
Let $\kappa>0$. Assume that 
$$
\xi_n=(\xi_j^n)_{j=1}^d:\R_+\times \O\times \T^d\to \R^d \text{ is $\Progress\otimes \Borel(\T^d)$-measurable for all $n\geq 1$},
$$ 
and that there exist $M_0>0$ and $\kappa_0\in [0,\kappa)$ such that, a.e.\ on $\R_+\times \O\times\T^d$,
\begin{align}
\label{eq:boundedness_stokes}
\|(\xi_n)_{n\geq 1}\|_{\ell^2}&\leq M_0 ,& &(\text{boundedness})\\
\label{eq:parabolicity_stokes}
\frac{1}{2}\sum_{n\geq 1}\big( \xi_n\cdot \eta\big)^2  
&\leq \kappa_0 |\eta|^2 \  \text{ for all } \eta\in \R^d.& &(\text{parabolicity})
\end{align}
Then there exists $p_0(d,\kappa,\kappa_0,M_0)>2$ such that, for all stopping times $\tau$ with values $[0,T]$ with $T<\infty$, $p\in [2,p_0]$, $q\in [2,p]$ and progressively measurable processes $f$ and $g$ satisfying
$$
f\in L^p((0,\tau)\times \O;\Hs^{-1,q}(\T^d)) \quad \text{ and }\quad g=(g_n)_{n\geq 1} \in L^p((0,\tau)\times \O;\Ls^{q}(\T^d;\ell^2)),
$$
there exists a unique strong solution $v\in L^p((0,\tau)\times \O;\Hs^{1,q}(\T^d))$ to \eqref{eq:stokes_turbulent} and
\begin{align}
\label{eq:maximal_regularity_estimate_meyers_stokes}
\|v\|_{L^p((0,\tau)\times \O;H^{1,q}(\T^d;\R^d))}
\lesssim_{d,\kappa,\kappa_0,M_0,T,p}
&\,\|f\|_{L^p((0,\tau)\times \O;H^{-1,q}(\T^d;\R^d))}\\
\nonumber
+&\,\|g\|_{L^p((0,\tau)\times \O;L^{q}(\T^d;\ell^2(\N;\R^d)))}.
\end{align}
\end{theorem}

The comments below Theorem \ref{t:Meyers_parabolic} extend to the above result. In particular, if $\tau=\infty$, then the LHS\eqref{eq:maximal_regularity_estimate_meyers_stokes} can be replaced by 
\begin{align*}
&\|v\|_{L^p(\O;C([0,T];\Bs^{1-2/p}_{q,p}(\T^d)))},  \ \ \text{and}\\
&\|v\|_{L^p( \O;H^{\theta,p}(0,T;\Hs^{1-2\theta,q}(\T^d)))}  \text{ for $\theta\in [0,\tfrac{1}{2})$ provided $p>2$},
\end{align*}
and the first of the above also holds with $p=2$ (and thus $q=2$).
Moreover, due to \cite[Proposition 3.10]{AV19_QSEE_1}, the above estimates also hold if $v(\cdot,0)\in L^p_{\F_0}(\O;\Bs^{1-2/p}_{q,p}(\T^d))$ provided on the RHS\eqref{eq:maximal_regularity_estimate_meyers_stokes} one also add $\|v(\cdot,0)\|_{L^p(\O;\Bs^{1-2/p}_{q,p}(\T^d))}$.

\begin{proof}
The proof is a variant of the one used in Theorem \ref{t:Meyers_parabolic}. 
In particular, we consider the following SPDE:
\begin{equation}
\label{eq:stokes_turbulent2}
\left\{
\begin{aligned}
\partial_t v& =\kappa (\Delta  v-v) +\qxi v+ f  + \sum_{n\geq 1}\big(\p\big[(\xi_n\cdot\nabla) v \big]+g_n\big)\,\dot{W}_t^n 
&\text{ on }&\T^d,\\
v(0,\cdot)&=0 &\text{ on }&\T^d,
\end{aligned}
\right.
\end{equation}
which is introduced for technical convenience and allows us to obtain a constant $p_0>2$ that is independent of the time horizon $T<\infty$, as stated in Theorem \ref{t:Meyers_Stokes}; see also Remark \ref{rem:time_independence_p}.
Thus, in the remaining part of the proof, we prove existence and the estimate \eqref{eq:maximal_regularity_estimate_meyers_stokes} for strong solutions to \eqref{eq:stokes_turbulent2}.

Since the deterministic part of the turbulent Stokes system \eqref{eq:stokes_turbulent2} (i.e., with $\xi_n\equiv 0$) is more intricate than that of \eqref{eq:parabolic_linear}, we apply the perturbation argument--via the method of continuity from Theorem \ref{t:Meyers_parabolic}--in two stages. In the first step, we establish maximal $L^p$-regularity for the deterministic component, which then enables a reduction to the case $f \equiv 0$ in \eqref{eq:stokes_turbulent2}.

\smallskip

\emph{Step 1: (Maximal $L^p$-regularity -- deterministic problem) There exist $p_1>2$ and constants $(K_{p})_{p\in [2,p_1]}$ such that, for all $p\in [2,p_1]$, $q\in [2,p]$, $T\in (0,\infty]$, and $f\in L^p(0,T;\Hs^{-1,q})$ the following Cauchy problem
\begin{equation}
\label{eq:stokes_step1_problem_w}
\partial_t w(t,\cdot) =\kappa (\Delta  w(t,\cdot)-w(t,\cdot))+\qxi w(t,\cdot)+ f(t,\cdot) ,\quad w(0,\cdot)=0,
\end{equation}
on $\T^d$,
has a unique strong solution $w\in L^p(0,T;\Hs^{1,q})$ satisfying} 
$$
\|w\|_{L^p(0,T;H^{1,q})}\leq K_p \|f\|_{L^p(0,T;H^{-1,q})}.
$$

The proof of Step 1 follows as the proof of Theorem \ref{t:Meyers_parabolic} regarding \eqref{eq:stokes_step1_problem_w} as a perturbation of the case $\xi_n\equiv 0$. For clarity, we divide the proof into two substeps.

\smallskip

\emph{Substep 1a: (Analysis of constants -- case $\xi_n\equiv 0$)
For each $p>2$ there exists $C_p>0$ satisfying $\displaystyle{\limsup_{p\downarrow 2} C_p< 1/\kappa_0}$ such that, for all $q\in [2,p]$, $T\in (0,\infty]$ and $f\in L^p(0,T;\Hs^{-1,q})$, there exists a unique strong solution $w\in L^p(0,T;\Hs^{1,q})$ to \eqref{eq:stokes_step1_problem_w} $\xi_n\equiv 0$ (and hence $Q_\xi\equiv 0$) and}
\begin{equation}
\label{eq:a_priori_estimate_stokes_1}
\| w\|_{L^p(0,T;H^{1,q})}\leq C_p \|f\|_{L^p(0,T;H^{-1,q})}.
\end{equation}
As in Step 1 of Theorem \ref{t:Meyers_parabolic}, the existence of such $w$ follows from the boundedness $H^{\infty}$-calculus of the operator $w\mapsto -\Delta w+w$ on $(\Hs^{1,q'})^*$ with domain $\Hs^{1,q}$. By repeating the argument in \eqref{eq:asymptotic_behaviour_Rstarp_step1_parabolic0}-\eqref{eq:N22_constant_parabolic} and $\kappa_0<\kappa$, it remains to show the validity of \eqref{eq:a_priori_estimate_stokes_1} with $p=q=2$ and $C_2=1/\kappa$. To prove the latter, note that, by computing $\frac{\dd }{\dd t }\|w\|_{L^2}^2$ we obtain 
\begin{align*}
\frac{1}{2}\|w(t)\|_{L^2}^2+\kappa \int_{0}^t \| w(s)\|_{H^1}^2 \,\dd s
&= \int_0^t \langle  w(s),f(s) \rangle \,\dd s\\
&\leq \frac{1}{2\kappa}\int_0^t \|f(s)\|_{H^{-1}}^2\,\dd s +\frac{\kappa}{2}\int_0^t \|w(s)\|^2_{H^{1}}\,\dd s
\end{align*}
for all $t\in \R_+$.
The above immediately yields  \eqref{eq:a_priori_estimate_stokes_1} with $p=q=2$ and $C_2=1/\kappa$.

\smallskip

\emph{Substep 1b: Proof of the claim of Step 1.}
We prove the claim of Step 1 by (the deterministic version of) the method of continuity \cite[Proposition 3.13]{AV19_QSEE_2} (the reader is referred to the text at the beginning of Step 2 in Theorem \ref{t:Meyers_parabolic} for some comments). 
Thus, for all $\lambda\in [0,1]$ consider, on $\T^d$,
\begin{equation}
\label{eq:stokes_step1_problem_w_lambda}
\partial_t w_{\lambda}=\kappa (\Delta  w_{\lambda}-w_{\lambda}) +\lambda \qxi w_{\lambda}+ f  ,\qquad w_{\lambda}(0,\cdot)=0.
\end{equation}
As in the proof of Step 2 of Theorem \ref{t:Meyers_parabolic},  from the method of continuity and Step 1 of the current proof, it remains to prove that any strong solution $w_{\lambda}$ to the above that lies in $L^p(0,T;H^{1.q})$ for some $T\in (0,\infty]$ satisfies an a priori estimate with constant independent of $\lambda\in [0,1]$. Firstly, by \eqref{eq:boundedness_stokes}, it follows that
\begin{equation}
\label{eq:qxi_lower_order}
\|\qxi w \|_{(\Hs^{1,q'})^*}\leq D_q \|w\|_{\Hs^{1,q}} \ \ \text{ for $q\in [2,\infty)$ and $w\in H^{1,q'}$.} 
\end{equation}
Secondly, we show that we can choose $D_q$ in \eqref{eq:qxi_lower_order} such that 
\begin{equation}
\label{eq:qxi_lower_order_constant}
\lim_{q \downarrow 2} D_{q}=\kappa_0.
\end{equation}
By interpolation (again, see \eqref{eq:asymptotic_behaviour_Rstarp_step1_parabolic0}-\eqref{eq:N22_constant_parabolic}), it is enough to prove that \eqref{eq:qxi_lower_order} holds with $C_2=\sqrt{2\kappa_0}$ whenever $q=2$. Note that, for all $w\in \Hs^{1}$,
\begin{align*}
\|\qxi w\|_{(\Hs^{1})^*}
&= \frac{1}{2}
\sup_{\phi\,\in \,\Hs^{1}\,:\, \|\phi\|_{H^{1}}\leq 1}\, 
\sum_{n\geq 1}\int_{\T^d}  \q[(\xi_n\cdot\nabla)w]\cdot\q[(\xi_n\cdot\nabla) \phi]\,\dd x
\\
&\stackrel{(i)}{\leq} \frac{1}{2} \sup_{\phi\,\in \, \Hs^{1}\,:\, \|\phi\|_{H^{1}}\leq 1}\,
\| ( \q[(\xi_n\cdot\nabla)w])_{n\geq 1}\|_{L^2(\ell^2)}
\| ( \q[(\xi_n\cdot\nabla)\phi])_{n\geq 1}\|_{L^2(\ell^2)} 
\\
&\stackrel{(ii)}{\leq} \frac{1}{2} \sup_{\phi\,\in \, \Hs^{1}\,:\, \|\phi\|_{H^{1}}\leq 1}\,
\| ( (\xi_n\cdot\nabla)w)_{n\geq 1}\|_{L^2(\ell^2)}
\| ( (\xi_n\cdot\nabla)\phi)_{n\geq 1}\|_{L^2(\ell^2)} \\
&\stackrel{\eqref{eq:parabolicity_stokes}}{\leq} \kappa_0 \|\nabla w\|_{L^2},
\end{align*}
where in $(i)$ we used the Cauchy-Scwartz inequality in $L^2(\ell^2)=L^2(\T^d;\ell^2(\N;\R^d))$, and in $(ii)$ that $\|\q\|_{\calL(L^2)}=1$ as $\q$ is an ortogonal projection on $L^2$.

Now, similarly to Step 2 of Theorem \ref{t:Meyers_parabolic}, one can prove an a-priori estimate for $\|w_{\lambda}\|_{L^p(0,T;H^{1,q})}$ with constant independent of $\lambda$ and $T\in (0,\infty]$ provided $2\leq q\leq p$ is sufficiently small due to \eqref{eq:qxi_lower_order_constant}, $\kappa_0<\kappa$ and substep 1a.

\smallskip

\emph{Step 2: Conclusion}.
The proof of Theorem \ref{t:Meyers_Stokes} follows by applying again the method of continuity. Hence, for $\lambda\in [0,1]$ and a stopping time $\tau$ taking values in $[0,\infty]$, we consider
\begin{equation}
\label{eq:stokes_turbulent_final}
\left\{
\begin{aligned}
\partial_t v_{\lambda}& =\kappa (\Delta  v_{\lambda} -v_{\lambda})+\qxi v_{\lambda}+ f  + \sum_{n\geq 1}\big(\lambda\, \p\big[(\xi_n\cdot\nabla) v_{\lambda} \big]+g_n\big)\,\dot{W}_t^n ,\\
v_{\lambda}(0,\cdot)&=0,
\end{aligned}
\right.
\end{equation}
on $\T^d$,
where $f\in L^p((0,\tau)\times \O;(\Hs^{1,q'})^*)$ and $g\in L^p((0,\tau)\times \O;L^{q}(\ell^2))$ are given progressively measurable processes. By the method of continuity \cite[Proposition 3.13]{AV19_QSEE_2},  it is enough to prove the existence of $p_0>2$ and $C_0>0$ such that, for all $p\in[2,p_0]$, $q\in [2,p]$, $\lambda\in [0,1]$, and any stopping time $\tau$ and any progressively measurable processes $f$ and $g=(g_n)_{n\geq 1}$ as above,
\begin{equation}
\label{eq:claimed_estimate_independent_stokes_lambda}
\|v_{\lambda}\|_{L^p((0,\tau)\times \O;H^{1,q})}
\leq C_p\big(
\|f\|_{L^p((0,\tau)\times \O;(\Hs^{1,q'})^*)}
+
\|g\|_{L^p((0,\tau)\times \O;L^{q}(\ell^2))}\big)
\end{equation}
where $v_{\lambda}\in L^p((0,\tau)\times \O;H^{1,q})$ is a \emph{given} strong solution to \eqref{eq:stokes_turbulent_final}. We emphasize that, as for \eqref{eq:claimed_estimate_independent_lambda}, the key in the above estimate is the independence of $C_p$ on $\lambda\in [0,1]$. Moreover, due to the method of continuity, the solution $v_\lambda$ is already given a priori.

Next, we turn to the proof of \eqref{eq:claimed_estimate_independent_stokes_lambda} 
Without loss of generality, we may assume $p_0\leq p_1$ where $p_1$ is as in Step 1 of the current proof. 
As in Step 2 of Theorem \ref{t:Meyers_parabolic}, due to Step 1, it is enough to prove \eqref{eq:claimed_estimate_independent_stokes_lambda} with $f\equiv 0$. 
We now repeat the argument of Theorem \ref{t:Meyers_parabolic} by analyzing first the constant in the energy inequality, and afterwards, we argue by perturbation. 

\smallskip

\emph{Substep 2a: (Analysis of constants -- case $\lambda= 0$ and $f\equiv 0$)
For each $p\in [2,p_1]$ there exists $K_p>0$ satisfying $\displaystyle{\limsup_{p\downarrow 2}K_p= 1}$, such that for all stopping time $\tau$ with values in $[0,\infty]$, any progressively measurable $g=(g_n)_{n\geq 1}\in L^p((0,\tau)\times \O;L^q(\ell^2))$ there exists a unique strong solution $v\in L^p((0,\tau)\times \O;\Hs^{1,q})$ to \eqref{eq:stokes_turbulent_final} with $\lambda= 0$ and $f\equiv 0$. Moreover, the latter satisfies}
\begin{align}
\label{eq:a_priori_estimate_1}
 \Big\|\Big[(\kappa-\kappa_0)(|v|^2 +|\nabla v|^2)+\|(\p[(\xi_n\cdot\nabla)v]\big)_{n\geq 1}\big\|^2_{\ell^2}\Big]^{1/2}\Big\|_{L^p((0,\tau)\times \O;L^{q})}&\\
\nonumber
\leq K_p\|g\|_{L^{p}((0,\tau)\times \O;L^{q}(\ell^2))}&.
\end{align}
In the following, $\tau$ is a fixed stopping time with values in $[0,\infty]$.
The existence of such $v$ follows from Step 1 and \cite[Theorem 3.9]{VP18} (or by approximation, the process $g$ with step processes and arguing as in \eqref{eq:remove_f_g_meyers_parabolic}). In the remaining part of this step, we show \eqref{eq:a_priori_estimate_1} with the claimed limiting behaviour of $K_p$.
The intuition behind the RHS\eqref{eq:a_priori_estimate_1} is that such a quantity in the case $p=q=2$ appears naturally in the energy balance with $K_2=1$.
To see this, first note that, by combining the It\^o formula \cite[Theorem 4.2.5]{LR15} and a standard integration by parts argument, one has
\begin{equation}
\label{eq:stokes_energy_balance}
\kappa \| v\|_{L^2(\R_+\times\O;H^1)} -\|(\Q[(\xi_n\cdot\nabla)v])_{n\geq 1}\|_{L^2(\R_+\times\O;L^2(\ell^2))}\leq \|g\|_{L^2(\R_+\times\O;L^2(\ell^2))}.
\end{equation}
Since $\q+\p=\mathrm{Id}_{L^2}$ and $\q,\p$ are orthogonal projections on $L^2$, 
\begin{align*}
\|(\p[(\xi_n\cdot\nabla) v])_{n\geq 1}\|_{L^2(\ell^2)}^2
&=\|((\xi_n\cdot\nabla) v)_{n\geq 1}\|_{L^2(\ell^2)}^2
-\|(\q[(\xi_n\cdot\nabla) v])_{n\geq 1}\|_{L^2(\ell^2)}^2\\
&\stackrel{\eqref{eq:parabolicity_stokes}}{\leq} 
\kappa_0\|\nabla v\|_{L^2}^2- \|(\q[(\xi_n\cdot\nabla) v])_{n\geq 1}\|_{L^2(\ell^2)}^2.
\end{align*}
Combing the above with \eqref{eq:stokes_energy_balance}, one obtains \eqref{eq:a_priori_estimate_1} with $q=p=2$ and $K_2=1$.

We conclude by arguing by an interpolation argument similar to \eqref{eq:asymptotic_behaviour_Rstarp_step1_parabolic0}-\eqref{eq:N22_constant_parabolic}. To this end, consider the operator
\begin{align*}
\Stok_{q,p}:L^p_{\Progress}(\R_+\times \O;\Ls^q(\T^d;\ell^2))
&\to L^p(\R_+\times \O;L^q(\T^d;\R^{d}\times \R^{d\times d}\times\ell^2)))\\
g&\mapsto (\sqrt{\kappa-\kappa_0}v,\sqrt{\kappa-\kappa_0}\nabla v, (\p[(\xi_n\cdot\nabla v)])_{n\geq 1}),
\end{align*}
where $v$ is the strong solution to \eqref{eq:stokes_turbulent_final} with $\lambda= 0$ and $f\equiv 0$, which exists due to the comment below \eqref{eq:a_priori_estimate_1}. We emphasize that the norm on product space $\R^{d}\times \R^{d\times d}\times\ell^2$ is the Euclidean one: $\|(\cdot,\cdot,\cdot)\|_{\R^{d}\times \R^{d\times d}\times\ell^2}^2
=\|\cdot\|_{\R^d}^2 +
\|\cdot\|_{\R^{d\times d}}^2+\|\cdot\|_{\ell^2}^2$. 
The latter choice comes from the fact that $\|\Stok_{2,2}\|_{\calL}\leq 1$ as we have proved that \eqref{eq:a_priori_estimate_1} holds with $K_2=1$. 
The well-definiteness of $\Stok_{q,p}$ follows from the above-noticed existence and uniqueness of strong solutions $v\in L^p(\R_+\times \O;H^{1,q})$ of \eqref{eq:stokes_turbulent2} with data $g\in L^p_{\Progress}(\R_+\times \O;\Ls^q(\T^d;\ell^2))$ due to Step 1.  
Thus, the claim of Substep 2a follows by complex interpolation \cite[Theorem 2.2.6]{Analysis1}.
 
\smallskip

\emph{Substep 2b: Proof of Theorem \ref{t:Meyers_Stokes}}. 
In this substep, we prove the a priori estimate \eqref{eq:claimed_estimate_independent_stokes_lambda} for strong solutions to \eqref{eq:stokes_turbulent_final}. Now, let $v_\lambda\in L^p((0,\tau)\times \O;\Hs^{1,q})$ be a strong solution to \eqref{eq:stokes_turbulent_final} where $\tau$ is a given stopping time.  
From the estimate of Substep 2a and the elementary inequalities $(a^r+b^r)\leq (a+b)^r\leq 2^{r-1}(a^r+b^r)$ that is valid for all $r\geq 1$ and $a,b\geq 0$, we obtain the existence of a constant $N_p$ such that $\lim_{p\downarrow 2}N_p=1$ such that  
\begin{align*}
& \|v_\lambda\|_{L^p((0,\tau)\times \O;H^{1,q})}^2
+\|(\p[(\xi_n\cdot\nabla)v_\lambda])\|_{L^p((0,\tau)\times \O;L^q(\ell^2))}^2\\
&\leq N_p\Big\|\Big[(\kappa-\kappa_0)(|v_\lambda|^2 +|\nabla v_\lambda|^2)+\|(\p[(\xi_n\cdot\nabla)v_\lambda]\big)_{n\geq 1}\big\|^2_{\ell^2}\Big]^{1/2}\Big\|_{L^p((0,\tau)\times \O;L^{q})}^2\\
&\leq N_p K_p \| (g_n + \lambda\p [(\xi_n\cdot\nabla)v_\lambda])_{n\geq 1} \|_{L^{p}((0,\tau)\times \O;L^{q}(\ell^2))}^2\\
&\leq N_p K_p(1+\delta^{-1})\| g\|_{L^{p}((0,\tau)\times \O;L^{q}(\ell^2))}^2 \\
&\quad+ N_p K_p (1+\delta)
\| \p [(\xi_n\cdot\nabla)v_\lambda])_{n\geq 1} \|_{L^{p}((0,\tau)\times \O;L^{q}(\ell^2))}^2,
\end{align*}
where $\delta>0$ is a positive constant, which will be fixed later. To conclude, it remains to show that the last term can be absorbed on the LHS of the corresponding estimate. To this end, from \eqref{eq:boundedness_stokes}, it follows that, for all $p\in [2,p_1]$ and $q\in [2,p]$ (here $p_1$ is as in Substep 2a), 
\begin{align*}
&N_p K_p (1+\delta)
\| \p [(\xi_n\cdot\nabla)v_\lambda])_{n\geq 1} \|_{L^{p}((0,\tau)\times \O;L^{q}(\ell^2))}^2\\
&\leq 
\| \p [(\xi_n\cdot\nabla)v_\lambda])_{n\geq 1} \|_{L^{p}((0,\tau)\times \O;L^{q}(\ell^2))}^2
+ (1-N_pK_p (1+\delta)) C_p M_0 \|v_\lambda\|_{L^p((0,\tau)\times \O;H^{1,q})},
\end{align*}
where $C_p=\sup_{2\leq q\leq p}\|\p\|_{\calL(L^q)}$. Note that, by interpolation and the fact that $\|\p\|_{\calL(L^2)}=1$ (as $\p$ is an orthogonal projection on $L^2$), it follows that $C_p\to 1$ as $p\downarrow 2$. 
Hence, the claimed a priori estimate \eqref{eq:claimed_estimate_independent_stokes_lambda} for $v_\lambda$ follows by collecting the previous inequalities, and choosing $p\in [2,p_1]$ and $\delta>0$ such that $(1-N_pK_p (1+\delta)) C_p M_0<\kappa-\kappa_0$. Note that the latter choice is possible as $\lim_{p\downarrow 2}N_p=\limsup_{p\downarrow 2}K_p=1$, where the latter follows from Substep 2a.
\end{proof}

\subsubsection*{Acknowledgements}
The author thanks Sebastian Bechtel, Nicolas Clozeau, and Max Sauerbrey for their valuable suggestions and discussions. The author also acknowledges Franco Flandoli for his insights into the heuristic derivation presented in Subsection \ref{ss:heuristics_viscosity_dependent}. Additionally, the author thanks Umberto Pappalettera for fruitful discussions and comments, which significantly improved the presentation of the results. Finally, the author is grateful to the anonymous referee for a careful reading of the manuscript and for providing several insightful comments.

\medskip

\noindent
{\bf Data availability.} This manuscript has no associated data.
\smallskip

\noindent
{\bf Declaration – Conflict of interest.} The author has no conflict of interest. 

\def\polhk#1{\setbox0=\hbox{#1}{\ooalign{\hidewidth
  \lower1.5ex\hbox{`}\hidewidth\crcr\unhbox0}}} \def\cprime{$'$}

\end{document}